\documentclass[a4paper,12pt]{book}
\usepackage{epsfig,enumerate,amsmath,amsfonts,amssymb,amsthm, mathrsfs}
\usepackage{latexsym}
\usepackage[all]{xy}
\usepackage{graphicx}
\usepackage{amsmath,amsfonts,amssymb,amsthm,epsfig,epstopdf,array}
\usepackage{verbatim}
\theoremstyle{definition}

\newtheorem{defn}{Definition}[section]

\newtheorem{exmp}{Example}[section]
\newtheorem{thm}{Theorem}[section]
\newtheorem{lem}[thm]{Lemma}

\newtheorem*{cor}{Corollary}
\newtheorem*{res}{Result}
\theoremstyle{remark}
\newtheorem{rem}{Remark}[chapter]
\newtheorem*{note}{Note}

\begin{document}
\newpage
\thispagestyle{empty}
\pagenumbering{roman}
\titlepage

\begin{titlepage}
 \begin{center}
\LARGE
\textbf{NEVANLINNA THEORY AND NORMAL FAMILIES } 
\end{center}
\normalsize
$$\textbf{A Dissertation Submitted to  the  University of Delhi}$$
$$\textbf{in Partial Fulfilment of the Requirements}$$
$$\textbf{for the Award of the Degree of}$$
\large
$$\textbf{MASTER OF PHILOSOPHY}$$
\normalsize
$$\textbf{IN}$$
\large
$$\textbf{MATHEMATICS}$$
\large
\\
$$\textbf{By}$$
\large
$$\textbf{MANISHA SAINI}$$
\\

\vskip6em
%\vfill
\normalsize
$$\textbf{DEPARTMENT\,OF\,MATHEMATICS\,}$$
$$\textbf{UNIVERSITY\,OF\,\,DELHI}$$
$$\textbf{DELHI-110007}$$
$$\textbf{January 2015}$$
\end{titlepage}

\thispagestyle{empty}
\begin{center}{\textbf{DECLARATION}}
\end{center}

\vskip2em
\fontsize{13}{23}
\selectfont
\quad This dissertation entitled \textquotedblleft NEVANLINNA THEORY AND NORMAL FAMILIES \textquotedblright , a critical survey of the work done by various authors in this area, has been prepared by me under the supervision of Dr. Sanjay Kumar, Associate Professor, Department of Mathematics, Deen Dayal Upadhyaya College, University of Delhi for the submission
to the University of Delhi as partial fulfilment for the award of the degree of \textbf{Master of Philosophy in Mathematics}.

I, hereby also declare that, to the best of my knowledge, the work presented in this dissertation, either in part or full, has not been submitted to any other University or elsewhere for the award of any degree
 or diploma.
\vskip3em
\noindent
\rightline{\textbf{(\,Manisha Saini\,) \,}}
%\rightline{Signature of Candidate}
 %\baselineskip8pt
\vskip9em
\noindent\baselineskip16pt
\textbf{Dr. Sanjay Kumar}            \hspace{5cm} \textbf{\,Prof. Ajay Kumar }\\
Supervisor  \& Associate Professor  \hspace{3cm}              Head \\
Department of Mathematics   \hspace{3.9cm}         Department of Mathematics \\
Deen Dayal Upadhyaya College \hspace{3.3cm}        University of Delhi\\
University of Delhi    \hspace{5.6cm}                        Delhi-110007                                                                                                                           %\rightline {Delhi-110007 \;\qquad\hspace{3.9cm}}

\baselineskip24pt
\newpage
\vskip4em
\thispagestyle{empty}

\begin{center}{\textbf{\Large{ACKNOWLEDGEMENT}}}
\end{center}

\fontsize{12}{20}
\selectfont

\vskip1em
\quad It is a great pleasure for me to express intense gratitude to my M.Phil. dissertation advisor Dr. Sanjay Kumar for his keen interest in my topic of dissertation,
 guidance, encouragement and patience throughout the writing of the dissertation. Despite of his  busy schedule, he has given me enough time. I am indeed indebted to him.
 
I am delighted to take the opportunity to express my sincere thanks to Prof. Ajay Kumar, Head, Department of Mathematics, University of Delhi, for providing the necessary facilities without any delay in work. 

I would like to thank my parents for their constant support. I give my heartfelt thanks to my friends and seniors  for their timely help and the moral support they provided me during my dissertation work. I am specially thankful to Gopal Datt, Naveen Gupta and Kushal Lalwani for their help and support. I am thankful to UGC for providing me fellowship.

I would like to thank  the staff of Central Science Library for being cooperative.
Finally, I thank the office staff of Department of Mathematics for being so helpful. 

\vskip3em \fontsize{12}{20} \selectfont \noindent
\;\;\;\;\;\;\hfill{\textbf{ (\,Manisha Saini\,) \,}} \eject
\fontsize{12}{19.5}
\selectfont
\newpage
\newpage
\vskip4em

\tableofcontents
\newpage
\pagenumbering{arabic}
\thispagestyle{plain}
\chapter{Introduction}
\section{About Nevanlinna Theory}

\quad The Fundamental Theorem of Algebra says that a polynomial of degree $k$ in one complex variable will take on every complex value precisely $k$ times, counting according to its multiplicities. Picard in 1879 generalized the Fundamental Theorem of Algebra by proving that a transcendental entire function which is a sort of polynomial of infinite degree, must take on all except atmost one complex value infinitely many times. A decade after Picard\textquoteright s theorem, work of J. Hadamard in 1892 proved a strong connection between the growth order of an entire function and the distribution of the function\textquoteright s zeros. E. Borel in 1897 then proved a connection between the growth rate of the maximum modulus of an entire function and the asymptotic frequency with which it must attend all but except atmost one complex value. Finally Rolf Nevanlinna in 1925 found the right way to measure the growth of a meromorphic function and developed the theory of value distribution which also known as Nevanlinna Theory. Other main contributors to this field in the first half of the twentieth century were the leading mathematicians Lars Ahlfors, Andr$\acute{e }$ Bloch, Henri Cartan, Otto Frostman, Frithiof Nevanlinna, Tatsujiro Shimizu, Oswald Teichmuller and Georges Valiron.

In Nevanlinna theory the most important mathematical tools for a meromorphic function $f $ in $\mid z \mid \leq r$, are its proximity function
$$m(r,f)=\frac{1}{2\pi}\int_{0}^{2\pi}\log^{+}\mid f(re^{\iota\theta})\mid \,d\theta,$$
counting function $n(r,f)$ which counts the number of poles of $f$ in $\mid z \mid \leq r$, integrated counting function $$N(r,f)=\int_{0}^{r} \frac{n(r,f)-n(0,f)}{t} \, dt +n(0,f) \log r$$ and the characteristic function $$T(r,f)=m(r,f)+N(r,f).$$
\quad The first fundamental theorem and second fundamental theorem are main parts of the theory. The first main theorem gives upper bound for the counting function $N\left(r,\frac{1}{f-a}\right)$ for any $a \in \mathbb{C}_\infty$ and for large $r$. Whereas second main theorem provides a lower bound on the sum of any finite collection of counting functions $N\left(r,\frac{1}{f-a_j}\right)$ where $a_j \in \mathbb{C}$ for certain large radii $r$. Almost immediately after R. Nevanlinna proved the second main theorem, F. Nevanlinna, R. Nevanlinna\textquoteright s brother gave a geometric proof of the second main theorem. Both the \textquotedblleft geometric\textquotedblright and \textquotedblleft logarithmic derivative\textquotedblright approaches to the second main theorem  has certain advantages and disadvantages. Despite the geometric investigation and interpretation of the main terms in Nevanlinna\textquoteright s theorems the so called \textquotedblleft error term\textquotedblright  in Nevanlinna\textquoteright s second main theorem was largely ignored. Motivated by an analogy between Nevanlinna theory and diophantine approximation theory, discovered independently by C.F. Osgood,P. Vojta and S. Lang recognized that the careful study of the error term in Nevanlinna\textquoteright s main theorem would be of interest in itself. Because of the connection between Nevanlinn theory and diophantine approximation, diophantine approximation has been the motivation for so much of the current work in Nevanlinna theory. Ofcourse, it is not only \textquotedblleft classical \textquotedblright complex analysis that uses Nevanlinna theory - author such as Rippon and Stallard have used the theory in their investigations of complex dynamics. Nevanlinna\textquoteright s theory provides a new look at some old theorems, such as Picard\textquoteright s theorem which we have involved in our work. This theory also has account of applications in complex differential equation. 

The outcomes of this theory are many important generalization such as Ahlfors theory of covering surface, Hayman\textquoteright s alternative, Milloux theory etc., which we have not dissuced in this work.       
\section{A short note on normal families}

\quad The seeds of notion of normal family lie in the Bolzano-Weierstrass property: every bounded infinite set of points has a limit point. Paul Montel initiated his study of normal families in 1907 and proved the first of many sufficient conditions for a normal family. His definition of normal families of holomorphic functions may be stated as follows:

A family of holomorphic functions in a domain $\Omega$ is said to be normal if every sequence of functions $\{\,f_n, n=1,2,3,\ldots \,\}$ of the family posses a subsequence $\{\,f_{n_k}, k=1,2,3,\ldots \,\}$ which converges locally uniformly on $\Omega$ to a holomorphic function or to $\infty.$

In order to extend this definition to the case of family of meromorphic functions it is necessary to have a definition of locally uniformly convergence of a sequence of meromorphic functions in a given domain. Landau and Carath$ \acute{e}$odory first gave a such definition. Later Ostrowski pointed out that it can be equivalently defined by means of the spherical distance between two points of the extended complex plane $\mathbb{C}_\infty$. A locally uniformly convergent sequence of meromorphic function in domain $\Omega$ with respect to spherical distance has limit function which is a meromorphic function in $\Omega$ or constant $\infty$. So definition of normal families of meromorphic functions in $\Omega$ can be defined as:

A family of meromorphic function in a domain $\Omega$ is said to be normal if every sequence of meromorphic functions $\{\,f_n, n=1,2,3,\ldots \,\}$ from the family posses a subsequence $\{\, f_{n_k}, k=1,2,3,\ldots \,\}$ which is locally uniformly convergent on $\Omega$ with respect to the spherical distance.

In particular for a family of holomorphic functions in a domain this definition is equivalent to definition given for holomorphic functions.

P. Montel did pioneering work in this field. One of his most important theorem named after Montel states that a family of meromorphic functions on a domain is normal if it omits three distinct points of the complex plane, infact it has been coined fundamental normality test by Schiff ~\cite{sch}. The famous E. Picard\textquoteright s theorem was proved by using elliptic modular function theory, at that time. Its proof using Montel\textquoteright s theorem is given in Shiff ~\cite{sch}. Montel subsequently discovered the important connection between normal families and celebrated theorem of Schottky, Landau, Vitali - Porter, Lind$\acute{e}$l$\ddot{o}$f and Julia. Montel\textquoteright s final words on the subject was written in 1946. After Montel many normality criterion have been developed such as Marty\textquoteright s theorem which states that a family of meromorphic functions is normal if and only if its spherical derivative is locally bounded. In 1975, Zalcman gave a result of far reaching consequences of non - normality of family of meromorphic functions. Robinson - Zalcman Heuristic principle is another useful normality criterion. A very common application of normal families is in the usual proof of the Riemann mapping theorem, whose original proof was based upon the Dirichlet principle. A significant application of the theory of normal families was made by Fatou and Julia in their deep study of complex dynamical systems. A more recent application has been in the theory of normal functions, developed in order to study the boundary behaviour of meromorphic functions.     

Our present work explores interplay between normal families of holomorphic, of meromorphic functions and Nevanlinna Theory, many normality criterion have been proved using Nevanlinna Theory. Most important thing here to be pointed out is that we give several results which justifies Bloch\textquoteright s principle which says that if a family of meromorphic functions on a domain $\Omega$ satisfies a property which reduces a meromorphic function to a constant on complex plane $\mathbb{C}$ then that family is normal on that domain $\Omega$.  
\section{Summary}
\quad This dissertation is divided into six chapters followed by bibliography at the end. A summary of these chapters is given as follows:

\textbf{Chapter 1} entitled \textquotedblleft Introduction\textquotedblright give introduction to Nevanlinna theory and normal families.
\\

 \textbf{Section 1.1} discusses the origin and development of Nevanlinna theory where as in \textbf{Section 1.2} we have presented for normal families.
\\

In \textbf{Section 1.3} chapterwise organization of the dissertation is presented.
\\

\textbf{Chapter 2} entitled \textquotedblleft Nevanlinna Theory \textquotedblright presents the basic tools of Nevanlinna theory. We begin Nevanlinna theory from its starting point \textquotedblleft Poisson-Jensen formula \textquotedblright.
\\
 
In \textbf{Section 2.1} some basic definitions are given which are very well known to us. A proof of Poisson-Jensen formula is presented in this section.
\\

In \textbf{Section 2.2} we define basic tools of Nevanlinna theory such as positive part of $\log$ function $\log^{+}$, proximity function $m(r,f)$, counting function $N(r,f)$ and characteristic function $T(r,f)$ for a meromorphic function defined in $\mid z \mid \leq r$. Properties of these functions related to each other also proved here. Some examples to understand these functions are given
\\

In \textbf{Section 2.3} we define order of growth of a meromorphic function and measure order for some functions. Out of two main theorems of Nevanlinna theory, first fundamental theorem is proved in this section. A proof of Nevanlinna\textquoteright s Inequality is also given.
\\

In \textbf{Section 2.4} by proving a result by H. Cartan, namely Cartan\textquoteright s Identity it has been shown that characteristic function $T(r,f)$ is a increasing convex function of $r$. 
\\

In \textbf{Section 2.5} second fundamental theorem of Nevanlinna theory is stated without proof. An alternative proof of Picard\textquoteright s theorem using Nevanlinna theory is given. Finally Nevanlinna\textquoteright s Estimate and Hiong\textquoteright s Estimate are stated for using them in later work.       
\\

\textbf{Chapter 3} entitled \textquotedblleft Normal Families \textquotedblright gives basic information about normal families. First important definitions, required in defining normal families are developed.
\\

\textbf{Section 3.1} starts with stereographic projection followed by chordal metric and spherical metric. 
\\

\textbf{Section 3.2} includes different types of convergence a sequence of function. Weierstrass theorem, Hurwitz theorem etc. useful results of normal families are proved in this section. 
\\

In \textbf{Section 3.2} we define normal family of analytic (meromorphic) functions. Various results and normality criterion for these families are given.
\\

\textbf{Chapter 4} entitled \textquotedblleft Basic Normality Criterion \textquotedblright contains many normality criterion which has been proved using Nevanlinna theory. An important result of normal families namely, Montel\textquoteright s theorem is proved here. 
\\

\textbf{Section 4.1} starts with a proof of Marty\textquoteright s theorem which give a necessary and sufficient condition for a family $\mathcal{F}$ of meromorphic functions to be normal. We have proved Montel\textquoteright s theorem using Nevanlinna theory. Here we define a composition $f_\alpha$ of $f \in \mathcal{F}$ with $\phi_\alpha$ where $$\phi_\alpha(z)= \frac{z+\alpha}{1+\bar{\alpha}z}, \mid \alpha \mid <1. $$ This composition is used frequently in this dissertation. A generalization of Montel\textquoteright s theorem is also proved.
\\

In \textbf{Section 4.2} a new composition of $f \in \mathcal{F}$ with $\psi_\alpha$ where $$\psi_\alpha(z)= r^2 \frac{z-\alpha}{r^2-\bar{\alpha}z}, \mid \alpha <1$$ which has been very useful in proving many inequality for not normal families of holomorphic functions. These inequalities are not true for family of meromorphic functions. 
\\

\textbf{Chapter 5} entitled \textquotedblleft Non-vanishing Families of Holomorphic Functions \textquotedblright discusses the generalization of Montel-Miranda theorem due to Chuang ~\cite{sch} in 1940. Firstly from a non-vanishing family of holomorphic functions a new family is constructed and relation of normality of these two families is developed. Then generalization of Montel-Miranda theorem is proved. In next section we have tried to drop the condition of non-vanishing families. 
\\  

\textbf{Chapter 6} entitled \textquotedblleft A Result due to Hayman \textquotedblright includes an idea for normality of a family of holomorphic function given by Hayman ~\cite{haym}. Here also we construct a new family from a family of holomorphic function and follow Hayman to show dependence of normality of the initial family on normality of new family. 

\chapter{Nevanlinna Theory}

\quad In this chapter basic tools of Nevanlinna theory will be developed. Here we shall discuss fundamental properties of these notions and two main theorems of Nevanlinna theory in brief.
\section{Poisson-Jensen Formula}

\quad Poisson-Jensen formula is starting point of the Nevanlinna theory. We will first prove Poission-Jensen formula and then describe Jensen\textquoteright s formula which will be used to define Nevanlinna functionals. Here our function of interest are meromorphic functions. Before moving to Poisson-Jensen formula let us recall some definitions and facts.
\begin{defn} 
\textbf{Meromorphic Function}:

A meromorphic function is an analytic function at each point of its domain except on isolated points which are its poles. 
\end{defn} 
  Meromorphic functions can be consider as the quotient of two analytic functions which has poles at zeros of denominator and zeros at zeros of numerator. For example, $\tan z$ is meromorphic function has zeros at zeros of $\sin z$ that is at $z=n \pi, n \in \mathbb{Z}$ and poles at zeros of $\cos z$ that is at $z=(2n+1) \frac{\pi}{2} $, where $ n \in  \mathbb{Z}$.
\\
\begin{note}
Let $f$ and $g$ be two meromorphic functions having pole (or zero) at same point of order $j$ and $k$ then $f.g$ has pole (or zero) at the same point of order $j+k$. If $f$ has pole of order $j$ and $g$ has zero of order $k$ at $z_0$ then 
\\
$f(z)=\frac{h(z)}{(z-z_0)^k}$ and $g(z)= (z-z_0)^{k} l(z)$ where $h$ and $l$ are analytic functions at $z_0$ and $l(z_0) \neq 0$.   
\\
This gives $f.g$ has zero at $z_0$ of order $\max(0,k-j)$ or pole of order $\max(0,j-k)$.
\\
Also if $j=k$ then $f.g$ neither has zero nor pole at $z_0$.
\end{note}
\begin{defn}
For given $z_0 \in \mathbb{C}$ and $f \not \equiv 0$ analytic on $z_0$, order or multiplicity of zero of $f$ at $z=z_0$ is defined as least positive integer $n$ such that in Taylor expansion of $f$ about $z=z_0$  the coefficient of $(z-z_0)^n$ is non-zero.

In a similar way, order of a pole at $z_0$ of $f$ is defined as the multiplicity of the zero at $z_0$ of $\frac{1}{f}$. Pole with multiplicity one is called simple pole.    
\end{defn}
For example $\sec z$ has simple pole at $(2n+1) \frac{\pi}{2}, n \in \mathbb{Z}$ and has no zero. $\sin^{2} z$ has zero of order 2 at $z=n \pi$ and no pole.  

It is clear from Taylor expansion of $f$ that if $f$ has zero at $z_0$ of order $m$ then $f\textquoteright$ has zero at $z_0$ of order $m-1$ and if $f$ has pole of order $p$ at $z_0$ then $f$ has pole of order $p+1$ at $z_0$. 
\begin{defn}
\textbf{Green\textquoteright s Function}

Let $G$ be a region in $\mathbb{C}$ and $y \in G$, Green\textquoteright s function of $G$ with singularity at $y$ is a function $$ g(-,y):G \rightarrow \mathbb{R}$$ with the properties:
\\
\item[(i)] $g(x,y)$ is a harmonic function of $x$ within G, except at $x=y$ where it is positively infinite so that the function
$$g(x,y)+\log \mid x-y \mid$$
is  also    harmonic for $x=y$.
\\
\item[(ii)] $g(x,y)$ vanishes  at  every  boundary  point  $x=\zeta$  of  the  region  G. 
\end{defn} 
\begin{defn}
\textbf{Poisson Integral formula}

Let $u$ be harmonic in the open disc $D(R)$, $R < \infty$ and continuous on the closed disc $\overline{D(R)}$. Let $z$ be a  point  inside  the  disc  then
$$u(z)=\frac{1}{2\pi} \int_{0}^{2 \pi}u(Re^{\iota \theta})P(z,Re^{\iota\theta}) \,d{\theta},$$
where $$P(z,\zeta)=\frac{ \mid \zeta \mid^{2}- \mid z \mid^{2}}{\mid \zeta-z \mid^{2}}, \mid z \mid< \mid \zeta \mid$$
is called Poisson kernal.
\end{defn}
\begin{thm} 
\textbf{Poisson-Jensen formula}

Suppose that $f(z)$ is meromorphic in $\mid {z} \mid \leq R$ $(0 < R < \infty)$ and that $a_ \mu$; $ \mu = 1$  to $ M$ are zeros and $b_ \nu$;  $ \nu = 1$  to $ N$ are poles of $f(z)$ in $ \mid {z} \mid < R$. If $z=re^{ \iota \phi}$; $0 \leq r<R$ and $f(z) \neq 0, \infty$ then we have ,
\begin{align*}
\log{ \mid f(z) \mid} &= \frac{1}{2 \pi} \int_{0}^{2 \pi} \log \mid f(Re^{ \iota \theta}) \mid \frac{R^2-r^2}{R^2-2Rr \cos({ \theta- \phi})+r^2} \, d{ \phi} \\
&+ \sum_{ \mu=1}^{M} \log \big | \frac{R(z-a_ \mu)}{R^2- \bar{a_ \mu}z} \big |  - \sum_{ \nu =1}^{N} \log \big | \frac{R(z-b_ \nu)}{R^2- \bar{b_\nu}z} \big |
\end{align*}
\end{thm}
\begin{proof}
Since $f$ is meromorphic in $ \mid z \mid \leq R$ for the same values of $z$, $ \log \mid f \mid$ has positive and negative logarithmic poles respectively so that the difference,
$$ \log \mid f \mid-k_ \nu \log \mid \frac{1}{z-b_\nu} \mid $$ and $$ \log \mid f \mid-h_ \mu \log \mid z-a_ \mu \mid$$
are harmonic for $z=a_ \mu$ respectively where $k_ \nu$ represents the multiplicity of the pole $b_ \nu$ and $h_ \mu$ that of the zero $a_ \mu$.

Suppose that the number $ \rho, 0< \rho<R $ is determined so that the function $f$ is non-zero and has no pole on the circle $ \mid z \mid= \rho$. 
\\
Let $$g(z,w)= \log \mid \frac{ \rho^2- \bar{w}z}{ \rho(z-w)} \mid$$
be the Green\textquoteright s function for the disk $ \mid z \mid \leq \rho$.
\\
At each pole $b_ \nu$ and at each zero $a_ \mu$ in the  disk $ \mid z \mid< \rho$ the expressions
\begin{equation}\label{1}
 g(z,b_ \nu)- \log \big | \frac{1}{z-b_ \nu} \big |
\end{equation} and \begin{equation} \label{2}
g(z,a_ \mu)+ \log \mid z-a_ \mu \mid 
\end{equation} 
\\
are harmonic.
 \\
 Now the expression
 \begin{equation}\label{3}
 \log \mid f(z) \mid- \sum_{ \mid b_ \nu \mid < \rho} g(z,b_ \nu)+ \sum_{ \mid a_ \mu \mid < \rho}g(z,a_ \mu)
\end{equation}
\\ 
from (\ref{1}) and (\ref{2}) consequently defines a function that is harmonic for $ \mid z \mid \leq \rho,$ where the sums are over all the poles and zeros in $ \mid z \mid < \rho$ taking into account the corresponding multiplicities.

Applying the Poisson  Integral formula to (\ref{3}) we get for $z=re^{ \iota \phi}$, $r< \rho$ 
\begin{align*}
& \log \mid f(z) \mid - \sum_{ \mid b_ \nu \mid < \rho} g(z, b_\nu) + \sum_{\mid a_ \mu \mid < \rho} g(z, a_\mu) \\ 
&= \frac{1}{2 \pi} \int_{0}^{2 \pi} \left[ \log \mid f( \rho e^{\iota \theta})- \sum_{\mid b_ \nu \mid < \rho} g(\rho e^{\iota \phi}, b_ \nu)  + \sum_{\mid a_ \mu \mid < \rho} g(\rho e^{\iota \phi}, a_\mu) \right] \frac{\rho^2 - r^2}{\mid \rho e ^{\iota \theta}- re^{\iota \phi} \mid ^2} \, d {\theta}.
\end{align*}
Since Green\textquoteright s function vanishes for $ \mid z \mid = \rho$ we have
\begin{align}
\log \mid f(z) \mid  &= \frac{1}{2 \pi} \int_{0}^{2 \pi} \log \mid f( \rho e^{ \iota \theta}) \mid \frac{ \rho^2-r^2}{ \mid \rho e^{ \iota \theta} \mid^2+ \mid re^{ \iota \phi} \mid^2-2 \rho r \cos( \theta-\phi)} \, d{ \theta} \notag \\
&+ \sum_{ \mid b_ \nu \mid < \rho}g(z,b_ \nu)- \sum_{ \mid a_ \mu \mid < \rho}g(z,a_ \mu),\label{4}
\end{align} where $ z=re^{ \iota \phi}; \quad r< \rho. $
\\
Since
\begin{align*}
\mid \rho e^{\iota \theta}-re^{ \iota \phi} \mid^2 &=\left( \rho e^{ \iota \theta}-re^{ \iota \phi} \right) \overline{ \left( \rho e^{ \iota \theta}-re^{ \iota \phi} \right)} \\
&=\left( \rho e^{ \iota \theta}-re^{ \iota \phi} \right) \left( \rho e^{- \iota \theta}o-re^{- \iota \phi} \right) \\
& =\rho^2- \rho r \left( \cos( \theta- \phi)+ \iota \sin( \theta- \phi)+ \cos( \theta- \phi) - \iota \sin( \theta- \phi) \right) \\
&+r^2 \\
&=\rho^2+r^2-2 \rho r \cos( \theta- \phi).
\end{align*}

This formula was deduced under the assumption that no poles $b_ \nu$, nor zeros $a_ \mu$ are to be found on the circle $ \mid z \mid= \rho$. It is clear that all the terms in (\ref{4}) are continuous even for the excluded values of $ \rho$, from which one concludes that the equality holds for all $ \rho<R$.
\\
Thus we get
\begin{align}
\log\mid f(z)\mid &= \frac{1}{2 \pi}\int_{0}^{2 \pi} \log\mid f(Re^{ \iota \theta})\mid\frac{R^2-r^2}{R^2+r^2-2rR \cos( \theta- \phi)} \, d{ \theta} \\
&+ \sum_{\nu=1}^{M}\log \big |\frac{R^2- \bar{b_\nu}z}{R(z-b_\nu)} \big |  -\sum_{\mu=1}^{N}\log\big | \frac{R^2-\bar{a_\mu}z}{R(z-a_\mu)}\big |, \label{5}
\end{align}
for $z=re^{ \iota \phi}$ and $f(z) \neq 0, \infty$, where $\nu=1$ to $M$ and $\mu=1$ to $N$ are respectively number of poles and zeros of $f$ in $\mid z\mid <\rho$.

In particular, if one sets $z=0$ then (\ref{5}) becomes
\\
$$\log\mid f(0)\mid=\frac{1}{2 \pi} \int_{0}^{2\pi}\log\mid f(Re^{\iota\theta})\mid \, d{\theta}+\sum_{\nu=1}^{M}\log\frac{R}{\mid b_\nu\mid}-\sum_{\mu=1}^{N}\log\frac{R}{\mid a_\mu\mid}$$
\\
this expression is known as \textbf{Jensen\textquoteright s formula}.
\end{proof} %\hspace*{ \fill} \\
Charles Picard in 1876 gave a very important result about meromorphic functions called \textbf{Picard\textquoteright s theorem} which states that a meromorphic function $f$ can omit atmost two values.
\\
Picard proved this using modular function theory but we will later see that this is just a consequence of Nevanlinna Theory.
\section{Characteristic Function}
\quad This section is completely devoted to the development of basic tools of Nevanlinna theory. Following Nevanlinna we proceed to rewrite Jensen\textquoteright s formula. We define
\begin{defn}  
\begin{equation}
  \log^{+}x=
 \begin{cases}
   \log x,   &\text{if $x \geq 1 $}\\ 
   0         &\text{if $0 \leq x < 1$}
   \end{cases}
   \notag 
\end{equation}
Or simply $\log^{+}x= \max \{\,\log x,0 \, \}$.
\end{defn}
We now develop some properties of $ \log^{+}x$.
\begin{thm} \label{ch2,sec2,t1}
\textbf{Properties}
\\
Let $ \alpha, \beta, \alpha_1, \alpha_2, . . . ,\alpha_p \geq 0$ be any real numbers then ,
\\
\item[(i)] $\log^{+}\alpha \geq \log \alpha$
\\
\item[(ii)]$\log^{+}\alpha \leq \log^{+} \beta$, if  $\alpha \leq \beta$
\\
\item[(iii)]$\log \alpha = \log^{+} \alpha- \log^{+} {1/ \alpha}$, for  any $\alpha > 0$
\\ 
\item[(iv)]$\log^{+} \left(\prod_{i=1}^{p} \alpha_{i} \right) \leq \sum_{i=1}^{p} \log^{+} \alpha_i$
\\
\item[(v)]$\log^{+}\left( \sum_{i=1}^{p} \alpha_i \right) \leq \log p+ \sum_{i=1}^{p} \log^{+}\alpha_i $
\end{thm} %\hspace*{\fill} \\
\begin{proof}
\item[(i)] As $ \log^{+}$ is positive part of $ \log$ therefore (i) is true. 
\\ 
\item[(ii)] Since $\log$ is an increasing function so is $\log^{+}$.
\\
\item[(iii)]As if $ \alpha \geq 1$ then $\log^{+}\left(1/ \alpha \right)=0$ as ${1/ \alpha} \leq 1 $
therefore, $\log \alpha= \log^{+}{ \alpha}-\log^{+}\left( 1/ \alpha \right).$
\\
If $ \alpha < 1$ then
$ \log^{+}{ \alpha}=0 $ and $ \log^{+} \left( 1/\alpha \right) \neq 0$ as $ 1/ \alpha > 1$, $ \log^{+}\left( 1/\alpha \right)>0$ but $ \log \alpha < 0$ 
thus $$ \log \alpha = -\log^{+} \frac{1}{\alpha}  =   \log^{+} \alpha - \log^{+} \frac{1}{ \alpha}.$$
\\ 
\item[(iv)] For  $$ \prod_{i=1}^{p} \alpha_i < 1$$ we have
$$\log^{+} \prod_{i=1}^{p} \alpha_i = 0$$ 
\\ 
and $ \prod_{i=1}^{p} \alpha_i < 1\quad \mbox{gives } \alpha_k <1 $ for atleast one $1 \leq k \leq p $ and $ \alpha_j \geq 1, j \neq k$ which implies $\log^{+} \alpha_j \geq  0 $ for all $j \neq k $ therefore $$\sum_{i=1}^{p} \log^{+} \alpha_i \geq 0= \log^{+}\left( \prod_{i=1}^{p} \alpha_i \right).$$
\\
For $$\prod_{i=1}^{p} \alpha_i \geq 1$$ 
$$\log^{+} \left( \prod_{i=1}^{p} \alpha_i \right) = \log \left( \prod_{i=1}^{p} \alpha_i\right) =\sum_{i=1}^{p} \log \alpha_i \leq \sum_{i=1}^{p} \log^{+} \alpha_i\quad \mbox{( by (i))}$$
\\
\item[(v)]Since $$\sum_{i=1}^{p} \alpha_i \leq p \max_{1 \leq i \leq p} \alpha_i$$ so
\begin{align*}
 \log^{+} \sum_{i=1}^{p} \alpha_i &\leq \log^{+} \left \{ p \max_{1 \leq i \leq p} \alpha_i \right \} \quad \mbox{by } (ii) \\
&\leq \log^{+}p+ \log^{+} \max_{1 \leq i \leq p} \alpha_i \quad \mbox{by }(iv) \\
&=\log p + \log^{+} \alpha_j, \quad \mbox
{where}  \max_{1 \leq i \leq p}{ \alpha_i}= \alpha_j \quad \mbox{for some } 1 \leq j \leq p \\
& \leq \log p+ \sum_{i=1}^{p} \log^{+} \alpha_i.
\end{align*}
\end{proof} 
\begin{defn} 
\textbf{Proximity function}
\\
For any meromorphic function $f$ in $\mid z \mid \leq R$
$$ m(r,f)=\frac{1}{2 \pi} \int_{0}^{2 \pi} \log^{+} \mid f(re^{\iota \theta}) \mid \, d {\theta},$$ where $r<R$.
% \hspace*{\fill} \\
m$(r,f)$ gives the logarithmic average value of $f$ over the circle $\mid z \mid = r$.
\end{defn}% \hspace*{\fill} \\
\begin{defn}
\textbf{Counting function}

For meromorphic function $f$ in $\mid z \mid \leq R$, $n(R,f)=n(R,\infty)$ denote the number of poles of $f$ in $\mid z \mid \leq R$ counting according to multiplicities.
\\
$n\left(R,\frac{1}{f} \right)=n(R,0)$ denote the number of zeros of $f$ in $\mid z \mid \leq R$ counting according to multiplicities.
\\
$n\left(R,\frac{1}{f-a}\right)= n(R,a)$ denote these number of $a$-points in $\mid z \mid \leq R$ counting according to multiplicities.
\end{defn}
\begin{rem}\label{rem1}
Now let us consider $\sum_{\mu =1}^{N} \log \frac{R}{\mid a_{\mu}\mid}$ from the Jensen\textquoteright s formula. We have
\begin{align*}
\sum_{\mu =1}^{N} \log \frac{R}{\mid a_{\mu}\mid}&=\sum_{\mu=1}^{N} \int_{\mid a_{\mu}\mid}^{R} \frac{1}{t} \, dt=\sum_{\mu=1}^{N} \int_{0}^{R} \frac{\chi_{\mu}(t)}{t} \, dt \\
& = \int_{0}^{R} \frac{1}{t} \sum_{k=1}^{n} \chi_{k}(t) \, dt = \int_{0}^{R} \frac{n(t,0)}{t} \, dt \end{align*}  
where $\chi_{k}(r)$ is characteristic function of the disc $r \leq \mid a_k \mid$.
\\
Similarly, $$\sum_{\nu=1}^{M} \log \frac{R}{\mid b_{\nu} \mid} = \int_{0}^{R} \frac{n(t,\frac{1}{f})}{t} \, dt.$$%\hspace*{\fill} \\
\end{rem}
\begin{defn}
\textbf{Integrated Counting function}

For $f$ meromorphic in $\mid z \mid \leq R$
$$N(r,f)= \int_{0}^{r}\frac{n(t,f)-n(0,f)}{t} \, dt+n(0,f) \log r,\quad \mbox{where } r<R$$
denotes the integrated counting function.
Similarly 
$$N \left( r,\frac{1}{f}\right)=\int_{0}^{r} \frac{n\left(t,\frac{1}{f}\right) -n \left(0,\frac{1}{f}\right)}{t} \, dt+n\left(0,\frac{1}{f}\right) \log r. $$
\end{defn}
\begin{defn} 
For meromorphic function $f$ in $\mid z \mid \leq R$ and $a \in \mathbb{C}$
$$\bar{n}(r,f)= n(r,f)-n(r,f\textquoteright)$$ denotes distinct number of poles in $\mid z \mid \leq r$.
Similarly $\bar{n}(r,a)$ denotes the number of distinct $a$-points in $\mid z \mid \leq r$ and
 $$ \bar{N}(r,a)= \int_{0}^{r} \frac{ \bar{n}(t,a)- \bar{n}(0,a)}{t} \, dt+ \bar{n}(0,a) \log r .$$  
\end{defn}%\hspace*{\fill} \\
\begin{rem} \label{rem2}
Now in Jensen\textquoteright s formula we have 
\begin{align*}
\log \mid f(0) \mid &= \frac{1}{2\pi} \int_{0}^{2\pi} \log^{+} \mid f(Re^{\iota\theta}) \mid \, d{\theta}-\frac{1}{2\pi} \int_{0}^{2\pi} \log^{+}\big | \frac{1}{f(Re^{\iota\theta})} \big | \, d{\theta}  \\
&+\sum_{\nu=1}^{M} \log \frac{R}{\mid b_\nu \mid} -\sum_{\mu=1}^{N} \log \frac{R}{\mid a_\mu \mid}
\end{align*}
Since  we have $f(0)\neq 0,\infty$, this gives $n \left(0, \frac{1}{f}\right)=0$ and $n(0,f)=0$ that is number of zeros and poles at the origin is zero.
\\
Thus we get 
$$N(R,f)=N(R,\infty)=\int_{0}^{R}\frac{n(t,f)}{t} \, dt=\sum_{\nu=1}^{M} \log \frac{R}{\mid b_\nu \mid}$$
and
$$N\left(R,\frac{1}{f}\right)=N(R,0)=\int_{0}^{R} \frac{n(t,0)}{t} \,dt=\sum_{\mu=1}^{N} \log \frac{R}{\mid a_\mu \mid}.$$
Thus from definition of proximity function, remark (\ref{rem1}) and remark (\ref{rem2}); Jensen\textquoteright s formula becomes
$$\log \mid f(0)\mid= m(R,f)-m\left(R,\frac{1}{f}\right)+N(R,f)-N\left(R,\frac{1}{f}\right).$$
The important thing here to notice is that left hand side is constant even when $R \rightarrow \infty$. This allows to measure the growth of $f$ effectively.
\end{rem}
\begin{defn}
\textbf{Characteristic function}

The characteristic function of a meromorphic function $f$ is $\mid z \mid 
\leq R$ is defined as  $$T(r,f)= m(r,f)+N(r,f),\quad \mbox{for} r<R.$$ 
\end{defn}
\begin{rem} \label{rem3}
Jensen\textquoteright s formula thus turns out to be
\[ \log\mid f(0) \mid= T(R,f)-T\left(R,\frac{1}{f}\right) \]
thus we get
$$T(R,f)=T\left(R,\frac{1}{f}\right)+ \log \mid f(0) \mid.$$
\end{rem}
\begin{rem} \label{rem4}
From properties of $\log^{+}$ we observe that by applying the inequalities to meromorphic functions $f_1, f_2,...,f_p$ we get
$$m\left(r,\sum_{i=1}^{p}f_i \right)\leq \sum_{i=1}^{p}m(r,f_i)+\log p$$ 
and
$$m\left(r,\prod_{i=1}^{p}f_i \right)\leq \sum_{i=1}^{p}m(r,f_i)$$ 
these inequalities are easy to observe.
\end{rem}

Next, let $f$ and $g$ be two meromorphic functions. Suppose without loss of generality that $f$ has pole at origin. If $g$ also has pole at origin then $n(0,f.g)=n(0,f)+n(0,g)$. If $g$ has neither pole nor zero at origin then $n(0,f.g)=n(0,f)+n(0,g)\quad [n(0,g)=0]$. If $g$ has zero at origin then either $f.g$ has no pole at origin or pole of lower multiplicity than $f$ at origin. Thus 
$$n(0,fg) \leq n(0,f)+n(0,g).$$
	  
Thus for any given pole of $f.g$ at any point with multiplicity $n$, the sum of the multiplicities of the poles of $f$ and $g$ at that point are atleast $n$. Thus $$ n(r,fg) \leq n(r,f)+n(r,g)$$
which gives  $$ N(r,f.g) \leq N(r,f) +N(r,g).$$

On a similar way if $f+g$ has pole at origin (say) then atleast one of $f$ or $g$ has a pole at origin and $$n(0,f+g) = \max \lbrace n(0,f),n(0,g) \rbrace \leq n(0,f)+n(0,g).$$
\\
This is true for any pole of $f+g$ thus $$ n(r,f+g) \leq n(r,f) + n(r,g) $$ 
\\
thus $$N(r,f+g) \leq N(r,f)+N(r,g). $$
Now, if $f(z)$ is the sum or product of the functions $f_i$, then the order of a pole of $f(z)$ at a point $z_0$ is at most equal to the sum of the orders of the poles of the $f_i$ at $z_0$. Thus we get
$$N\left(r,\sum_{i=1}^{p}f_i \right) \leq \sum_{i=1}^{p} N(r,f_i)$$
and
$$N \left(r,\prod_{i=1}^{p} f_i \right) \leq \sum_{i=1}^{p} N(r,f_i).$$
Also since $T(r,f)=m(r,f)+N(r,f)$ we get
$$ T \left(r,\sum_{i=1}^{p} f_i \right) \leq \sum_{i=1}^{p}T(r,f_i) + \log p $$
and
$$ T \left(r,\prod_{i=1}^{p}f_i \right) \leq \sum_{i=1}^{p} T(r,f_i). $$
In particular, taking $p=2, f_1(z) =f(z)$, $f_2(z)=-a$ =constant, we get
$$T(r,f-a) \leq T(r,f)+T(r,-a)+ \log 2 = T(r,f)+ \log^{+} \mid a \mid + \log 2$$
and
$$T(r,f)=T(r,f-a+a) \leq T(r,f-a)+ \log^{+} \mid a \mid + \log 2.$$
We get finally
$$\mid T(r,f)-T(r,f-a) \mid \leq \log^{+} \mid a \mid + \log 2. $$

Further we give some examples for a better understanding of proximity function, counting function and characteristic function.
\begin{defn}

For $f$ and $g$ meromorphic functions in $\mid z \mid \leq R$ we write
$$f=g+O(r)$$ 
if $$\lim_{r\rightarrow \infty} \frac{f-g}{r}< \infty $$
and
$$f=g+o(r)$$
if $$\lim_{r\rightarrow \infty} \frac{f-g}{r} =0.$$
\end{defn}
\begin{exmp}
$\sin r=o(r)$. In particular, $f(z)=O(1)$ means that $f$ is bounded. 
\end{exmp}
\begin{exmp}
$f(z) = e^{z} ; z \in \mathbb{C}$ then $N(R,f)=0$ as $e^{z}$ never attain $\infty$.
\begin{align*}
m(r,f) &= \frac{1}{2\pi} \int_{- \frac{\pi}{2}}^{\frac{\pi}{2}} \log^{+}\mid e^{re^{\iota \theta}} \mid \, d{\theta} =\frac{1}{2\pi} \int_{-\frac{\pi}{2}}^{\frac{\pi}{2}} \log^{+}e^{\mid re^{\iota \theta}\mid} \, d{\theta} =  \frac{1}{2 \pi} \int_{- \frac{\pi}{2}}^{\frac{\pi}{2}} \log^{+} e^{r \cos\theta} \, d{\theta} \\
&=\frac{1}{2\pi}\int_{- \frac{\pi}{2}}^{\frac{\pi}{2}} r \cos \theta \, d \theta = \frac{1}{2\pi} r \sin \theta \mid_{- \frac{\pi}{2}}^{\frac{\pi}{2}} = \frac{r}{\pi}
\end{align*}
This gives $$T(r,f)=m(r,f)+N(r,f)= \frac{r}{\pi}.$$
Now
$N \left(r,\frac{1}{f}\right)=0$, as $f=e^z$ never attain $\infty$ 
\begin{align*}
m\left(r,\frac{1}{f} \right)&=\frac{1}{2\pi} \int_{\frac{\pi}{2}}^{\frac{3\pi}{2}} \log^{+} \big | \frac{1}{e^{re^{\iota \theta}}} \big | \, d \theta = \frac{1}{2\pi}\int_{\frac{\pi}{2}}^{\frac{3\pi}{2}}log^{+} \mid e^{-r e^{\iota \theta}} \mid \, d \theta  \\
&=\frac{1}{2\pi}\int_{\frac{\pi}{2}}^{\frac{3\pi}{2}} \log^{+} e^{\mid -re^{\iota \theta}\mid} \, d \theta = \frac{1}{2 \pi} \int_{\frac{\pi}{2}}^{\frac{3 \pi}{2}} \log^{+} e^{-r \cos \theta} \, d \theta \\ 
& = \frac{1}{2\pi}\int_{\frac{\pi}{2}}^{\frac{3\pi}{2}} -r \cos \theta \, d \theta = \frac{r}{2\pi} \sin \theta \mid_{\frac{\pi}{2}}^{\frac{3\pi}{2}}
=\frac{r}{\pi}
\end{align*}
thus $$T\left(r,\frac{1}{f}\right)= \frac{r}{\pi}.$$
\end{exmp}

In the next examples we try to develop relation between proximity function, counting function and characteristic function of two meromorphic functions $f$ and $g$.
\begin{exmp} 
Let $f$ be a meromorphic function in complex plane and let $g(z)= f(z^k)$ and $z_0$ be a pole of $f$ of order c then $g$ has pole of order $c$ at the points $z^k=z_0$ that means $k$ distinct points when $z_0 \neq 0$ and if $z_0=0$ then $g$ has a pole of order $ck$ at $z_0=0$. Thus for each pole of $f$ there corresponds $k$ poles of $g$ therefore $n(r,g)= k n(r^k,f)$.
\begin{align*}
N(r,g)&= \int_{0}^{r} \frac{n(t,g)-n(0,g)}{t} dt \, + n(0,g) \log r \\
&= \int_{0}^{r} \frac{k n(t^k,f)-kn(0,f)}{t} dt \, + kn(0,f) \log r  \\
&= \int_{0}^{r^k} \frac{n(s,f)-n(0,f)}{s} ds \, +n(0,f) \log r^k \qquad (\mbox{putting} \quad s=t^k)\\
&= N(r^k,f).
\end{align*}
Also
\begin{align*}
m(r,g) &=\frac{1}{2 \pi} \int_{0}^{2 \pi}                              \log^{+} \mid g(re^{\iota \theta}) \mid d\theta \, \\
&= \frac{1}{2 \pi} \int_{0}^{2 \pi} \log^{+} \mid f(r^k e^{\iota k \theta}) \mid d \theta \, \\
&= \frac{1}{2\pi} \int_{0}^{2 k \pi} \log^{+} \mid f(r^k e^{\iota \phi}) \mid \frac{d \phi}{k}\quad (\mbox{putting} \quad \phi= k \theta)  \\
&= \frac{k}{2 k \pi} \int_{0}^{2 \pi} \log^{+} \mid f(r^k e^{\iota \phi}) \mid d \phi \, \\
&= m(r^k,f).
\end{align*} 
Therefore $$T(r,g)=m(r^k,f)+N(r^k,f) =T(r^k,f)$$
so 
$$T(r,e^{z^k}) = T(r^{k},f)= r^k/ \pi.$$
\end{exmp}
\begin{exmp}
Let $g=f^k$ where $k$ is positive integer. Let $f$ has a pole at $z_0$ of order $c$ then $f^k$ has pole at $z_0$ of order $ck$. Thus $$n(r,f^k)=k n(r,f)$$ and 
\begin{align*}
N(r,f^k) &= \int_{0}^{r} \frac{n(t,f^k) -n(0,f^k)}{t} dt \, + n(0,f^k) \log r \\
&= \int_{0}^{r} \frac{k n(t,f)-k n(0,f)}{t} dt \, + k n(0,f) \log r \\
&= k N(r,f).
\end{align*}
Since $\log^{+} \mid f^k \mid = k \log^{+} \mid f \mid$ this gives $m(r,f^k) = k m(r,f)$ therefore $T(r,f^k) = k T(r,f)$. Thus $$T(r,e^{kz}) =k T(r,e^z) =\frac{kr}{ \pi}.$$
\end{exmp}
\begin{exmp}
Let $a$ be a finite non zero constant and $g(z)= f(az)$. If $f$ has pole at $z_0$ of order c then $g$ has pole at $z_0/a$ of order $c$ that means corresponding to each pole of $f, g$ also has only one pole of same order. Thus $n(r,g)= n(ar,f)$. Since $n(r,f)$ is a function of $r$ and $f$ it is invariant under rotation therefore $$n(r,g) = n(\mid a \mid r,f)$$ and 
\begin{align*}
N(r,g) &= \int_{0}^{r} \frac{n(t,g)-n(0,g)}{t} dt \, + n(0,g) \log r \\
& = \int_{0}^{r} \frac{n(\mid a \mid t,f) - n(0 ,f)}{t} dt \, + n(0,f) \log r \quad  \\
&= \int_{0}^{\mid a \mid r} \frac{n(s,f)-n(0,f)}{s} ds \, + n(0,f) \log r + n(0,f ) \log \mid a \mid  \\
&- n(0,f) \log \mid a \mid  \quad (\mbox{putting} \quad s= \mid a \mid t)\\
& = N(\mid a \mid r,f ) -n(0,f) \log \mid a \mid.
\end{align*}
Also since $m(r,f)$ is invariant under rotation therefore we can only perform substitution for $\theta$ to get $m(r,g) =m(\mid a \mid r,f)$ so $ T(r,g)=T(\mid a \mid r,f)- n(0,f) \log \mid a \mid$. Therefore $$T(r,e^{a z}) = T(\mid a \mid r,e^z) = \frac{\mid a \mid r}{\pi}$$ or $$T(r,e^{\iota z}) = \frac{r}{\pi}.$$
\end{exmp}
\begin{exmp}
Let $P(z)$ be a polynomial of degree k over $\mathbb{C}$;
$P(z) = a_kz^k+a_{k-1} z^{k-1}+...+ a_0$ then $e^{P(z)}$ is an entire function therefore $T(r,e^{P(z)})=m(r,e^{P(z)})$. Also
\begin{align*}
\log^{+} \mid e^{P(z)} \mid &= \log^{+} \mid e^{a_k z^k+...+a_0} \mid \\
&= \log^{+} \mid e^{a_kz^k} e^{a_{k-1}z^{k-1}}...e^{a_0} \mid \\
& \leq \sum_{j=0}^{k} \log^{+} \mid e^{a_j z^j} \mid 
\end{align*}
which gives
$ m(r,e^{P(z)}) \leq \sum_{j=0}^{k} m(r,e^{a_j z^j})$
and $ m(r,e^{a_j z^j}) =m(\mid a_j \mid r,e^{z^j}) = \frac{\mid a_j \mid r^j}{\pi} =O(r^j) $ by previous examples.
\\
Thus 
\begin{align*}
T(r,e^P) &\leq \sum_{j=0}^{k} \frac{\mid a_j \mid r^j}{\pi}  \\
&=\frac{\mid a_k \mid r^k}{ \pi}+ \sum_{j=0}^{j-1} \frac{\mid a_j\mid r^j}{\pi} \\
&= \frac{ \mid a_k \mid r^k}{ \pi} + O(r^{k-1}) 
\end{align*}
\end{exmp}
\section{First Fundamental Theorem }

\quad Nevanlinna theory has two main theorem called first fundamental theorem and second fundamental theorem. In this section we will state and prove first fundamental theorem of Nevanlinna theory for this we need Jensen\textquoteright s formula. By first fundamental theorem of Nevanlinna theory it follows that the sum $m+N$ is independent of $a \in \mathbb{C}$. 
\begin{defn}
\textbf{Order of Growth} 

Let $f$ be a meromorphic function in the complex plane then $\rho(f)$, the order of growth of $f$ is defined as 
$$\rho(f)= \limsup_{r\rightarrow \infty} \frac{\log^{+} T(r,f)}{\log r} $$.  
\\
$f$ is said to have finite order if $\rho(f) \leq \infty$ equivalently $T(r,f) =O(r)$. We saw earlier that $T(r,e^z)=r/\pi $, so $e^z$ has order 1 and similarly $\rho(e^{nz}) =n$.
\end{defn}
\begin{thm} \label{fft}
\textbf{First Fundamental Theorem}

if $f(z)$ is meromorphic in $\mid z \mid \leq R $ then for any $a \in \mathbb{C}$ and $0<r<R$, 
$$m\left(r,\frac{1}{f-a}\right)+N\left(r,\frac{1}{f-a}\right)=T(r,f)- \log \mid f(0)-a \mid +\epsilon(a,r),$$ where $\mid \epsilon(a,r) \mid \leq \log^{+} \mid a \mid + \log 2.$
\end{thm}
\begin{proof}%\hspace*{\fill} \\
Since 
\begin{align*}
m\left(r,\frac{1}{f-a}\right)+N\left(r,\frac{1}{f-a}\right) &=T\left(r,\frac{1}{f-a}\right) \\
&=T(r,f-a)-\log \mid (f-a)(0) \mid \quad \\
&= T(r,f)- \log \mid f(0)-a \mid + \epsilon(a,r),
\end{align*} where $\mid \epsilon(a,r) \mid \leq \log^{+} \mid a \mid + \log 2$.
\end{proof}%\hspace*{\fill} \\
\begin{rem} \label{rem5}
If there is no confusion we can choose to write 
$$m\left(r,\frac{1}{f-a}\right)=m(r,a)$$ and $$N\left(r,\frac{1}{f-a}\right)=N(r,a),$$ if $a$ is finite.
\\
Thus, $m(r,a)+N(r,a)=T(r,f)+O(1)$, \quad where O(1) is bounded for a finite or infinity. This is just a restatement of first fundamental theorem of Nevanlinna theory, since $\epsilon(a,r) $ and $\log \mid f(0)-a \mid $ is bounded so is O(1).
\end{rem}
\begin{exmp}
Let $S$ be a M$\ddot{o}$bius transformation; $S(f) =\frac{af+b}{cf+d}, ad-bc \neq 0$ then by first fundamental theorem
\begin{align*}
T\left(r,S(f)\right) &= T\left(r,\frac{af+b}{cf+d}\right) =T(r,a'+ \frac{b'}{f+d'}) \\
&\leq T\left(r,\frac{1}{f+d'} \right) +O(1) = T(r,f)+O(1).
\end{align*}
\end{exmp}
\begin{thm} 
\label{nifpf}
\textbf{Nevanlinna\textquoteright s Inequality for proximity function}

Let $f$ be analytic function in $ \mid z \mid \leq R $ and 
$$M(r,f) = \max \{\,\ \mid f(z) \mid , \mid z \mid = r \,\} $$ then
$$T(r,f)\leq \log^{+}M(r,f)\leq \frac{R+r}{R-r}m(r,f),$$ where $ 0 \leq r <R $
\end{thm}
\begin{proof}
Since $f$ is analytic function in $\mid z \mid = R $ thus $N(r,f)=0$ for $0 \leq r < R $. Thus, 
$$T(r,f) = m(r,f) = \frac{1}{2 \pi} \int_{0}^{2\pi} \log^{+} \mid f(re^{\iota \theta}) \mid \, d \theta   \leq \log^{+} M(r,f).$$
Since $ \mid f(re^{\iota \theta}) \mid \leq M(r,f) $, for all $ 0 \leq \theta \leq 2 \pi $ this gives that 
$$\log^{+} \mid f(re^{\iota \theta}) \mid \leq \log^{+} M(r,f),$$ for all $ 0 \leq \theta \leq 2 \pi $.
\\
Now, $M(r,f) = \mid f(z_0) \mid $, for some $z_0 = re^{\iota \theta}$, where $0 \leq \theta \leq 2\pi.$ Let $a_1,a_2,a_3,\ldots,a_p $ be zeros of $f$ in $\mid z \mid < R $ then by Poisson-Jensen formula we get 
\begin{align*}
\log M(r,f) &= \log \mid f(z_0) \mid  \\
&=\frac{1}{2 \pi} \int_{0}^{2\pi} \log \mid f(re^{\iota \phi}) \mid \frac{R^2-r^2}{R^2-2rR\cos (\theta-\phi) +r^2}\, d \phi \\
&+\sum_{i=1}^{p} \log \big | \frac{R(z-a_i)}{R^2-\bar{a_i} z} \big | \\
&\leq\frac{1}{2 \pi} \int_{0}^{2\pi} \log^{+} \mid f(re^{\iota \phi}) \mid \frac{R^2-r^2}{R^2-2rR\cos (\theta-\phi) +r^2}\, d \phi  \\
&\leq \frac{R+r}{R-r} m(R,f).  
\end{align*}
Since $\cos (\theta - \phi) \leq 1 $ this gives $ -2Rr \cos(\theta - \phi) \geq -2Rr  $ which implies $ R^2 - 2Rr \cos(\theta- \phi) +r ^2 \geq R^2 - 2Rr + ^2 = (R-r)^2.$ Also as $M(r,f) \geq 0$ therefore $\log M(r,f) =\log^{+} M(r,f)$. So $$T(r,f) \leq \log^{+} M(r,f) \leq \frac{R+r}{R-r} T(R,f).$$ 
\end{proof}
\hspace{\fill} \\
\section{Cartan\textquoteright s Identity and Convexity Theorem}

\quad H. Cartan gave an useful identity for the characteristic function which explained behaviour of characteristic function. We start by proving Cartan\textquoteright s identity theorem.
\begin{thm} 
%\hspace*{\fill} \\
Suppose that $f(z)$ is meromorphic in $\mid z \mid \leq R $ then
$$T(r,f)=\frac{1}{2\pi}\int_{0}^{2\pi}N(r,e^{\iota \theta}) \,d \theta + \log^{+} \mid f(R e^{\iota \theta}) \mid, $$ where $ 0<r<R.$
\end{thm}
\begin{proof}
Since $$ \log \mid f(0) \mid = \frac{1}{2\pi} \int_{0}^{2\pi} \log \mid f(Re^{\iota \theta}) \mid \, d \theta + N(R,f)-N\left(R,\frac{1}{f}\right)$$
is the Jensen\textquoteright s formula for meromorphic function $f,$ applying this for $f(z)=a-z$ and $R=1$ we get since $f$ is polynomial therefore has pole at $ \infty$ that is $N(R,f) =0$ and
\\
\begin{equation}
N(R,\frac{1}{f})= 
\begin{cases}
  1 ,    & \text{if $\mid a \mid < 1$} \\
  0 ,    & \text{if $\mid a \mid \geq 1$}
\end{cases} 
\notag 
\end{equation}
thus
\begin{align*}
\frac{1}{2\pi} \int_{0}^{2 \pi} \log \mid f(Re^{\iota \theta}) \mid \, d \theta &= \frac{1}{2\pi} \int_{0}^{2\pi} \log \mid a-e^{\iota\theta} \mid \, d \theta \\
& =
   \begin{cases}
 \log \mid a \mid     & \text{ if $\mid a \mid \geq 1$} \\
 \log \mid a \mid - \log \mid a \mid   & \text{ if $\mid a \mid <1$}
\end{cases}
\end{align*}
therefore
\begin{align} \label{6}
\frac{1}{2\pi} \int_{0}^{2\pi} \log \mid a-e^{\iota \theta} \mid \, d \theta = \log^{+} \mid a \mid.  
\end{align}
%\hspace{\fill} \\
Now apply Jensen\textquoteright s formula for $ f(z)-e^{\iota \theta},$ we get
$$\log \mid f(0)-e^{\iota \theta} \mid =\frac{1}{2\pi} \int_{0}^{2\pi} \log \mid f(re^{\iota \phi})-e^{\iota \theta} \mid \, d \phi +N(r,f-e^{\iota \theta})-N(r,\frac{1}{f-e^{\iota \theta}}).$$
Integrating both sides w.r.t. $\theta$
\begin{align*}
\int_{0}^{2\pi} \log \mid f(0)-e^{\iota \theta} \mid \, d \theta &= \frac{1}{2\pi} \int_{0}^{2\pi} \left( \int_{0}^{2\pi} \log \mid f(re^{\iota \phi})-e^{\iota \theta}\mid \, d \phi \right) \, d \theta \\
&+N(r,\infty)- \int_{0}^{2\pi} N(r,e^{ \iota \theta}) \, d \theta.
\end{align*}

 As double integral is absolutely convergent we can change the order of integration in double integration and using (\ref{6}) we get
$$2\pi \log^{+}\mid f(0) \mid =\int_{0}^{2\pi} \log^{+} \mid f(re^{\iota \phi}) \mid \, d \phi + N(r,\infty) - \int_{0}^{2\pi}N(r, e^{\iota \theta}) \, d \theta$$ 
which implies
$$\log^{+} \mid f(0) \mid = \frac{1}{2\pi} \int_{0}^{2\pi} \log^{+} \mid f(re^{\iota \phi})\mid \, d \phi + N(r,f-e^{\iota \theta})-\frac{1}{2\pi} \int_{0}^{2\pi} N(r,r^{\iota \theta}) \, d \theta$$
which further implies
$$\log^{+} \mid f(0) \mid + \frac{1}{2\pi} \int_{0}^{2\pi}N(r,e^{\iota \theta}) \, d \theta = m(r,f)+N(r, f) = T(r,f).$$ 
\end{proof}%\hspace*{\fill} \\
\begin{cor} 
The characteristic function $T(r,f)$ is an increasing convex function of $\log r$; for $0<r<R$.
\end {cor}%\hspace*{\fill} \\
\begin{proof}
Since $N(r,e^{\iota \theta})$ is increasing convex function of $ \log r $. We deduce at once that from the Cartan\textquoteright s identity that $T(r,f)$ has the same property.  
\end{proof}
\begin{note}
It is evident that $T(r,a)$ is a convex increasing function of $ \log r $. On the other hand $m(r,a)$ be neither convex nor increasing. For example consider, $f(z)=\frac{z}{1-z^2}$ then
 $\mid f(z) \mid= \mid \frac{z}{1-z^2} \mid \leq \frac{\mid z \mid}{1-\mid z \mid ^2} < 1$ for $\mid z \mid < \frac{1}{2}$. This gives that $ \log^{+} \mid f(z) \mid = 0 $ for $ \mid z \mid < \frac{1}{2}$ which gives $  m(r,f)=0$ for $r < \frac{1}{2} $ and $\mid f(z) \mid < 1$ for $\mid z \mid > 2$. This also gives that $\log^{+} \mid f(z) \mid =0 $ which implies $  m(r,f)=0$ for $r > 2.$ On the other hand since $ f(\pm 1) = \infty \Rightarrow \log^{+} f(e^{\iota \theta}) = \infty \Rightarrow m(1,f) > 0$. Thus $m(1,f)$ is neither increasing nor convex.
 \end{note}
\section{Second Fundamental Theorem}
\quad Second fundamental theorem is an important result in Nevanlinna theory. Using this result a short proof of Picard\textquoteright s theorem can be given. Here we will state and prove second fundamental theorem. Firstly a short explanation of $S(r,f)$ and $N_1(r,f)$ is given, where $S(r,f)$ is known as the error term of the second fundamental theorem and $N_1(r,f)$ is another term in this main theorem.
\\
\begin{defn} Let $f$ be a meromorphic function in $\mid z \mid < R_0 \leq + \infty$ 
\begin{align*}
N_1(r,f) &=2N(r,f)-N(r,f')+N \left(r,\frac{1}{f'}\right) \\ 
&= N(r,f) -\bar{N}(r,f)+N\left(r,\frac{1}{f'}\right)
\end{align*}
and $S(r,f)$ is given such that if 
\\
\item[(i)] If $ R_0 =+\infty$, $$ S(r,f) = O(\log T(r,f) + O(\log r),$$
as $ r \rightarrow \infty$ through all values if $f(z)$ has finite order and as $r \rightarrow \infty$ outside a set E of finite linear measure otherwise.
\\
\item[(ii)] If $0 < R_0 < + \infty$, $$ S(r,f) =O\lbrace \log^{+}
 T(r,f)+ \log \frac{1}{R_0-r}\rbrace, $$
as $r \rightarrow R_0$ outside a set E such that $$ \int_{E} \frac{1}{R_0-r} \, dr < + \infty$$
\end{defn}  
\begin{thm} \label{sft}
\textbf{Second Fundamental Theorem} (see [~\cite{hay}])

Suppose that $f(z)$ is non-constant meromorphic function in $\mid z \mid \leq r$. Let $a_1, a_2,\ldots, a_q, q>2$ be distinct complex numbers, $\delta > 0$ and suppose that 
$\mid a_\mu-a_\nu \mid \geq \delta \quad \mbox{for } 1 \leq \mu <\nu \leq q $ then
$$m(r,f)+ \sum_{\nu=1}^{q}m(r,a_\nu) \leq 2 T(r,f)-N_1(r)+S(r),$$
where $N_1(r)$ is positive and is given by 
\\
$N_1(r)= N\left(r,\frac{1}{f\textquoteright}\right)+ 2 N(r,f)-N(r,f\textquoteright)$
and 
\\
$S(r)=m\left(r,\frac{f\textquoteright}{f}\right)+m\left(r,\sum_{\nu=1}^{q}\frac{f\textquoteright}{f-a_\nu}\right)+q \log^{+} \frac{3q}{\delta}+\log2 +\log \mid \frac{1}{f\textquoteright(0)} \mid $
\\
with modifications if $f(0)=\infty$ or $f\textquoteright(0)=0$.
\\
This result is also well known as \textbf{Nevanlinna\textquoteright s fundamental inequality}.
\end{thm}
\textbf{Alternative proof of Picard\textquoteright s Theorem}
\\
Suppose $f$ is an meromorphic function such that $f$ omits $0,1 $ and $\infty $ and by first fundamental theorem \ref{fft}
 $$ m(r,a)=T(r,f) +O(1)$$ for $a= 0,1$ and $\infty $ so that $$ m(r,0)+m(r,1) +m(r,\infty) = 3 T(r,f) +O(1) $$ 
which is a contradiction to second fundamental theorem \ref{sft} of Nevanlinna theory. This gives Picard\textquoteright s theorem. 

Next we state Nevanlinna\textquoteright s Estimate ~\cite{hay} and Hiong\textquoteright s Estimate ~\cite{hio} which will be used in later chapters.
\begin{res}
Suppose that $f(z)$ is meromorphic in $\mid z \mid \leq R $, $0<r<R$ and $f(0) \neq 0, \infty $  
\begin{align*}
m\big(r,\frac{f\textquoteright}{f}\big)<& 4 \log^{+} T(R,f)+4 \log^{+} \log^{+}\big| \frac{1}{f(0)} \big|+ 5 \log^{+} R \\
&+ 6 \log^{+} \frac{1}{R-r} + \log^{+} \frac{1}{r}+14.
\end{align*}
\ \ \ \ \ \ \ \ \ \ \ \ \ \ \ \ \ \ \ \ \ \ \ \ \ \ \ \ \ \ \ \ \ \ \ \ \ \ \ \ \ \ \ \ \ \ \ \ \ \ \ \ (Nevanlinna Estimate) \label{nev est}
\begin{align*}
m\left(r,\frac{f^{(k)}}{f}\right)<C \left\lbrace \log^{+}T(R,f)+\log^{+} \frac{1}{R-r}+ \log^{+} \frac{1}{r} \log^{+}R+\log^{+}\log^{+} \big | \frac{1}{f(0)} \big |+1 \right\rbrace
\end{align*}
\ \ \ \ \ \ \ \ \ \ \ \ \ \ \ \ \ \ \ \ \ \ \ \ \ \ \ \ \ \ \ \ \ \ \ \ \ \ \ \ \ \ \ \ \ \ \ \ \ \ \ \ (Hiong Estimate)\label{hio est}
\end{res}
\chapter{Normal Families}

\quad Concept of normal families came in light after P. Montel\textquoteright s remarkable works on normal families. Here a brief discussion of normal families is given so that we become familiar with definition, terms, results etc. to be helpful in later chapters. 

\section{Spherical Metric}

\quad In this section we describe spherical metric on extended complex plane, an important metric that we have to use through all this present work. We start with  stereographic projection.
\\
\\
\textbf{Stereographic Projection}

In stereographic projection we wish to identify extended complex plane $\mathbb{C}_ \infty $ by a unit sphere in $\mathbb{R}^3$. For this, we wish to introduce a distance function on $\mathbb{C}_ \infty $.

Let $S=\{ \, (a,b,c) \in \mathbb{R}^3; a^2+b^2+(c-\frac{1}{2})^2=\frac{1}{4} \,\}$ be sphere in $\mathbb{R}^3$ called Riemann sphere so that $XY$- plane  is tangent to $S$ at origin representing the complex plane $\mathbb{C}$ and $N=(0,0,1)$ is the north pole. Any point $z \in \mathbb{C}$ has its co-ordinates as $(a,b)$ in $XY$- plane and the line joining $N$ and $z$ cuts the sphere $S$ exactly at one point say $Z$ with co-ordinates $(a_1,b_1,c_1)$. Equation of line joining $(a,b,0) $ and $(a_1,b_1,c_1)$ is 
\begin{equation}
\frac{a_1-0}{a-0}=\frac{b_1-0}{b-0} = \frac{c_1-1}{0-1} =t
\notag
\end{equation} 
which gives $$a_1=at, b_1=bt \quad \mbox{and } c_1=1-t.$$
Since $Z$ lies on $S$ therefore $(a_1,b_1,c_1)$ satisfies equation of $S$ so we get 
$$ t^2 \mid z \mid^2+ (1-t-\frac{1}{2})^2 = \frac{1}{4}$$ implying $t=0 $ or $t=\frac{1}{\mid z \mid ^2+1}.$

For $t=0$ we get north pole $N$ corresponding to $\infty$ and for $t= \frac{1}{ \mid z \mid ^2+1}$ we have $$ a_1= \frac{a}{\mid z \mid^2+1}, b_1=\frac{b}{\mid z \mid^2+1} \quad \mbox{and } c_1= \frac{\mid z \mid}{\mid z \mid^2+1}.$$

In this way corresponding to each point in extended complex plane $\mathbb{C}_\infty$ we can find exactly one point in $S.$

Let us now define distance function on the extended  plane in the following manner:

 for $z=(a,b)$, $z_1=(a\textquoteright,b\textquoteright)$ in $\mathbb{C}_\infty$ define distance between z and $z_1$ written as $ \chi(z, z_1)$, to be the Euclidean distance between the corresponding points Z and $Z_1$ in $\mathbb{R}^3$. If $Z=(a_1,b_1,c_1)$ and $Z_1 =(a_2,b_2,c_2)$ then 
$$\chi(z,z_1) = [(a_1-a_2)^2+(b_1-b_2)^2+(c_1-c_2)^2]^ \frac{1}{2}$$
which on simplification gives $$\chi(z,z_1)= \frac{\mid z-z_1 \mid}{\sqrt{1+\mid z \mid^2} \sqrt{1+\mid z_1\mid^2}}$$ 
we define $\chi(z,z_1)$ as chordal distance between $z$ and $z_1$. Clearly $\chi (z,z_1) = \chi (\frac{1}{z}, \frac{1}{z_1})$  and $ \chi (z,z_1) \leq \mid z-z_1 \mid $.

An arc or curve say $\gamma  $  in $\mathbb{C}$ under stereographic projection has spherical arc length element
\[ ds= \frac{\mid dz \mid}{1+z^2} \]
on the Riemann sphere.
\\
Thus the spherical length of $\gamma$ is given by
\[ L(\gamma)= \int_{\gamma} \frac{\mid dz \mid}{1+z^2}. \]

Now we have another metric on S define as,
\[ \sigma(z,z_1)= inf  \{\,L(\gamma) \,\} \] 
where infimum is taken over all differentiable curve $\gamma $  over S which join $z$ and $z_1$.
\\
This metric $\sigma $ is termed as \textbf{spherical metric} on S. However, these two metrics on S are equivalent as 

$$\chi(z,z_1) \leq \sigma (z,z_1) \leq \frac{\pi}{2}\chi(z,z_1).$$
Thus we can treat these two metrics as the one metric.
\hspace*{\fill} \\
\\
\textbf{Spherical Derivative}

Let $f(z)$ be meromorphic on a domain $\Omega$ if $ z \in \Omega $ is not a pole the derivative in the spherical metric, called the spherical derivative, is given by,
\begin{align*}
 f^{\sharp}(z)&= \lim_{z\textquoteright \rightarrow z } \frac{\chi (f(z),f(z\textquoteright))}{ \mid z-z\textquoteright \mid} \\
 &  = \lim_{z\textquoteright\rightarrow z } \frac{\mid f(z)-f(z\textquoteright) \mid}{\mid z -z\textquoteright \mid} \frac{1}{\sqrt{1+\mid f(z) \mid ^2}} \frac{1}{\sqrt{1+\mid f(z\textquoteright) \mid ^2}} \\
 &= \frac{\mid f\textquoteright(z) \mid}{1+ \mid f(z) \mid ^2}
\end{align*}
and if $z$ is a pole of $f$ then define
$$f^{\sharp}(z) = \lim_{w \rightarrow z } \frac{\mid f\textquoteright(w) \mid}{1+\mid f(w) \mid ^2}.$$
Thus $f^{\sharp}(z)$ is continuous. Also $f^{\sharp}(z)=\left( \frac{1}{f(z)} \right)^{\sharp}$ is easy to verify.

Let $\gamma$ be a differentiable arc or curve in $\Omega$ then the image of $ \gamma$ on the Riemann sphere has arc length element \[ ds= f^{\sharp}(z) \mid dz \mid.
\]
Spherical length of $\gamma$ is given by
\[ \int_{\gamma} f^{\sharp}(z) \, \mid dz \mid. \]
In particular, let $f(z)$ be meromorphic in $\mid z \mid \leq r $ denote 
\[ L(r)=\int_{\mid z \mid =r} f^{\sharp}(z) \, \mid dz \mid = \int_{0}^{2\pi} \frac{\mid f\textquoteright(re^{\iota \theta}) \mid r}{1+ \mid f(re^{\iota \theta})\mid ^2 } \, d \theta \]
is the length of the image of the circle $\mid z \mid = r $ on the Riemann sphere.
%\hspace*{\fill} \\
\section{Types of Convergence}

\quad Depending on different types of metrics on $\mathbb{C}$ convergence of a family of functions is distinguished here.
\begin{defn} 
\textbf{Uniform Convergence}

A sequence of function ${f_n}$ converges uniformly to $f$ on a set $\Omega \subseteq \mathbb{C}$ if for $\epsilon > 0$ there exist a number $n_0$ such that 
$$\mid f_n(z)- f(z) \mid < \epsilon, $$ for all $n > n_0 \in \mathbb{N} $ and for all $z \in \Omega. $
\end{defn} 
\begin{defn}
\textbf{Spherical uniform Convergence} 

A sequence of functions ${f_n}$ converges spherically uniformly to $f$ on a set $\Omega \subseteq \mathbb{C}$ if for $\epsilon>0$ there exist a number $n_0 \in \mathbb{N}$ such that 
$$\chi (f_n (z), f(z)) < \epsilon, $$ for all $ n > n_0 $ and for all $z \in \Omega$.
\end{defn}
\begin{note}
In spherical uniform convergence we use chordal metric instead of spherical metric as these two metrics are equivalent.

Also note that if ${f_n}$ converges uniformly to $f$ on $\Omega \subseteq \mathbb{C}$ then it converges spherically uniformly to $f$ on $\Omega$  but not conversely. 
\end{note}
\begin{lem} (see ~\cite{sch}) \label{f bdd}
Suppose that ${f_n}$ be a sequence of function on $\Omega$ and $f_n \rightarrow f $ spherically uniformly on $\Omega$ then $f_n \rightarrow f $ uniformly if $f$ is bounded on $\Omega$.
\end{lem}
\begin{defn} 
A sequence of function ${f_n}$ converges \textbf{ uniformly} (\textbf {spherically uniformly}) on compact subsets of a domain $\Omega$ to a function $f(z)$ if for any compact subset $ K \subseteq 
\Omega$ and $\epsilon > 0$ there exist a number $n_0 \in \mathbb{N}$ and $n_0 = n_0 (K, \epsilon)$ such that
$$\mid f_n(z)-f(z) \mid < \epsilon \quad ( \chi (f_n(z),f(z)) < \epsilon ),
$$ 
for all $n > n_0 $  and for all $z \in K $. 
\end{defn}
\begin{defn} \label{ch3,sec2,def4}
\textbf{Normal Convergence}

A sequence of function ${f_n}$ converges uniformly or spherically uniformly on compact subsets of a domain $\Omega$, then we say that the sequence converges normally in $\Omega$. 
\end{defn}
\begin{thm} \label{weierstrass}
\textbf{Weierstrass Theorem}

Let ${f_n}$ be a sequence of analytic functions on a domain $\Omega$ which converges uniformly on compact subsets of $\Omega$ to a function $f$ then $f$ is analytic in $\Omega$ and sequence of derivatives ${f^{(k)}_n}$ converges uniformly on compact subsets of $\Omega$ to $f^{(k)}$; $k=1,2,3,\ldots $.
\end{thm}
\begin{proof}
For any arbitrary $ z_0 \in \Omega $, choose a disk $ D(z_0,r)$  in $ \Omega $ and
\\
$ C_r = \{\, z ,  \mid z_0 - z \mid = r \, \} $. By the hypothesis we have that for given $\epsilon > 0 $ there exist $n_0 \in \mathbb{N}$ such that if $n \geq n_0 $ then
$$\mid f_n(w) - f(w) \mid < \epsilon \quad \mbox{for all }  w \in C_r. $$
Define
$$F_k(z)=\frac{k!}{2 \pi \iota} \int_{C_r}\frac{f(w)}{(w-z)^{k+1}} \, dw,  k= 1,2,3,\ldots$$
for $z \in d(z_0, \frac{r}{2})$. By Cauchy Integral formula
$$f_n^{(k)}(z)=\frac{k!}{2 \pi \iota} \int_{C_r} \frac{f_n(w)}{(w-z)^{k+1}} \, dw, k=1,2,3, \ldots. $$ 
Then
\begin{align*}
\mid F_k(z)-f_n^{(k)}(z) \mid & \leq \frac{k!}{2 \pi} \int_{C_r} \frac{\mid f(w)-f_n(w) \mid}{\mid w-z \mid ^{k+1}} \, \mid dw \mid  \\
& <  \frac{k!  \epsilon}{2 \pi} \int_{C_r} \frac{1}{\mid w -z \mid ^{k+1}} \, \mid dw \mid \\  
&= \frac{k!}{2\pi} \int_{C_r}\frac{1}{\mid w-z_0-(z-z_0)\mid^{k+1}} \, \mid dw \mid  \\
&\leq \frac{k!  \epsilon}{2 \pi}\int_{C_r} \frac{1}{(\mid w-z_0 \mid -\mid z-z_0 \mid)^{k+1}} \, \mid d w\mid.
\end{align*}
Since $w-z_0 = re^{\iota \theta}$, therefore $dw-dz_0=r \iota e^{\iota \theta}.$
Also $dz_0=0 $ and by taking absolute value on both side we get
\\
$\mid dw \mid = r $ 
\\
$\mid w-z_0 \mid = r $ and $\mid z-z_0 \mid < \frac{r}{2}$
\\
$\mid w-z_0 \mid - \mid z-z_0 \mid < r-\frac{r}{2}=\frac{r}{2}.$
\\
Thus
$$\mid F_k(z)-f_n^{(k)}(z) \mid \leq \frac{k! \epsilon r 2^{k+1}}{r^{k+1}} =\frac{k! \epsilon 2^{k+1}}{r^k}, $$  for $n \geq n_0, \quad \mbox{forall } z \in D(z_0,\frac{r}{2})$ that means $$f_n^{(k)} (z) \rightarrow F_k (z)$$ uniformly on $D(z_0, \frac{r}{2})$ for $k=0,1,2,3, \ldots$.
\\
For $k=0$
$$f_n(z)\rightarrow F_0 (z)$$ uniformly on $D(z_0, \frac{r}{2})$
and $$ f_n (z) \rightarrow f(z) $$ uniformly on $ D(z_0, \frac{r}{2})$ which gives $ f(z)=F_0(z) $ on $D(z_0, \frac{r}{2})$ and $F_0(z)$ is analytic on  $D(z_0, \frac{r}{2})$ as  $F_0(z)$ is differentiable on $D(z_0, \frac{r}{2})$.

Thus $f(z)$ is analytic on $D(z_0, \frac{r}{2})$ and by compactness argument we can generalize our result to whole $\Omega$.
\end{proof}
\begin{thm} \label{rouche}
\textbf{Rouche\textquoteright s Theorem} (see ~\cite{con})

Let $f \mbox{and} g$ meromorphic in a neighbourhood of $\overline D(z_0,r) $ with no zeros or poles on the circle $ \gamma = \{\, z, \mid z-z_0 \mid = r \, \}$. if $ Z_f,Z_g $ (and $P_f, P_g $) are number of zeros (and poles) of $f \mbox{and} g$ inside $\gamma$ counted according to their multiplicities and if $$\mid f(z)+g(z) \mid < \mid f(z) \mid $$ on $\gamma$, then
\begin{center}
$Z_f-P_f = Z_g-P_g$.
\end{center} 
\end{thm} 
\begin{thm} \label{hurwitz}
\textbf{Hurwitz Theorem} 

Let ${f_n}$ be a sequence of analytic functions on a domain $\Omega$ which converges uniformly on compact subsets of $\Omega$ to a non-constant analytic function $f(z)$. If $f(z_0) = 0$ for some $z_0 \in \Omega $, then for each $r>0$  sufficiently small, there exist an $N=N(r)$, such that for all $ n > N $, $f_n(z)$ has the same number of zeros in $D(z_0, r)$ as does $f(z)$ (counting according to multiplicities).   
\end{thm} 
\begin{proof}
Let $z_0 \in \Omega$ be arbitrary and $D(z_0,r)$ be a disk so small in $\Omega$ so that $f(z)$ does not vanish in $\{\, z , \mid z-z_0 \mid \leq r \, \} = K(z_0,r)$, except at $z_0$.
\\
Then for some $m>0$, $ \mid f(z) \mid > m $ on $\mid z-z_0 \mid = r $ by the continuity of $f$. As $f_n \rightarrow f $ uniformly on $\mid z-z_0 \mid = r $ we get
$$ \mid f_n(z)-f(z) \mid < m < \mid f(z) \mid,$$  forall $z$ such that $\mid z-z_0 \mid = r $ for sufficiently large $n$.

Thus by Rouche\textquoteright s theorem \ref{rouche}, we get that for sufficiently large $n$ (or can say that there is an $N$ such that for $n>N$) , $f_n(z)$ has same number of zeros in $D(z_0,r)$ as does $f(z)$ counting according to multiplicities.  
\end{proof} 
\begin{defn}
\textbf{Locally Boundedness}

A family of functions $\mathcal{F}$ is locally bounded on a domain $\Omega$ if for each $z_0 \in \Omega$, there is a positive number $M=M(z_0)$ and a neighbourhood $D(z_0,r) \subseteq \Omega $  such that $$\mid f(z) \mid \leq M, $$  for all  $z \in D(z_0, r)$ and  for all $f \in \mathcal{F} $. 

Or we can say in other words that $\mathcal{F}$ is uniformly bounded in a neighbourhood of each point of $\Omega$ therefore we can termed this as local uniform boundedness.

Since K compact therefore it has an finite open cover of such disk, it follows that a locally bounded family $\mathcal{F}$ is uniformly bounded on compact subsets of $\Omega$.
\end{defn}
\begin{lem} \label{f`lcb}
If $\mathcal{F}$ is family of locally bounded analytic functions on a domain $\Omega$, then the family of derivatives $\mathcal{F\textquoteright} = \left\lbrace f\textquoteright, f \in \mathbb{F} \right\rbrace$ form a locally bounded family in $ \Omega.$
\end{lem}
\begin{proof}
Since $\mathcal{F}$ is locally bounded therefore for any $z_0 \in \Omega$ there is a closed neighbourhood $K(z_0,r) = \left\lbrace z, \mid z-z_0 \mid \leq r \right\rbrace \subset \Omega$ and a constant $M=M(z_0)$ such that $\mid f(z) \mid \leq M, z \in K(z_0,r)$. Then for $z \in D(z_0,\frac{r}{2})$ and $w \in C_r=\left\lbrace z, \mid z-z_0 \mid =r \right\rbrace$, the Cauchy formula gives 
$$ \mid f\textquoteright (z) \mid \leq \frac{1}{2\pi} \int_{C_r} \frac{\mid f(w)\mid}{\mid w-z \mid ^2} \mid dw \mid < \frac{4M}{r},$$
for all $f\textquoteright \in \mathcal{F\textquoteright}$, so that $\mathcal{F\textquoteright}$ is locally bounded.
\end{proof}
Converse of above lemma is not true for this consider family $\mathcal{F} =\left\lbrace n: n=1,2, \ldots \right\rbrace$ which is not locally bounded whereas $\mathcal{F}\textquoteright= \left\lbrace \textbf{O} \right\rbrace$ is so. However a partial converse of above lemma does hold.
\begin{lem} \label{con f`lcb}
Let $\mathcal{F}$ be a family of analytic functions on $\Omega$ such that the family of derivatives $\mathcal{F\textquoteright}$ is locally bounded and suppose there is some $z_0 \in \Omega$ with $\mid f(z_0) \mid\leq M < \infty$ for all $f \in \mathcal{F}$. Then $\mathcal{F}$ is locally bounded.
\end{lem} 
\begin{proof}
Since $\mathcal{F\textquoteright}$ is locally bounded therefore there is bound for neighbourhood of each point. Let $z\in \Omega $ and $C \subset \Omega$ be a path from $z_0$ to $z$ then $$ \mid f(z) \mid \leq \mid f(z_0) \mid + \int_C \mid f\textquoteright(w) \mid \mid dw \mid \leq M+M_1. l(C),$$
$f\in \mathcal{F}$, where $M_1$ is bound for $\mathcal{F\textquoteright}$ in neighbourhood of each point on $C$. So $\mathcal{F}$ is locally bounded in $\Omega.$  
\end{proof}
\begin{defn}
\textbf{Spherical Continuous}

A function $f$ is spherical continuous on $z_0 \in \mathbb{C}$ if for every $\epsilon >0 $ there is $\delta > 0$ such that
$$\chi (f(z),f(z_0)) < \epsilon, $$ for $\mid z-z_0 \mid < \delta $.
\end{defn}
\begin{defn}
A family of functions $\mathcal{F}$ on $\Omega$ is said to be \textbf{equicontinuous} (\textbf{spherically equicontinuous}) at a point $z_0 \in \Omega$ if for all $\epsilon>0$ there is a $\delta > 0$ and $\delta=\delta (\epsilon , z_0)$ such that 
$$ \mid f(z),f(z_0) \mid < \epsilon \quad (\chi (f(z),f(z_0)) < \epsilon ),$$ whenever $\mid z-z_0 \mid < \delta \quad \mbox {for every } f \in \mathcal{F}$.
\end{defn}
\begin{lem} \label{f sph cts}
If $f$ is meromorphic on $\Omega \subseteq \mathbb{C}$
then $f$ is spherical continuous.
\end{lem}
\begin{proof}
For $z_0 \in \Omega $ if $f$ is analytic at $z_0 \in \Omega $ so $f$ is continuous at $z_0$ therefore for every $\epsilon > 0$ there is $\delta > 0 $ such that
$$\chi(f(z),f(z_0)) \leq \mid f(z)-f(z_0) \mid < \epsilon,$$  for  $\mid z-z_0 \mid < \delta $. Thus $f$ is spherical continuous on $z_0 \in \Omega $.

If $z_0$ is a pole of $f$ in $\Omega$ then $\frac{1}{f}$ is analytic at $z_0$ and as done above since $$\chi(f(z),f(z_0)) =\chi\left(\frac{1}{f(z)}, \frac{1}{f(z_0)}\right) $$ we get that $f$ is spherical continuous at $z_0$.
\end{proof}
\section{Normality of Analytic (Meromorphic) Functions}
\begin{defn} 
\textbf{Normal Family of analytic function}

A family of analytic functions $\mathcal{F}$ on a domain $\Omega \subseteq \mathbb{C}$ is said to be normal if every sequence of functions $ \{\, f_n \,\} \subseteq \mathcal{F}$ contains either a subsequence which converges to a limit function $ f,  \mbox{which is not identically equal to}$  $\infty $, uniformly on each compact subsets  of $\Omega$ or a subsequence which converges uniformly to $\infty$ on each compact subsets of $\Omega$.

By Weierstrass theorem \ref{weierstrass} $f$ is analytic function and for the second case we have that for any compact subset $K$ of $\Omega$ and constant $M>0$ $$ \mid f_n(z)\mid > M, $$ for all $z \in K $ and for sufficiently large $n$.
\end{defn} %\hspace*{\fill} \\
\begin{defn} 
\textbf{Normal Family of meromorphic function}

A family $\mathcal{F}$ of meromorphic functions in a domain $ \Omega$ is normal in $\Omega$ if every sequence in $\mathcal{F}$ has a subsequence converging to $f$ spherically uniformly on compact subsets of $\Omega$. 
\end{defn}%\hspace*{\fill} \\
\begin{lem} \label{f uni sph cts}
Let $ \{\, f_n \,\}$ be a sequence of spherically continuous functions converging to $f$ spherically uniformly on compact subset of $\Omega$ then $f$ is uniformly spherically continuous in $\Omega$.
\end{lem}
\begin{proof}
We have to show that $f$ is uniformly spherically continuous in $\Omega$ therefore we have to show that $f$ is uniformly spherically continuous on each point of $\Omega$. For this let $z_0 \in \Omega $. 

By uniform spherical convergence of $ \{\,f_n \,\}$ to $f$ we get that for given $\epsilon > 0 $ and K ($z_0 \in K $) compact subset of $\Omega$ there is an $ n_0 \in \mathbb{N} $ such that $$ \chi(f_n(z),f(z)) < \frac{\epsilon}{3}, $$ for all $z \in K \subseteq \Omega $ and for all $ n \geq n_0 $.

By the spherical continuity of $f_n \quad \mbox{forall } n $ in particular of $f_{n_0}$ at $z_0$ we get that there is $\delta > 0$ such that $$\chi(f_{n_0}(z),f_{n_0}(z_0)) < \frac{\epsilon}{3},$$ whenever $z, z_0 \in \Omega $ and $\mid z-z_0 \mid < \delta $. 
\\
Now consider
\begin{align*}
\chi(f(z),f(z_0)) &\leq \chi(f(z),f_{n_0}(z)) +\chi(f_{n_0}(z),f_{n_0}(z_0))+\chi(f_{n_0}(z_0),f(z_0)) \\
& < \frac{\epsilon}{3} + \frac{\epsilon}{3} + \frac{\epsilon}{3} = \epsilon,
\end{align*}
whenever $z,z_0 \in \Omega $ and $\mid z-z_0 \mid < \delta $.

This gives that $f$ is uniformly spherically continuous on $z_0 \in \Omega $ which implies uniformly spherical continuity of $f$ on $\Omega $. 
\end{proof}
\begin{lem} \label{np}
Let $\{\,f_n \,\}$ be a sequence of meromorphic functions on $\Omega \subseteq \mathbb{C}$. Then $\{\,f_n \,\}$ converges to $f$ spherically uniformly on $\Omega$ if and only if about each point $z_0$ there is a closed disk $ \overline D(z_0,r)$ in $ \Omega$ such that 
\[ \mid f_n - f \mid \rightarrow 0 \]
or
\[ \big | \frac{1}{f_n}-\frac{1}{f} \big | \rightarrow 0 \]
uniformly on  $ \overline D(z_0,r)$.
\end{lem}
\begin{proof}
If $ \mid f_n- f \mid \rightarrow 0$ or $\big | \frac{1}{f_n}- \frac{1}{f} \big | \rightarrow 0 $ on closed disk $ \overline D(z_0,r)$ around each $z_0 \in \Omega $.

Then since each compact subset of $\Omega$ can be consider as the finite union of such type of closed disks and $$\chi (f_n,f) \leq \mid f_n-f\mid $$ or $$\chi(f_n,f) = \chi\left(\frac{1}{f_n},\frac{1}{f}\right) \leq \big | \frac{1}{f_n}-\frac{1}{f} \big | $$
this gives that $\chi(f_n,f) \rightarrow 0 $ on each compact subset of $\Omega$ that mean $f_n \rightarrow f $ spherically uniformly on $\Omega$.

Conversely, suppose $f_n \rightarrow f $ spherically uniformly on $\Omega$. Then we have two cases: 
\\
\textbf{Case(i)}

There is an $z_0 \in \Omega$ such that $f(z_0) \neq \infty $ then by Lemma(\ref{f uni sph cts}) we get that $f$ is spherically continuous in $\Omega$, in particular at $z_0 \in \Omega $ thus we get for $ \epsilon > 0$ there is an $r>0 $ such that
\[ \chi \left(f(z), f(z_0)\right) < \epsilon \quad \mbox{whenever }\mid z -z_0 \mid < r. \]
 Therefore we can say that $f$ is bounded in a closed disk $ \overline D(z_0,r) \subseteq \Omega $. Thus by Lemma (\ref{f bdd}) we have that $f_n \rightarrow f$ uniformly on $ \overline D(z_0,r) \subseteq \Omega $. Here we can note that $f$ is analytic on $ D(z_0,r)$. 
\\
\textbf{Case(ii)}

There is an $ z_0 \in \Omega $ such that $ f(z_0) = \infty $. Then as done in case(i) for $f$ we get that there is a closed disk $ \overline D(z_0,r) \subseteq \Omega $ on which $\frac{1}{f}$ is bounded. Now since $\chi (\frac{1}{f_n}, \frac{1}{f}) = \chi (f_n, f)$ gives that $\frac{1}{f_n} \rightarrow \frac{1}{f}$ spherically uniformly on $\overline D(z_0,r)$. Again by Lemma (\ref{f bdd}) we get that $\frac{1}{f_n} \rightarrow \frac{1}{f}$ uniformly on $ \overline D(z_0,r)$. Clearly $\frac{1}{f}$ is analytic on $D(z_0,r)$.       
\end{proof}
\begin{cor} \label{f mero}
Let ${f_n}$ be a sequence of meromorphic function on $\Omega$ which converges spherically uniformly on compact subsets of  $\Omega$ to $f$ then $f$ is either meromorphic function on $\Omega$ or is identically equal to $\infty$. Here also we have that $\frac{1}{f}$ is analytic in $D(z_0,r)$.
\end{cor}
\begin{proof} 
Suppose $f$ is not identically equal to $\infty$ then there is atleast one $z_0 \in \Omega $ such that $f(z_0) \neq \infty $  then by case(i) of above Lemma give that $f$ is analytic in disk $D(z_0,r)$. Thus we can say that $f$ is analytic at the point where it is finite.

Next, we have to show that the only singularity that $f$ has is isolated singularity. For this let $f(z_0) = \infty $ and let $z_n \rightarrow z_0 $ where $\{\,z_n \,\} \subseteq \Omega $ such that $f(z_n) = \infty $ for all $n$. By case(ii) of above lemma we have that $\frac{1}{f}$ is analytic in $D(z_0,r)$ and $\frac{1}{f(z_n)} = 0 $ for all n which gives by identity theorem that $\frac{1}{f} \equiv 0 $ in $D(z_0,r)$ or $f\equiv \infty $ in $D(z_0,r) $.

Suppose $ X = \{\, z \in \Omega , f(z) =\infty \,\} = f^{-1} (\infty)$. This X will be the union of disk  around each singular point of $f$ thus X is open set. And since $f$ is continuous then inverse image of closed set is closed so X is closed here. So by connectedness of $\Omega $ we have that $X= \Omega$ thus we get $f$ is identically equal to $\infty $ on $\Omega $, which gives a contradiction. Therefore $f$ is meromorphic in $\Omega $.      
\end{proof}
\begin{exmp}

Let $\mathcal{F}$ =$ \{\, nz , n=1,2,3,\ldots \,\} $ be a family of analytic functions on $\mathbb{C}$. If 
$$f_n(z)=nz $$ then $$f_n(z) \rightarrow \infty $$ as $n \rightarrow \infty $ uniformly on each compact subsets of $\mathbb{C}$           
and $$f_n \rightarrow 0 $$ as $ n \rightarrow \infty $ uniformly on each compact subsets of $\mathbb{C}$. 

Thus as $f_n \rightarrow f $ uniformly on compact subsets of $\mathbb{C}$ where 
\\ 
\begin{equation} 
f(z)=
\begin{cases}
   0         & \text{if $ z=0 $} \\
    \infty   & \text{if $z \neq 0 $}
\end{cases}
\notag
\end{equation}

This gives that $\mathcal{F}$ is not normal in $\mathbb{C}$ or we can say that $\mathcal{F}$ is not normal in a domain  which contain 0.
\end{exmp}
\begin{exmp}
Consider meromorphic functions family $$ \mathcal{F} = \{ \frac{n}{z}; n=1,2,3,\ldots \, \} $$ on $\Omega$ then $f_n \rightarrow \infty$ on $\Omega$.
\end{exmp}

Now here are some theorem with no proof, which we will use in later chapters.
\begin{thm} \label{sch montel}
If $\mathcal{F}$  is a locally bounded family of analytic functions on a domain $\Omega$  then $\mathcal{F}$ is a normal family in $\Omega$. 
\end{thm}

Converse of this theorem is not true for example, let $f$ be a zero free analytic function in $\Omega$ and $\mathcal{F}=\{\, cf; c \in \mathbb{R}^+ \,\}$ then $\mathcal{F}$ has limit point as either finite or identically infinity so $\mathcal{F}$ is normal but it is not locally bounded.

However, it is possible to establish a partial converse of this theorem.
\begin{thm} \label{con montel}
Let $\mathcal{F}$ be a family of analytic functions on $\Omega$ such that every sequence of functions in $\mathcal{F}$ has a subsequence converging uniformly on compact subsets to an analytic function. Then $\mathcal{F}$ is locally bounded on $\Omega.$
\end{thm}
%\hspace*{\fill} \\
\begin{thm} \label{f sph equicts}
%\hspace*{\fill} \\
A family $\mathcal{F}$ of meromorphic functions in a domain $\Omega$ is normal in $\Omega$ if and only if $\mathcal{F}$ is spherically equicontinuous in $\Omega$.
\end{thm}
\begin{defn} 
$\mathcal{F}$ is normal of analytic (meromorphic) functions at a point $z_0$ of $\Omega$ if it is normal in a neighbourhood of $z_0$.
\end{defn}
\begin{thm} \label{aepn}
Let $\mathcal{F}$ be family of analytic (meromorphic) functions on $\Omega$ then $\mathcal{F}$ is normal in $\Omega$ if and only if it is normal at each point of $\Omega$.
\end{thm}

For the proofs of these results one can refer ~\cite{sch}.  
\chapter{Basic Normality Criterion}

\quad From this chapter our whole work is done through  Nevanlinna theory. Here, we will prove Montel \textquoteright s theorem using Nevanlinna theory. We are considering family of functions on unit disk $\Delta$ instead of a arbitrary domain $\Omega \subseteq \mathbb{C}$. We are familiar with the terms, definitions, useful results by previous chapters.
\section{Montel\textquoteright s Theorem}

From onwards we will give several conformations of Bloch\textquoteright s principle which states that a family of analytic (meromorphic) functions which have a common property $P$ in $\Omega$ will in general be normal family if $P$ reduces an analytic (meromorphic) function in $\mathbb{C}$ to a constant. Next, the normality of a family  of meromorphic functions is characterized by a condition in which the spherical derivative is locally bounded, due to Marty. 
\begin{thm}  \label{lemmaA}
\textbf{Marty\textquoteright s theorem} 

$ \mathcal{F} $  is a family of meromorphic functions. $\mathcal{F}$ is normal if and only $$ \{\,f^{\sharp}  = \frac{ \mid f \textquoteright  \mid}{1+ \mid f \mid ^2} : f \in \mathcal{F}\,\} $$ is locally bounded that means for each $r_0 <1$, there is number $M(r_0) < \infty$ such that
\[ f^{\sharp}(z) = \frac{ \mid f \textquoteright (z) \mid}{1+ \mid f(z) \mid ^2} < M(r_0), \]for all $\mid z \mid < r_0 $ and for all $f \in \mathcal {F}.$
\end{thm}
\begin{proof}
First suppose that $ \{\,f^{ \sharp} , f \in \mathcal{F} \,\}$ is locally bounded on $\Omega $. We have to show that $\mathcal{F}$ is normal. For this we will show that $ \mathcal{F} $ is spherically equicontinuous $\left( \mbox{Theorem } (\ref{f sph equicts})\right)$ at each point $ z_0 \in \Omega $. Let   K = $ \overline D(z_0 ,r_0) \subseteq \Omega $ be a closed disk around $ z_0 $. Since $ \{\, f^{\sharp} , f \in \mathcal{F}\,\} $ is locally bounded thus there is $M=M(r_0) >0 $ such that 
$$ f^{\sharp} (z) \leq M, \quad \mbox{forall }  z \in K. $$  

For any $z \in K $ let $\gamma $ be a straight line in $\Omega $ joining $z_0$ and $z$. Then the spherical length of $\gamma $ on Riemann sphere is given by 
\[ \int_{\gamma}f^{\sharp} (z)\, \mid dz \mid. \] 

And since chordal metric $ \chi(f(z_0) , f(z))$ gives the euclidean distance between $f(z_0) $ and $f(z)$ thus  $$\chi(f(z_0),f(z)) \leq \int_{\gamma} f^{\sharp}(\zeta) \, \mid d \zeta \mid \leq M \mid z-z_0 \mid,$$ for all $f \in \mathcal{F} $.
\\
Thus $\mathcal{F}$ is spherically equicontinuous on $\Omega$. Hence $\mathcal{F}$ is normal on $\Omega$ by Theorem (\ref{f sph equicts}).

Next let us suppose that $\mathcal{F}$ is normal on $\Omega$. On contrary let us suppose that there is an compact subset K of $\Omega$ on which $ \{\, f^{\sharp} ,f \in \mathcal{F} \,\}$ is not bounded. Therefore we can consider a sequence $\{\, f^{\sharp}_n \,\} $ which has maximum $\infty$ on K. By the normality of $\mathcal{F}$ we get that there is a subsequence $ \{ \, f_{n_k} \,\}$ converging to $f$ on K. Then there is closed disk around each point of  K on which $f$ or $ 1/f $ is analytic by Lemma(\ref{np}). If $f$ is analytic then by spherical convergence of $\{\, f_{n_k} \,\} $ we get $f_{n_k}$ are analytic for large $n_k$ thus we can apply Weierstrass theorem (\ref{weierstrass}) to get $ f^{\sharp}_{n_k} \rightarrow f^{\sharp} $ uniformly on closed disks around each point of K. Since $f^{\sharp}$ is continuous thus $f^{\sharp}_{n_k} $ are bounded on the disks in K. In case $1/f$ is analytic and by $g^{\sharp} = (1/g)^{\sharp}$ as done above we get $f_{n_k}^{\sharp}$ is bounded on the disks around points of K. And since K is compact so it can be consider as finite union of such disks, thus we get contradiction to our assumption. So $\{\, f^{\sharp} , f \in \mathcal{F} \,\}$ is uniformly bounded on each compact subsets of $\Omega$.       
\end{proof} 
\begin{lem} \label{lemmaB}
$\mathcal{F}$ be a family of holomorphic functions on $\Delta$. If there is an increasing, finite  valued function $\sum(r)$ , $0 \leq r <1$ such that
\[ M(r,f) \leq \sum(r), \quad  r_1 \leq r < 1 \]  independent of $f \in \mathcal{F}$, then $\mathcal{F}$ is normal.
\end{lem} 
\begin{proof}
Given condition implies that $\mathcal{F}$ is uniformly bounded in $\mid z \mid \leq r, \quad \mbox{for all } 0<r<1$ by maximum modulus principle which gives that $\mathcal{F }$ is locally bounded in $\Delta$. Thus by Theorem (\ref{sch montel}) $\mathcal{F}$ is normal in $\Delta$.
\end{proof} 
\begin{note}
From Nevalinna \textquoteright s inequality for proximity function we will replace $M(r,f)$ in  the sufficient condition of above Lemma (\ref{lemmaA}) by $m(r,f)$.
We can do this as follows :

Since $\mid f(re^{\iota \theta}) \mid \leq M(r,f)\leq \sum(r) $ for all $0 \leq \theta \leq 2 \pi $ by sufficient condition of above Lemma(\ref{lemmaA}). Now integrating on both side w.r.t. $\theta$ and dividing by $2\pi $ we get
\begin{align*}
 \frac{1}{2 \pi} \int_{0}^{2\pi} \log^{+} \mid f(re^{re^{\iota \theta}}) \mid \, d \theta & \leq  \log^{+} M(r,f) \\
&\leq \log^{+} \sum(r) \\
& \leq \sum(r) \qquad \mbox{ as } \log^{+}x \leq x 
 \end{align*}
 Thus $m(r,f) \leq \sum(r),\quad 0 \leq r <1 $ is sufficient condition for $\mathcal{F}$ to be normal in Lemma (\ref{lemmaA}).
 \end{note}
\begin{defn} 
We define for $\mid \alpha \mid \leq 1 $ a M$\ddot{o}$bius  transformation on $\Delta $ as follows: 
\[ \phi_ \alpha(z)= \frac{z- \alpha}{\bar{\alpha}z-1}, \quad \mbox{where } z\in \Delta \].
\end{defn}
\begin{rem}
M$\ddot{o}$bius transformation $ \phi_\alpha$ is one to one and onto transformation on $\Delta$.
\\
We can show this as :
\\
To show : $\phi_\alpha : \Delta \rightarrow \Delta $.
\\
For this we have to show that $\mid \phi_ \alpha (z) \mid <1$ or $\mid \frac{z-\alpha}{\bar{\alpha}z-1} \mid <1$ for all $z \in \Delta$. Now
\begin{equation}
\mid z- \alpha \mid \leq \mid z \mid + \mid \alpha \mid < \mid \alpha \mid + 1 \quad \mbox{as} \quad z \in \Delta 
\end{equation}
and since $\mid \frac{z-\alpha}{\bar{\alpha}z - 1} \mid <1$ is equivalent to show $$\mid z- \alpha \mid < \mid \bar{\alpha}z -1 \mid \leq \mid \bar{\alpha}z \mid +1 < \mid \alpha \mid +1, $$  as $z \in \Delta $ which is true by (4.1). Thus $\phi_\alpha$ for $\mid \alpha \mid \leq 1$ is M$\ddot{o}$bius transformation from $\Delta $ to $\Delta$.
\\
To show :  $\phi_\alpha$ for $\mid \alpha \mid \leq 1$ has inverse in $\Delta$ to $ \Delta$ so that  $\phi_\alpha$ for $\mid \alpha \mid \leq 1 $ is bijective map on $\Delta$.
\\
For this let $\phi_ \alpha (z) =w$, where $z,w \in \Delta $ or $\frac{z-\alpha}{\bar{\alpha}z-1} = w $ which implies
\\
$ z- \alpha = \bar{\alpha} zw-w$ $\Rightarrow w - \alpha = (\bar{\alpha}w -1)z$ implying
$z= \frac{w- \alpha}{\bar{\alpha}w-1}$
which again implies
$ \phi_ \alpha ^-1(w) = \frac{w- \alpha}{\bar{\alpha}w - 1}, $ forall $w \in \Delta $ and forall $\mid \alpha \mid < 1$.

Therefore $\{\, \phi_\alpha , \mid \alpha \mid < 1 \,\}$ is family of onto one to one M$\ddot{o}$bius transformation on $\Delta$.
\end{rem}
\begin{defn} 
For $f \in \mathcal{F}$ and for $\mid \alpha \mid \leq 1 $ we define 
$f_\alpha(z)= f(\phi _\alpha(z))$.
\\
We can easily verify that $$f_\alpha(0)= f(\alpha) \quad \mbox{and }  f_\alpha^{\textquoteright}(0) = f^{\textquoteright}(\alpha) \phi_\alpha ^{\textquoteright}(0).$$
\end{defn}
\begin{lem} \label{lemma1}
Suppose there is an $r_0<1$ such that corresponding to each $f \in \mathcal{F}$ there is an $\alpha= \alpha(f)$, $\mid \alpha \mid \leq r_0 $ with the property that the family $\{\, f_\alpha \,\}$ is normal. Then $\mathcal{F}$ is normal. 
\end{lem}
\begin{proof}
We can clearly say that $\{\, \phi_ \alpha \, \} $ and $\{\, \phi_\alpha^{-1} \, \}$ are uniformly bounded on $\Delta$. As for each $\alpha ; \mid \alpha \mid < 1 $ we have $\mid \phi_\alpha \mid < 1$ on $\Delta$.
\\
Similarly for $\{\,\phi_ \alpha^{-1} \,\}$. Thus $$M(r,\phi_\alpha) = max \{\, \mid \phi_ \alpha(z) \mid , \mid z \mid = r \,\} < 1 = \sum(r)$$ 
 independent of $\phi_\alpha$. Then by Lemma(\ref{lemmaB}) we get thus $\{\, \phi_ \alpha\,\}$ is normal in $\Omega$. On the same line we get $\{\, \phi_\alpha^{-1} , \mid \alpha \mid <1 \,\}$ is normal in $\Omega$. Now $\{\,f \in \mathcal{F} \,\} = \{\, f_\alpha o \phi_\alpha^{-1} , f \in \mathcal{F} \,\} $ and composition of two normal families is normal implies normality of $\mathcal{F}$.
\end{proof} 

Here is a local version of above Lemma is stated. As most of the time we required local normality of $\mathcal{F}$. 
\begin{lem} \label{lemma1`}
$\mathcal{F}$ is normal in a neighbourhood of $z_0 \in \Delta $ if there exist a sequence $\alpha_n \rightarrow z_0 $ such that $ \{\, f_{\alpha_n} \,\}$ is normal in some neighbourhood of the $z_0$. 
\end{lem}

Actually Lemma (\ref{lemma1`}) $\left(\mbox{Lemma }(4.1.4)\right)$ give necessary and sufficient condition for normality (local normality) of $\mathcal{F}$. Since for sufficient condition to be true let $\mathcal{F}$ is normal then $\{\, \phi_{\alpha}, \mid \alpha \mid \leq r_0 \, \} $ is normal then $\mathcal{F_{\alpha}}= \{\, f o \phi_{\alpha}, f \in \mathcal{F} \, \}$ is normal as composition of two normal families is normal.

Next we prove a lemma which will be used many times in our further work
\begin{lem} \label{lemma2}
Let $U(r)$ and $\gamma(r)$ be continuous non decreasing function of $r$, $r_1<r< 1$. If there is an $r_0$, $r_1<r_0<1$ and $b>1$ with
\begin{equation}\label{eq2}
U(r)<M+b \log \frac{1}{R-r}+c \log^{+} U(R) + \gamma(r), r_0<r<R<1
\end{equation}
then for $r \geq r_0$ we have $$U(r)< M_1 +4\gamma(R) + 2b \log \frac{1}{R-r},$$ $r_0 <r<R<1,$ where $M_1$ depends on $M, b$ and $c.$
\end{lem}
\begin{proof}
An argument by E. Borel (see ~\cite{hay}) yields that if $k \geq1$ for each fixed $r>r_0$ 
\begin{equation} \label{eq3}
U(\rho\textquoteright)< U(\rho)+k\log 2,\rho\textquoteright=\rho+\exp\left\lbrace-U(\rho/k)\right\rbrace,
\end{equation}
where $\rho > r$ belongs to a set of values of $\rho$ which can be enclosed in a finite or infinite number of intervals of length at most \begin{equation}\label{eq4}
2 / \exp \left\lbrace \frac{U(r)}{k} \right\rbrace.
\end{equation}
Let $ r_0 \leq r < R $ if 
\begin{equation}\label{eq5}
2 / \exp \left\lbrace \frac{U(r)}{k} \right\rbrace < R-r
\end{equation} 
then for $\rho, \quad r< \rho < R$ for which (\ref{eq3}) holds and from (\ref{eq3}) also we have
\begin{equation}\label{eq6}
\log \frac{1}{\rho\textquoteright-\rho}= \frac{1}{k} U(\rho).
\end{equation}
Let $k=2b$ so
\begin{align}
c \log^{+}U(\rho\textquoteright) &\leq c \log^{+}U(\rho)+ \log^{+} k\log 2 + \log2 \notag \\
&\leq c \log ^{+}U(\rho)+c \log^{+}k+c \log^{+} \log 2+c \log 2 \notag \\
&< c \log^{+}U(\rho)+c \log 2b+ 2c \log 2 \notag \\
&= c \log^{+}U(\rho)+c \left( \log b+ \log 2 \right) +2c \log2 \notag \\
&= c\log^{+}U(\rho)+c \left( \log b + 3 \log 2 \right)\label{eq11}
\end{align}
taking $r$ as $\rho$ and $R$ as $\rho\textquoteright$ in (\ref{eq2}) and using (\ref{eq6}) and (\ref{eq11}) we get 
\begin{align*}
U(\rho) &< M+b \log \frac{1}{\rho\textquoteright-\rho}+ c \log^{+} U(\rho\textquoteright)+\gamma(\rho) \\
&< M+\frac{1}{2}U(\rho)+c \log^{+} U(\rho) + \left( \log b +3 \log 2 \right) + \gamma(\rho)
\end{align*}
this implies 
\begin{equation}\label{eq12}
U(\rho)< 2M +2 \log^{+} U(\rho)+2c \left( \log b +3 \log 2 \right) + 2 \gamma(\rho)
\end{equation}
we have that $$\log^{+}U(\rho)< \sqrt{U(\rho)}$$ and $$U(\rho)> 16 c^2, \quad r_0< \rho <1$$ implies $$\frac{1}{U(\rho)}< \frac{1}{16c^2} $$ thus $$\frac{1}{\sqrt{U(\rho)}}< \frac{1}{4c}$$ we get then $$ \sqrt{U(\rho)}< \frac{1}{4c} U(\rho).$$ Thus unless $$U(\rho) \leq 16 c^2, \quad r_0<\rho<1$$ we have $$\log^{+}U(\rho)< \sqrt{U(\rho)}< \frac{1}{4c} U(\rho).$$
In (\ref{eq12}) using both cases simultaneously we get
\begin{align*}
U(\rho) < 16c^2+4M+4c \left( \log b +3 \log 2\right) +4 \gamma(\rho) 
\end{align*} 
since $r< \rho<R$ this implies $U(r)<U(\rho)$ and $\gamma(\rho)<\gamma(R)$ as $U$ and $\gamma$ increases so we an find an absolute constant $A$ with
\begin{equation}\label{eq13}
U(r)<A+4\gamma(R).
\end{equation}All that we have obtained is due to equation (\ref{eq5}) as if (\ref{eq5}) does not hold then
 \begin{equation}\label{eq14}
 U(r)< 2b \log \frac{2}{R-r}
\end{equation}
instead of (\ref{eq13}). Considering all possiblities that led (\ref{eq13}) and (\ref{eq14}) together we get \begin{align}
U(r)&< 16c^2+4M+4c\left\lbrace\log b+ 3 \log 2 \right\rbrace +4 \gamma(R)+2b \log \frac{1}{R-r} \notag \\
&=M_1+4\gamma(R)+2b \log \frac{1}{R-r},\label{eq16}
\end{align} 
where $r_0<r<R<1, \quad M_1=16c^2+4M+4c\left\lbrace\log b+ 3 \log 2 \right\rbrace.$
\end{proof}
\begin{rem}  Consider $R= \frac{r+1}{2}$ in above lemma then we get conclusion of lemma independent of $R$.
\end{rem}
\begin{rem} We can easily replace $4 \gamma(R)$ in conclusion of above lemma by 
\\
$k\gamma(R), \quad k>1$ by suitably changing constant $M_1$ and coefficient of $\log \frac{1}{R-r}$ in (\ref{eq16}).
\end{rem}
\begin{thm}\label{montel}
\textbf{Montel\textquoteright s Theorem}
\\
$\mathcal{F}$ is family of meromorphic functions in $\Delta $. If every $f \in \mathcal{F}$ omits the value 0,1 and $\infty$ then $\mathcal{F}$ is normal.
\end{thm}
Note here that Montel\textquoteright s theorem give an illustration to Bloch\textquoteright s principle.
\\
\textbf{Deduction.}

Here we deduce some inequalities for a meromorphic function from some well known theorems and inequalities :   
consider an arbitrary meromorphic function $f$ in $\Delta$ choose $ r_0 < 1 \quad \mbox{and} \quad \mid \alpha \mid < r_0$ subject to the restriction $f(\alpha) \neq 0, \infty$ and $ f^{\textquoteright} (\alpha) \neq 0 $.
\\
Nevanlinna fundamental inequality $\left( \mbox{Theorem } (\ref{fft}) \right)$ applying for $f_\alpha $, taking $q=2$, $a_1 =0 $ and $a_2 = 1 $ gives that
\begin{equation}\label{eq17}
m(r,f_\alpha)+m\left(r,\frac{1}{f_\alpha}\right)+m\left(r,\frac{1}{f_\alpha - 1}\right) \leq 2 T(r,f_\alpha) - N_1(r,f_\alpha) +S(r,f_\alpha)
\end{equation}
where $N_1(r,f_\alpha) = N\left(r,\frac{1}{f \textquoteright_{\alpha}}\right)+2N(r,f_\alpha)-N(r,f\textquoteright_{\alpha})$
\\
and $S(r,f_\alpha)= 2m\left(r,\frac{f\textquoteright_{\alpha}}{f_\alpha}\right)+m\left(r,\frac{f\textquoteright_{\alpha}}{f_\alpha-1}\right)+\log \big | \frac{1}{f\textquoteright_{\alpha}(0)} \big | + C. $
\\
Here C = absolute constant.
\\
Nevanlinna\textquoteright s first fundamental theorem $\left(\mbox{Theorem } (\ref{sft}) \right)$ for $f_\alpha $ gives that if $ f(\alpha) \neq 0,1 $ 
$$ m\left(r,\frac{1}{f_{\alpha}-1}\right)+ N\left(r,\frac{1}{f_{\alpha}-1}\right)= T(r,f_\alpha)-\log \mid f_\alpha(0)-1 \mid + \epsilon(a,R),$$ $a=1$ here
and $ \mid \epsilon(a,R) \mid < \log 2 $ since here  $\log^{+} \mid a \mid = 0 $ .
\begin{align*}
 T\left(r,\frac{1}{f_\alpha}\right) + \log \mid f_\alpha (0) \mid &= T(r,f_\alpha) \\
 &=m\left(r,\frac{1}{f_\alpha - 1}\right)+ N\left(r,\frac{1}{f_\alpha - 1 }\right) + \log \mid f_\alpha (0) -1 \mid + \epsilon^{*}. 
 \end{align*}
 \begin{align}
 m\left(r,\frac{1}{f_\alpha}\right)+N\left(r,\frac{1}{f_\alpha }\right)+\log \mid f_\alpha(0) \mid &= m\left(r,\frac{1}{f_\alpha -1}\right) \notag \\
 &+ N\left(r,\frac{1}{f_\alpha -1}\right) +\log \mid f_\alpha(0)-1 \mid +\epsilon^{*}. \label{eq19}
\end{align}
Now adding $N\left(r,\frac{1}{f_\alpha}\right)+N(r,f_\alpha) +N\left(r,\frac{1}{f_\alpha -1 }\right) + \log \mid f_\alpha(0) \lbrace f_\alpha(0)-1 \rbrace \mid $ on both side of (\ref{eq17}) we get 
\begin{align}
 m(r,f_\alpha) &+m\left(r, \frac{1}{f_\alpha}\right)+ m\left(r,\frac{1}{f_\alpha-1}\right) +N\left(r,\frac{1}{f_\alpha}\right) +N(r,f_\alpha) +N\left(r,\frac{1}{f_\alpha -1}\right)\notag  \\
 &+\log \mid f_\alpha(0)\lbrace f_\alpha(0) -1 \rbrace \mid \leq 2 T(r,f_\alpha) - N_1 (r,f_\alpha) + S(r,f_\alpha) +N\left(r,\frac{1}{f_\alpha}\right) \notag \\
 &+N(r,f_\alpha)+N\left(r,\frac{1}{f_\alpha -1 }\right)+ \log \mid f_\alpha(0)\lbrace f_\alpha (0) -1 \rbrace \mid. \label{eq22}
\end{align}
 Consider L.H.S. of (\ref{eq22}) 
\begin{align*}
m(r,f_\alpha)&+m\left(r,\frac{1}{f_\alpha}\right)+m\left(r,\frac{1}{f_\alpha - 1}\right) +N\left(r,\frac{1}{f_\alpha}\right) +N(r,f_\alpha)    \\
&+N\left(r,\frac{1}{f_\alpha-1}\right)+ \log \mid f_\alpha(0) \lbrace f_\alpha (0) -1 \rbrace  \mid \\
&= T(r,f_\alpha)+T\left(r,\frac{1}{f_\alpha}\right) + T\left(r,\frac{1}{f_\alpha - 1}\right) +\log \mid f_\alpha(0) \mid + \log \mid f_\alpha(0) -1 \mid \\
&= T(r,f_\alpha)+ \left[T\left(r,\frac{1}{f_\alpha}\right)+\log \mid f_\alpha(0) \mid \right] + \left[ T\left(r,\frac{1}{f_\alpha -1}\right) + \log \mid f_\alpha(0) -1 \mid \right] \\
&=T(r,f_\alpha)+ T(r,f_\alpha) + T(r,f_\alpha) - \epsilon ^{*}. 
\end{align*}  
Now consider R.H.S. of (\ref{eq22})
\begin{align*}
2T(r,f_\alpha)&-N\left(r,\frac{1}{f_\alpha}\right)-2N(r,f_\alpha) + N(r,f\textquoteright _\alpha) +2m\left(r,\frac{f\textquoteright _\alpha}{f_\alpha}\right) \\
&+m\left(r,\frac{f\textquoteright _\alpha}{f_\alpha -1}\right) +\log \big | \frac{1}{f\textquoteright _\alpha(0)} \big |  + C+N(r,\frac{1}{f_\alpha}) \\
&+N(r,\frac{1}{f_\alpha-1}) + N(r,f_\alpha) +\log \mid f_\alpha(0) \lbrace f_\alpha(0) -1 \rbrace \mid \\
&= 2T(r,f_\alpha) + \bar{N}(r,f_\alpha)+N\left(r,\frac{1}{f_\alpha}\right)+N\left(r,\frac{1}{f_\alpha-1}\right) -N\left(r,\frac{1}{f \textquoteright _\alpha}\right) \\
& + 2m\left(r,\frac{f\textquoteright _\alpha}{f_\alpha}\right)+m\left(r,\frac{f\textquoteright _\alpha}{f_\alpha-1}\right) + \log \big | \frac{f(\alpha) \lbrace f_\alpha(0)-1 \rbrace} {f\textquoteright_\alpha(0)} \big | + C,
\end{align*}
where $\bar{N}(r,f_\alpha) =N(r,f\textquoteright_\alpha)-N(r,f_\alpha) $ thus (\ref{eq22}) becomes
 \begin{align}
T(r,f_\alpha) &< \bar{N}(r,f_\alpha) +N\left(r,\frac{1}{f_\alpha}\right)+N\left(r,\frac{1}{f_\alpha-1}\right) - N\left(r,\frac{1}{f\textquoteright _\alpha}\right) \\
&+ 2m\left(r,\frac{f\textquoteright _\alpha}{f_\alpha}\right) + m\left(r,\frac{f\textquoteright _\alpha}{f_\alpha -1}\right) + \log \big | \frac{f(\alpha) \lbrace f(\alpha) -1 \rbrace}{f\textquoteright _\alpha(0)} \big | +C. \label{eq24}
\end{align}
\begin{proof}
\textbf{Proof of Montel\textquoteright s Theorem.}

As all $f \in \mathcal{F} $ omits 0, 1,$\infty$ that means $f(z) \neq 0,1,\infty$ forall $z \in \Delta $ therefore 
\\
$f_\alpha(z)= f(\phi_\alpha(z)) \neq 0,1,\infty \quad \forall z \in \Delta $ implies $ N(r,f_\alpha)=N\left(r,\frac{1}{f_\alpha}\right)=N\left(r,\frac{1}{f_\alpha-1}\right) = 0 $.
Thus (\ref{eq24}) gives 
\begin{equation}\label{eq25}
m(r,f_\alpha) < 2m\left(r,\frac{f\textquoteright _\alpha}{f_\alpha}\right)+ m\left(r,\frac{f\textquoteright _\alpha}{f_\alpha-1}\right) + \log \big | \frac{f(\alpha)\lbrace f(\alpha)-1}{f\textquoteright _\alpha (0)} \big | +C. 
\end{equation}
Next we use 
\begin{lem} \label{lemmaC}
$g$ is meromorphic in $\mid z \mid < 1$ and $\delta < r < R<1 $ and $g(0) \neq 0, \infty $ then
$$m\left(r,\frac{g\textquoteright}{g}\right) < 4 \log^{+}T(R,g)+ 6 \log \frac{1}{R-r}+4 \log^{+} \log^{+} \frac{1}{\mid g(0) \mid} +C_1, $$
where $C_1 = C_1(\delta) $ if $\delta \textquoteright > \delta $ we may take $C_1(\delta\textquoteright ) = C_1(\delta)$.
\\
This Lemma is nothing but the Nevanlinna \textquoteright s Estimate.
\end{lem}
Using Lemma (\ref{lemmaC}) for $ g=f_\alpha $ and $g = f_\alpha-1 $ also $ T(R,f_\alpha)=m(R,f_\alpha)$ and $T(r,f_\alpha-1) =m(r,f_\alpha - 1)$ since $ f_\alpha \neq 0, \infty $
$$m\left(r,\frac{f\textquoteright _\alpha}{f\alpha}\right) < 4 \log ^{+} T(R,f_\alpha) + 6 \log \frac{1}{R-r}+ 4 \log^{+} \log^{+} \frac{1}{\mid f_\alpha(0) \mid }+C_1$$ 
$$m\left(r,\frac{f\textquoteright _\alpha}{f_\alpha-1}\right)< 4 \log^{+}T(R,f_\alpha-1)+6 \log \frac{1}{R-r}+4 \log^{+} \log^{+} \frac{1}{\mid f_\alpha(0)-1 \mid}+C_1$$
using in (\ref{eq25})
\begin{align}
m(r,f_\alpha)=T(r,f_\alpha)&< 8 \log^{+}T(R,f_\alpha)+12 \log \frac{1}{f_\alpha}+8 \log^{+} \log^{+} \frac{1}{\mid f_\alpha(0) \mid } \notag \\
&+ 4 \log^{+}T(R,f_\alpha-1)+6 \log \frac{1}{R-r} + 4 \log^{+} \log^{+} \big | \frac{1}{f_\alpha(0)-1} \big | \notag \\
&+ \log \big | \frac{f(\alpha) \lbrace f(\alpha)-1 \rbrace 	}{f \textquoteright _\alpha (0)} \big | +C \notag \\
&< 12 \log^{+}T(R,f_\alpha)+18 \log \frac{1}{R-r} +8 \log^{+} \log^{+} \frac{1}{\mid f_\alpha(0) \mid}+\log \mid f(\alpha) \mid \notag \\
&+\log \mid f(\alpha)-1 \mid +4 \log^{+} log^{+} \big | \frac{1}{f(\alpha)-1} \big | + \log \big | \frac{1}{f \textquoteright _\alpha (0)} \big | +C, \label{eq30}
\end{align}
for $\frac{1}{2} < r < R < 1$. Since 
\begin{align*}
\log \mid f_\alpha(0) \lbrace f_\alpha(0)-1 \rbrace \mid &= \log \mid f_\alpha (0) \mid + \log \mid f_\alpha (0) -1 \mid \\
&=\log \mid f(\alpha) \mid + \log \mid f(\alpha)-1 \mid.
\end{align*} 
Also 
$$T(R,f_\alpha-1) \leq T(R,f_\alpha) + T(R,g) +\log 2,$$ where $g \equiv 1 $ on $\Delta$ and clearly $T(R,1) = 0 $.
\\
\textbf{To show}
\\
If $A> e $ and $u \in \mathbb{C}$ then 
$$\log \mid u \mid + A \log^{+} \log^{+} \mid \frac{1}{u} \mid \leq \log^{+} \mid u \mid + A \log A .$$
For this let $x= \mid u \mid >0 $ and $A>e$. Let $\phi (x)= x-A \log x $
$$\phi \textquoteright (x)=1-\frac{A}{x} =0 \Rightarrow x=A.$$
$$\phi {\textquoteright} {\textquoteright} (x) \mid _ {x=A} = \frac{A}{A^2} = \frac{1}{A} > 0.$$
Which gives $x=A$ is point of minima for $\phi (x)$
$\Rightarrow A-A \log A \leq x-A \log x, $ for all $x>0$. Now let $x=\log^{+} 1/x$ implies $$ A(1-\log A) \leq \log^{+} 	1/x- A \log^{+} \log^{+} 1/x$$
$$\Rightarrow A(\log A -1) + \log^{+}1/x \geq A \log^{+} \log^{+} 1/x.$$
Now since $\log x = \log^{+} x - \log ^{+} 1/x$ so
$\log^{+} x - \log x = \log ^{+} 1/x$ which gives 
\begin{align*}
 A \log^{+} \log^{+} \frac{1}{x} + \log x &\leq  \log^{+} x + A \log A - A \\ 
& \leq \log^{+} x + A \log A.
\end{align*}
( because $A>e$ therefore $ A $ is positive) therefore  we get the required result.
\\
Using in (\ref{eq30})
\begin{align}
T(r,f_\alpha)&< 12 \log^{+}T(R,f_\alpha) + 18 \log \frac{1}{R-r} + \log^{+} \mid f(\alpha) \mid + \log^{+} \mid f(\alpha)-1 \mid \notag \\
&+ \log \mid \frac{1}{f \textquoteright _ \alpha (0)} \mid + C, \label{eq32}
\end{align} for $\frac{1}{2} < r < R < 1 $.
\\
If $\mid \alpha \mid < r_0$ 
$$\phi_\alpha (z) = \frac{z-\alpha}{\bar{\alpha}z-1}$$
$$\phi \textquoteright_\alpha(z)=\frac{(\bar{\alpha}z-1)-(z-\alpha)\bar{\alpha}}{(\bar{\alpha}z-1)^2}= \frac{\mid \alpha \mid^2 -1}{(\bar{\alpha}z-1)^2}$$
$$\phi \textquoteright _\alpha(0) = \mid \alpha\mid ^2 - 1 < 0$$ implying
$$ \mid \frac{1}{\phi \textquoteright_\alpha(0)} \mid < S, $$ where $S= S(r_0)$ and $f_\alpha(z)= f(\phi _\alpha(z));$ $ f \textquoteright _\alpha(z)= f \textquoteright(\phi_\alpha(z)) \phi \textquoteright _\alpha(z)$ which gives
$$\big | \frac{1}{f\textquoteright_\alpha(0)} \big |= \big | \frac{1}{f \textquoteright (\alpha) \phi \textquoteright_\alpha(0)} \big | < \frac{1}{\mid f \textquoteright (\alpha) \mid}S$$ 
$$\Rightarrow \log \mid \frac{1}{f \textquoteright_\alpha(0)} \mid < \log \mid \frac{1}{f \textquoteright (\alpha)} \mid + \log S $$ because  $\log $ is increasing function. Thus we get in (\ref{eq32})
\begin{align}
T(r,f_\alpha)&< 12 \log^{+} T(R,f_\alpha) +18 \log \frac{1}{R-r}+ \log^{+} \mid f(\alpha) \mid\notag  \\
& + \log^{+} \mid f(\alpha)-1 \mid + \log \mid \frac{1}{f \textquoteright(\alpha)} \mid + \log S +C. \label{eq34}
\end{align}
Also 
\\
$\mid f(\alpha)-1 \mid \leq \mid f(\alpha) \mid +1$ since $\log^{+}$ is an increasing function thus we get \begin{align*}
\log^{+} \mid f(\alpha)-1 \mid &\leq \log^{+} \left[\mid f(\alpha) \mid +1 \right] \notag \\
&< \log^{+} \mid f(\alpha) \mid +\log^{+} 1 + \log 2.
\end{align*} 
Now (\ref{eq34}) becomes 
\begin{align}
T(r,f_\alpha)&< 12 \log^{+} T(R,f_\alpha) +18 \log \frac{1}{R-r} +2 \log^{+} \mid f(\alpha)\mid \\ 
& + \log \big | \frac{1}{f \textquoteright (\alpha)} \big |+C_1+ S_1,\label{eq36}
\end{align} here $C_1=C+ \log2 , S_1 =\log S $ for $\frac{1}{2}<r<R<1 , \mid \alpha \mid < r_0 < 1$ and $S_1=S_1(r_0)$.
\\
We have two cases
\\
\textbf{Case(i)}
Let $\mathcal{F}_1$ be an infinite denumerable subcollection of the family $\mathcal{F}$. Now we have inequality (\ref{eq36}) for each $f \in \mathcal{F}_1$.
\\
Let there are constants $M< \infty$, $r_0 < 1$ and an infinite subsets $\mathcal{F}_2$ of $\mathcal{F_1}$ with the property that if $f \in \mathcal{F}_2$ there is an $\alpha = \alpha(f) , \mid \alpha \mid < r_0$ with 
\begin{equation}\label{eq37}
2 \log^{+} \mid f(\alpha) \mid + \log \mid \frac{1}{f \textquoteright(\alpha)} \mid < M. 
\end{equation}
%\hspace*{\fill}
Using (\ref{eq37}) in (\ref{eq36}) we have 
$$T(r,f_\alpha)< M +18 \log \frac{1}{R-r} + 12\log^{+}T(R,f_\alpha)+ \gamma (r),$$ 
for $\frac{1}{2}<r<R<1 , \mid \alpha \mid < r_0 < 1 $ here $\gamma (r) =0, $ for each $r<1$.
Then by Lemma (\ref{lemma2}), for $T(r,f_\alpha)=U(r) , \gamma(r) =0 $ for each $r<1$ also $U(r)=T(r,f_\alpha)=m(r,f_\alpha)$ is independent of $f \in \mathcal{F}_2$ therefore above is true for all $f \in \mathcal{F}_2$. This gives that 
$$T(r,f_\alpha)=m(r,f_\alpha)< M_1+4 \gamma (R)+2 .18 \log \frac{1}{R-r},$$ $ \frac{1}{2}<r<R<1,$ for all $f \in \mathcal{F}_2 \subset \mathcal{F}_1 \subseteq \mathcal{F}$ this implies $$ m(r,f_\alpha) < M_1 +36 \log \frac{1}{R-r} = \sum(r).$$
Thus by Lemma (\ref{lemma2}), $\lbrace  f_\alpha , f \in \mathcal{F} \rbrace $ is normal and by Lemma (\ref{lemma1}) $\mathcal{F}$ is normal. 
\\
\textbf{Case(ii)}
If condition (\ref{eq37}) does not hold then for given $M< \infty , r_0 < 1$  
$$2 \log^{+} \mid f(\alpha) \mid + \log \mid \frac{1}{f \textquoteright (\alpha)} \mid > M , \mid \alpha \mid \leq r_0, $$
for all but a finite number of $f  \in \mathcal{F}_1$. Consider 
\begin{align*}
\mid 2 \log^{+} \mid f(\alpha) \mid &+ \log \big| \frac{1}{f \textquoteright (\alpha)} \mid - \log \left\lbrace \frac{1+ \mid f(\alpha) \mid^2}{\mid f \textquoteright (\alpha) \mid} \right\rbrace \big| \\
& = \mid 2 \log^{+} \mid f(\alpha) \mid - \log (1+ \mid f(\alpha) \mid ^2  \\
&= \big| \log^{+} \frac{\mid f( \alpha) \mid ^2}{1+\mid f(\alpha \mid ^2)} \big| = 0 
\end{align*}
because  $$\frac{\mid f( \alpha) \mid ^2}{1+\mid f(\alpha) \mid ^2}<1$$ implying $$ -\log2 \leq 2 \log^{+} \mid f(\alpha) \mid + \log \mid \frac{1}{f \textquoteright (\alpha)} \mid - \log \left\lbrace \frac{1+ \mid f(\alpha) \mid ^2}{\mid f \textquoteright (\alpha ) \mid } \right\rbrace \leq \log 2$$
which again implies
 $$ -\log 2 +A \leq 2 \log^{+} \mid f(\alpha) \mid + \log \mid \frac{1}{f \textquoteright (\alpha)} \mid \leq \log 2 + A, $$
 where $A=\log \lbrace \frac{1+ \mid f(\alpha) \mid ^2}{\mid f \textquoteright (\alpha ) \mid } \rbrace$. In (\ref{eq36})
\begin{align*}
 0&< 2\log^{+} \mid f(\alpha) \mid + \log \mid \frac{1}{f \textquoteright (\alpha)} \mid \\
&< \log2 +A \\
&= A+ C {\textquoteright}{\textquoteright}
\end{align*}
$$\Rightarrow 0< \log \left\lbrace \frac{1+ \mid f(\alpha) \mid ^2}{\mid f \textquoteright (\alpha ) \mid } \right\rbrace $$
which gives
$$\log1=0 > \log \frac{\mid f \textquoteright (\alpha) \mid}{1+ \mid f(\alpha) \mid^2}$$ 
this gives $$\frac{\mid f \textquoteright (\alpha) \mid}{1+\mid f(\alpha) \mid ^2} < 1,$$ for all $\alpha ; \mid \alpha \mid < r_0<1$. As $f$ is arbitrary here so $$\frac{\mid f \textquoteright (z) \mid}{1+\mid f(z) \mid ^2} < 1, $$ for $\mid z \mid \leq r_0<1$, independent of $f \in \mathcal{F}$. Then by Lemma (\ref{lemmaA}) we have that $\mathcal{F}$ is normal.
\end{proof}
\begin{thm} 
\textbf{Generalization of Montel\textquoteright s theorem }

$\mathcal{F} $ be a family of meromorphic functions in $\Delta$ and suppose all poles of multiplicities $\geq h$, all zeros of multiplicities $\geq k$ and all zeros of $f(z)-1$  of multiplicities $\geq l $ with $$\frac{1}{h}+ \frac{1}{k}+\frac{1}{l} =\mu < 1.$$ 
\\
Then $\mathcal{F}$ is normal.
\end{thm}
\begin{proof}
In inequalities (\ref{eq24}) we proceed as follows 
\begin{align}
T(r,f_\alpha)&< \bar{N}(r,f_\alpha)+N\left(r,\frac{1}{f_\alpha}\right)+N\left(r,\frac{1}{f_\alpha-1}\right)-N\left(r,\frac{1}{f \textquoteright _\alpha}\right) +2m\left(r,\frac{f\textquoteright_\alpha}{f_\alpha}\right) \notag \\
&+m\left(r,\frac{f\textquoteright_\alpha}{f_\alpha-1}\right) +\log \big | \frac{f(\alpha) \lbrace f(\alpha)-1 \rbrace}{f\textquoteright_\alpha(0)} \big |+C,\label{eq39}
\end{align}
where $f(\alpha) \neq 0,1,\infty$ and $f \textquoteright(\alpha) \neq 0 $. We know
\\
$\bar{N}(r,f_\alpha)= N(r,f \textquoteright_\alpha)-N(r,f_\alpha)$ and $\bar{N}\left(r,\frac{1}{f_\alpha}\right)=N\left(r,\frac{1}{f_\alpha}\right)-N\left(r,\frac{1}{f\textquoteright_\alpha}\right).$
\\
Similarly $\bar{N}\left(r,\frac{1}{f_\alpha-1}\right)=N\left(r,\frac{1}{f_\alpha-1}\right)-N \left(r,\frac{1}{f\textquoteright_\alpha}\right)$ since $(f_\alpha-1) \textquoteright = f\textquoteright_{\alpha}$.
\\
So in (\ref{eq39}) we have
\begin{align*}
T(r,f_\alpha)&<\bar{N}(r,f_\alpha)+N \left(r,\frac{1}{f_\alpha}\right)-N \left(r,\frac{1}{f\textquoteright_\alpha}\right)+N \left(r,\frac{1}{f_\alpha-1}\right) \\
&-N \left(r,\frac{1}{f\textquoteright_\alpha}\right)+N \left(r,\frac{1}{f\textquoteright_\alpha}\right) +2m \left(r,\frac{f\textquoteright_\alpha}{f_\alpha}\right)+m \left(r,\frac{f\textquoteright_\alpha}{f_\alpha-1}\right) \\
&+ \log \big | \frac{f(\alpha) \lbrace f(\alpha)-1 \rbrace}{f\textquoteright_\alpha (0)} \big | +C \\
&= \bar{N}(r,f_\alpha)+\bar{N}\left(r,\frac{1}{f_\alpha}\right)+\bar{N} \left(r,\frac{1}{f_\alpha-1}\right)+2m \left(r,\frac{f \textquoteright_\alpha}{f_\alpha}\right) \\
&+m \left(r,\frac{f\textquoteright_\alpha}{f_\alpha-1}\right) +\log \big | \frac{f(\alpha) \lbrace f(\alpha)-1 \rbrace}{f\textquoteright_\alpha(0)} \big | +C.
\end{align*}
Now we can see that $$\bar{N}(r,f_\alpha) \leq \frac{1}{h} N(r,f_\alpha) \leq \frac{1}{h}T(r,f_\alpha )$$
$$\bar{N} \left(r,\frac{1}{f_\alpha} \right) \leq \frac{1}{k} N \left(r,\frac{1}{f_\alpha} \right) \leq T \left(r,\frac{1}{f_\alpha} \right) =\frac{1}{k} \left\lbrace T(r,f_\alpha) - \log \mid f_\alpha(0) \mid \right\rbrace $$ 
\begin{align*}
\bar{N}\left(r,\frac{1}{f_\alpha-1}\right) &\leq \frac{1}{l} N \left(r,\frac{1}{f_\alpha-1}\right) \leq \frac{1}{l} T\left(r,\frac{1}{f_\alpha-1}\right) \\
&= \frac{1}{l} \left\lbrace T(r,f_\alpha-1) - \log \mid f(\alpha)-1 \mid \right\rbrace \\
& \leq \frac{1}{l} \left\lbrace T(r,f_\alpha)+T(r,-1) +\log 2 -\log \mid f(\alpha) -1 \mid \right\rbrace .
\end{align*}
Because $$T(r,f_\alpha-1)=T\left(r,\frac{1}{f_\alpha-1}\right) +\log \mid f(\alpha)-1\mid $$
and $$T(r,f_\alpha-1) \leq T(r,f_\alpha)+ T(r,-1)+ \log 2 $$ 
\\
$N(r,-1)=0 $ because no poles for constant functions and\begin{align*}
m(r,-1) &= \frac{1}{2\pi} \int_{0}^{2\pi} \log^{+} \mid g(re^{\iota \theta}) \mid \, d\theta \\
& =\frac{1}{2\pi} \int_{0}^{2\pi} \log^{+} \mid -1 \mid \, d\theta =0
\end{align*} where $g \equiv -1 $ on $\Delta$. This implies $$T(r,-1) =0.$$ Using this we have 
\begin{align*}
T(r,f_\alpha)&< \mu T(r,f_\alpha) +2m\left(r,\frac{f \textquoteright_\alpha}{f_\alpha}\right)+m\left(r,\frac{f \textquoteright_\alpha}{f_\alpha-1}\right)+ \left( 1-\frac{1}{k} \right) \log \mid f(\alpha) \mid \\
&+ \left( 1-\frac{1}{l} \right) \log \mid f(\alpha)-1 \mid + \log \big | \frac{1}{f \textquoteright_\alpha(0)} \big | +C
\end{align*}
using Lemma (\ref{lemmaC}) (Nevanlinna\textquoteright s  Estimate)
\begin{align*}
 \left( 1-\mu \right) T(r,f_\alpha)& < 2 \left[ 4 \log^{+} \log^{+} \big | \frac{1}{f(\alpha)} \big | +4 \log^{+} T(R,f_\alpha) +6 \log \frac{1}{R-r} \right]  \\
 &+ 4 \log^{+}T(R,f_\alpha-1)+ 6 \log \frac{1}{R-r} + 4 \log^{+} \log^{+} \big | \frac{1}{f(\alpha)-1} \big | \\
 &+ \left( 1-\frac{1}{k} \right) \log \mid f(\alpha) \mid + \left( 1-\frac{1}{l} \right) \log \mid f(\alpha)-1 \mid \\
 &+ \log \big | \frac{1}{f \textquoteright_\alpha(0)} \big | + C  \\
 &< 12 \log^{+} T(R,f_\alpha) +18 \log \frac{1}{R-r} +8 \log^{+} \log^{+} \big | \frac{1}{f(\alpha)} \big | \\
 &+ \left( 1-\frac{1}{k} \right) \log \mid f(\alpha) \mid +4 \log^{+} \log^{+} \big |\frac{1}{f(\alpha)-1} \big | \\
 &+ \left( 1-\frac{1}{l} \right) \log \mid f(\alpha)-1 \mid + \log \big | \frac{1}{f \textquoteright_\alpha(0)} \big | +C \\
& \leq  12 \log^{+} T(R,f_\alpha) +18 \log \frac{1}{R-r} +8 \log^{+} \log^{+} \big | \frac{1}{f(\alpha)} \big | \\
&+ \left( 1-\frac{1}{k} \right) \log^{+} \mid f(\alpha)\mid  +4 \log^{+} \log^{+} \mid \frac{1}{f(\alpha)-1} \mid \\
&+ \left( 1-\frac{1}{l} \right) \log^{+} \big | f(\alpha) -1 \big |+ \log \big | \frac{1}{f \textquoteright_\alpha(0)} \big |+C
\end{align*} since $\log x \leq \log^{+} x$.
\\
Now as $$\mid f(\alpha) -1 \mid \leq \mid f(\alpha) \mid +1 $$
so
\begin{align*}
\log^{+} \mid f(\alpha)-1 \mid &\leq \log^{+} \left( \mid f(\alpha) \mid + 1 \right) \\
&\leq \log^{+} \mid f(\alpha) \mid + \log^{+}1+\log 2 
\end{align*}
and $$f \textquoteright_\alpha(0)= f \textquoteright(\alpha) \phi \textquoteright_\alpha(0)$$
\\
$$\Rightarrow \log \mid \frac{1}{f \textquoteright_\alpha(0)} \mid= \log \mid \frac{1}{f \textquoteright(\alpha)} \mid + \log \mid \frac{1}{\phi \textquoteright_\alpha(0)} \mid.$$
\\
Thus we get
\begin{align}
\left( 1-\mu \right)T(r,f_\alpha)& < 12 \log^{+} T(R,f_\alpha) +18\log \frac{1}{R-r} + \left( 1- \frac{1}{k} \right) \log^{+} \mid f(\alpha) \mid \notag \\
&+ \left( 1-\frac{1}{l} \right)\log^{+} \mid f(\alpha) \mid  + \log \frac{1}{\mid f \textquoteright(\alpha) \mid} +C+S \notag \\
&= 12 \log^{+} T(R,f_\alpha) +18 \log \frac{1}{R-r} + \left( 2-\frac{1}{k} -\frac{1}{l} \right) \log^{+} \mid f(\alpha) \mid \notag \\
&+ \log \mid \frac{1}{f \textquoteright(\alpha)} \mid +C+S,\label{eq43}
\end{align} where $S=S(r_0) \mid \alpha \mid <r_0 , r_0 <r<R<1$.
\\
To show that $\mathcal{F}$ is normal it will be shown that $\mathcal{F}$ is normal at each point $z_0 \in \Delta$. In particular take $z_0 = 0 \in \Delta$. Let $\mathcal{F}_1$ be a denumerable infinite subcollection of $\mathcal{F}$. Now two cases arises:
\\
\textbf{Case(i)}
\\
There exists functions $f_n \in \mathcal{F}_1 , n=1,2,3,\ldots $ and $\alpha_n = \alpha_n(f_n) \neq 0,$
\\
$ \alpha_n \rightarrow 0 , f_n(\alpha_n) \neq 0,1,\infty;$ $f_n \textquoteright(\alpha_n)\neq 0 $ and $ M< \infty$ such that
$$\left(2-\frac{1}{k} - \frac{1}{l}\right) \log^{+} \mid f_n(\alpha_n) \mid + \log \big | \frac{1}{f_n \textquoteright(\alpha_n)} \big |< M. $$ 
Thus (\ref{eq43}) becomes
$$(1- \mu)T(r,f_{n,\alpha_n})< 12 \log^{+}T(R,f_{n, \alpha_n})+ 18 \log \frac{1}{R-r} + M_1,$$ $ r_0< r<R<1$ 
\begin{align*}
T(r,f_{n,\alpha_n}) &< \frac{12}{1- \mu} \log^{+}T(R,f_{n,\alpha_n} )+ \frac{18}{1- \mu} \log \frac{1}{R-r} +M_2 \\
&=M_2 + b \log \frac{1}{R-r}+ c \log^{+} T(R,f_{n,\alpha_n}),
\end{align*}
$r_0<r<R<1 , \mid \alpha_n \mid <r_0$.
\\
Let $U(r)= T(r,f_{n,\alpha_n}) ; \gamma(r) \equiv 0 $ for each $r<1$
\\
then by Lemma (\ref{lemma2})
\begin{equation} \label{eq44}
T(r,f_{n,\alpha_n})< M_3 +2b \log \frac{1}{R-r} = \sum(r),\quad r_0 <r<R<1
\end{equation}

Next to show that origin can not be a limit point of poles. For this let if $x_{n_k}$ be poles of $f_{n_k}$ and $x_{n_k} \rightarrow 0 $ for some subsequence $n_k$; then $ x_{n_k}- \alpha_{n_k} \rightarrow 0 $ (write $ x_{n_k} , n_k \quad \mbox{and} \quad f_{n_k}$ as $x,n$ and $f$).
\\
For $r> r_0$ $\left( \mbox{from }
(\ref{eq44})\right)$
\begin{align*}
\sum(r) &> N(r,f) = \int_{0}^{r} \frac{n(t,f)-n(0,f)}{t} \, dt + n(0,f) \log r \\
&> \int_{0}^{r} \frac{n(t,f) - n(0,f)}{t} \, dt > \int_{x-\alpha}^{r} \frac{n(t,f) - n(0,f)}{t} \, dt \\
& > \int_{x-\alpha}^{r} \frac{1}{t} \, dt = \log \frac{r}{x-\alpha}.
\end{align*}
Now
$$\log \frac{ x- \alpha}{r} \leq \log^{+} \frac{x-\alpha}{r} \leq \log ^{+} (x-\alpha) + \log ^{+} \frac{1}{r} + \log 2 $$ 
so
$$\log \frac{r}{x-\alpha} \geq  - \log ^{+}( x-\alpha)- \log^{+} \frac{1}{r} - \log 2 $$
thus
$$\log \frac{r}{x-\alpha} \geq  \log ^{+} \frac{1}{x -\alpha} + \log ^{+} r - \log 2= \log^{+} \frac{1}{x-\alpha} -\log 2,$$ $ r<1$ therefore $\log^{+} r = 0 $ so $$\sum(r) > \log \frac{r}{x- \alpha} \geq \log^{+} \frac{1}{x-\alpha} - \log 2 = \log \frac{1}{x-\alpha} - \log 2$$
but $x-\alpha \rightarrow 0 \Rightarrow \frac{1}{x-\alpha} \rightarrow \infty $ this implies $\log \frac{1}{x-\alpha}$ gives  large enough value. So this gives a contradiction to (\ref{eq44}). Thus there is a neighbourhood $\{\,\ z ; \mid z \mid < \delta \,\}$ of origin in which functions $f$ are holomorphic and thus (\ref{eq44}) reduces to 
$$m(r,f_ \alpha)< \sum(r),\quad 0 \leq r < \delta $$ 
for $f \in \mathcal{F}_1$ which gives normality of $\mathcal{F}_1$ in this neighbourhood of origin and thus $\mathcal{F}$ at origin.
\\
\textbf{Case(ii)}
\\
There is an $\delta> 0$ with the property that if $\mid \alpha \mid < \delta , f(\alpha) \neq 0,1,\infty;\quad f \textquoteright (\alpha) \neq 0$ 
$$\left(2- \frac{1}{k}- \frac{1}{l}\right) \log^{+} \mid f(\alpha) \mid + \log \mid \frac{1}{f(\alpha)} \mid > M, $$ for all but finitely many $f \in \mathcal{F} _1 $. Since $ \frac{1}{k} + \frac{1}{l} $ is positive quantity thus $\left(2- \frac{1}{k} -\frac{1}{l}\right) < 2 $
implies $$ \left(2-\frac{1}{k} -\frac{1}{l}\right) \log^{+} \mid f(\alpha) \mid < 2 \log ^{+} \mid f(\alpha) \mid$$
\\
which gives 
$$ M< \left(2-\frac{1}{k} - \frac{1}{l}\right) \log^{+} \mid f(\alpha)\mid + \log \big |\frac{1}{f \textquoteright (\alpha)} \big | < 2 \log^{+} \mid f(\alpha) \mid + \log \big |\frac{1}{f \textquoteright (\alpha)} \big |, $$ 
for all but finitely many $f \in \mathcal{F}_1$. Now as done in the case (ii) of Montel\textquoteright s theorem (\ref{montel}) we get $$\frac{\mid f \textquoteright (\alpha)\mid}{1+\mid f(\alpha) \mid ^2} <1,$$ 
for $\mid \alpha \mid < \delta , f(\alpha) \neq 0,1,\infty ; f \textquoteright (\alpha) \neq 0 \quad f \in \mathcal{F}_1$. Now by continuity of $f$ $$\frac{\mid f \textquoteright (z) \mid}{1+\mid f(z) \mid ^2} <1,$$ for $\mid z \mid < \delta$. Thus by Marty\textquoteright s theorem (\ref{lemmaA}) $\mathcal{F}_1$ is normal in $\mid z \mid < \delta $ thus $\mathcal{F}$ is normal in neighbourhood of origin.   
\end{proof} 
\section{A New Composition}
\quad We next require a new composition like $f_\alpha$ in order to establish more results. The Lemmas here require that the family $\mathcal{F}$ is not normal.  
\begin{defn} 
For fix $r_0 <1$ and let $r_0 <r<1$ if $\mid \alpha \mid \leq r_0 $ define
\\
$$ \psi_\alpha(z)= r^2 \frac{(z-\alpha)}{r^2-\bar{\alpha}z} \qquad \forall \quad z \in \Delta .$$.
\end{defn}
Here also we use composition of $f$, from the family, with $\psi _\alpha$ written as $fo \psi _\alpha$.  
\begin{lem} 
\textbf{Properties of $\psi_\alpha$} 
\\
$\psi_\alpha$ is a conformal map of $\mid z \mid \leq r $ onto  $\mid z \mid \leq r $. Also $f o \psi_\alpha$ is defined if $\mid z \mid < r^2 \frac{(1+r_0)}{r_0+r^2}$.
\end{lem}
\begin{proof}
We will first show that $ \mid \psi_\alpha(z) \mid \leq r,$ forall $z \in \Delta$.
\\
Consider 
\begin{align*}
\mid \psi_\alpha(z) \mid &\leq r^2 \frac{(\mid z \mid + \mid \alpha \mid) }{r^2 -\mid \bar{\alpha}z \mid} \leq r^2 \frac{(r+\mid \alpha \mid)}{r^2-\mid \bar{\alpha}z \mid}, \qquad \because \quad \mid z \mid \leq r<1 \\
&\leq  \frac{r(r^2-\mid \bar{\alpha }z \mid)+r \mid \bar{\alpha}z \mid +r^2 \mid \alpha \mid}{r^2-\mid \bar{\alpha}z \mid} \\
&= r+\frac{r \mid \bar{\alpha}z \mid +r^2 \mid \alpha \mid}{r^2-\mid \bar{\alpha}z \mid} \leq r,
\end{align*}
because $\mid z \mid \leq r$ and $\mid \bar{\alpha} \mid = \mid \alpha \mid \leq r_0<1 \Rightarrow \mid \bar{\alpha}\mid \mid z \mid <r $ which gives
$ r^2 -\mid \bar{\alpha}z \mid<r^2-r=r(r-1) < 0 $
and so 
$\frac{r \mid \bar{\alpha}z \mid +r^2 \mid \alpha \mid}{r^2-\mid \bar{\alpha}z \mid} < 0$.
Thus $\psi_\alpha$ maps $\mid z \mid \leq r $ to $\mid z \mid \leq r <1 $.
\\
Now $\psi_\alpha(z) =w $ gives $$ r^2 \frac{(z-\alpha)}{r^2 -\bar{\alpha}z} =w, \quad z \in \mid z \mid \leq r $$
so
$$z = r^2 \frac{w+\alpha}{r^2 + \bar{\alpha}w}.$$
Thus $$\psi^{-1}_\alpha(z)=\psi_-\alpha(z)=r^2 \frac{z+\alpha}{r^2+\bar{\alpha}z}, \quad z \in \mid z \mid \leq r.$$
As $$\psi^{-1}_\alpha o \psi_\alpha(z)= I \quad \mbox{on} \quad \mid z \mid \leq r $$
which gives $ \psi_\alpha $ is onto on $\mid z \mid \leq r $ so $ \psi_\alpha$ maps $\mid z \mid \leq r $ onto $\mid z \mid \leq r $. Similarly we can show that $\psi _\alpha$ is one to one on $\mid z \mid \leq r$.	
\\
\textbf{Claim}:
\\
$\psi_\alpha$ is conformal on $\mid z \mid \leq r $ i.e. $\psi \textquoteright _\alpha(z) \neq 0 $ for $\mid z \mid \leq r $.
$$\psi \textquoteright _\alpha(z)=r^2 \frac{r^2 -\mid \alpha \mid^2}{(r^2-\bar{\alpha}z)^2} =0,$$ if $ \mid \alpha \mid ^2-r^2=0 $ $\Rightarrow \mid \alpha \mid ^2 =r^2$ but $\mid \alpha \mid < r_0 <r \Rightarrow \mid \alpha \mid^2< r^2$.
Thus $$\psi \textquoteright _\alpha(z) \neq 0 \quad \mbox{on} \quad \mid z \mid \leq r.$$ 
\\
\textbf{To Show}:
$\psi_\alpha(z)$ is defined for $\mid z \mid < r^2 \frac{(1+r_0)}{r_0+r^2}$.

We know that $\psi _\alpha(z)$ is not defined for $z$ for which $$r^2-\bar{\alpha}z=0 $$ which gives
 \begin{align*}
  \mid z \mid &= \frac{r^2}{\mid \bar{\alpha} \mid } \geq \frac{r^2}{r_0} \quad \because \quad \mid \bar{\alpha} \mid = \mid \alpha \mid \leq r_0 \\
&= \frac{r^2}{r_0} \frac{(1+r_0)}{(1+r_0)} \\
&> \frac{r^2(1+r_0)}{r_0+r^2}, \quad \mbox{since } r_0^2< r^2 \Rightarrow r_0+r_0^2<r_0+r^ 2. 
  \end{align*}
%If 
%\begin{align*}
%\mid z \mid &=\frac{r^2(1+r_0)}{r_0+r^2} > \frac{r^2(r_0+r^2)}{r_0(r_0+r^2)}  \quad \because \quad r< r\\
%& =\frac{r^2}{r_0}
%\end{align*}
Thus $\psi _\alpha(z)$ is defined for $\mid z \mid \leq \frac{r^2(1+r_0)}{r_0+r^2} $. And $f o \psi_{\alpha}$ is defined where $\psi_\alpha $ is defined that means in $\mid z \mid \leq \frac{r^2(1+r_0)}{r_0+r^2}$.
\end{proof}
\begin{lem} \label{lemma5}
Let $\tau$ and $\theta$ be related as $$re^{\iota \tau}= \psi_\alpha(re^{\iota \theta}), \quad 0\leq \theta \leq 2\pi, 0 \leq \tau \leq 2\pi $$
where $ \mid \alpha \mid \leq r_0 <r$. Then there exist an constant $k=k(r,r_0)$ with the property that 
$$\frac{1}{k} < \big| \frac{d \tau }{d \theta} \big|	< k.$$
\end{lem}
\begin{proof}
We will proof this first by computation. We have $$re^{\iota \tau}= \psi_\alpha(re^{\iota \theta}) = r^2 \left( \frac{re^{\iota \theta}- \alpha}{r^2-\bar{\alpha}re^{\iota \theta}} \right)$$ $$ \Rightarrow e^{\iota \tau} = \frac{re^{\iota \theta}-\alpha}{r-\bar{\alpha}e^{\iota \theta}}.$$
Differentating on both side with respect to $\theta$
$$i e^{i \tau} \frac{d\tau}{d \theta} = i \frac{(r^2- \mid \alpha \mid ^2)}{\mid r-\bar{\alpha}e^{\iota \theta} \mid^2} e^{\iota \theta}$$
\begin{align*}
\big| \frac{d \tau}{d \theta} \big| &= \big| \frac{(r^2-\mid \alpha \mid^2)}{\mid r-\bar{\alpha}e^{\iota \theta} \mid^2} \frac{e^{\iota \theta}}{e^{\iota \tau}} \big|  = \frac{r^2-\mid \alpha \mid^2}{\mid r-\bar{\alpha}e^{\iota \theta} \mid^2} \\
& \leq \frac{r^2-\mid \alpha \mid ^2}{(r-\mid \bar{\alpha} \mid)^2} =\frac{r+\mid \alpha \mid}{r-\mid \alpha \mid} \\
&< \frac{r+r_0}{r-r_0} =k ,\quad k=k(r,r_0) \quad \mbox{constant}. 
\end{align*} 
As $\mid \alpha \mid \leq r_0 \Rightarrow r+ \mid \alpha \mid \leq r+ r_0$ and $\mid \alpha \mid \leq r_0 <r \Rightarrow r-\mid \alpha \mid \geq r-r_0$.
Also 
\begin{align*}
\big| \frac{d \tau}{d\theta} \big| &= \frac{r^2 -\mid \alpha \mid^2}{\mid r-\bar{\alpha}e^{\iota \theta}\mid^2} = \frac{r^2-\mid \alpha \mid^2}{\mid r-\bar{\alpha}e^{\iota \theta}\mid^2} \\
& \geq \frac{r^2-\mid \alpha \mid^2}{(r+\mid \bar{\alpha} \mid)^2} = \frac{r-\mid \alpha \mid}{ r +\mid \alpha \mid} =\frac{1}{\left(\frac{r+\mid \alpha \mid}{r-\mid \alpha \mid}\right)} \\
&\geq \frac{1}{\left(\frac{r+r_0}{r-r_0}\right)} > \frac{1}{k}.
\end{align*}
Thus $$\frac{1}{k} < \big| \frac{d \tau}{d \theta} \big| < k.$$ 

Next we prove same lemma theortically. For this we proceed as follows:
\\
$\psi_\alpha$ is not defined for $r^2=\bar{\alpha}z \Rightarrow \mid z \mid = \frac{r^2}{\mid \bar{\alpha} \mid}> \frac{r^2}{r_0}$. And if $\mid z \mid =\frac{r^2}{r_0}$ then $$r^2=\mid \bar{\alpha}z \mid =\mid \bar{\alpha}\mid \frac{r^2}{r_0} \Rightarrow \mid \alpha \mid =r_0 $$ which is not true. Therefore, $\psi_\alpha$ is not defined for $\mid z \mid \geq \frac{r^2}{r_0}$. Therefore $\psi_\alpha$is defined only on $\mid z \mid < \frac{r^2}{r_0} ,\quad \mid \alpha \mid <r_0.$

Now any compact subset of $\{\, z ; \mid z \mid < \frac{r^2}{r_0} \, \}$ is given by $\{\, z ; \mid z \mid \leq r_1< \frac{r^2}{r_0} \, \}$.
\\
Since
\begin{align*}
\mid \psi_\alpha (z) \mid &\leq r^2 \frac{\mid z \mid + \mid \alpha \mid}{r^2 -\mid \bar{\alpha}z \mid} \leq r^2 \frac{r_1+ \mid \alpha \mid}{r^2-\mid \bar{\alpha}\mid r_1} \\
&< r^2 \frac{\frac{r^2}{r_0}+ \mid \alpha \mid}{r^2-\mid \bar{\alpha}\mid \frac{r^2}{r_0}} = \frac{r^2+\mid \alpha \mid r_0}{r_0-\mid \bar{\alpha} \mid} \\
&<\frac{r^2+r_0^2}{r_0-\mid \bar{\alpha} \mid} \leq \frac{r^2+r_0^2}{r_0}\quad (=c \quad \mbox{say}).
\end{align*}
This gives $\psi_\alpha$ is uniformly bounded on compact subsets of $\{\, z ; \mid z \mid < \frac{r^2}{r_0} \, \}$ and since $\frac{r^2}{r_0} > r \Rightarrow \psi_\alpha $ is uniformly bounded on $ \{\, z ; \mid z \mid \leq r \, \}$.
\\
Next 
\begin{align*}
\psi \textquoteright_\alpha(z)&=r^2 \frac{r^2-\mid \alpha \mid^2}{(r^2-\bar{\alpha}z)^2} =r^2 \frac{r^2-\mid \alpha \mid^2}{(r^2-\bar{\alpha}z)(r^2-\alpha \bar{z})} \\
&=\left( r^2 \frac{z-\alpha}{r^2-\bar{\alpha}z} \right) \left(\frac{r^2-\mid \alpha \mid^2}{(r^2-\alpha\bar{z})(z-\alpha)}\right) \\
&=\psi_\alpha(z) \frac{r^2-\mid \alpha \mid^2}{(r^2-\alpha\bar{z})(z-\alpha)}
\end{align*}  
Consider
\begin{align*}
\frac {r^2-\mid \alpha \mid^2}{(r^2-\alpha\bar{z})(z-\alpha)} &\leq  \frac{r^2-\mid \alpha \mid^2}{\left(\mid z\mid -\mid \alpha \mid \right)\left(r^2-\mid \alpha\mid \mid \bar{z} \mid\right)} = - \frac{r^2-\mid \alpha \mid^2}{\left(-\mid z \mid + \mid \alpha \mid \right)\left(r^2-\mid \alpha\mid \mid \bar{z} \mid\right)} \\
&\leq -\frac{r^2-\mid \alpha \mid^2}{\left(\mid \alpha \mid -r\right)\left(r^2-\mid \bar{\alpha}r\right)} \quad \because \mid z \mid \leq r \\
&= - \frac{r^2-\mid \alpha \mid^2}{\left(\mid \alpha \mid -r\right)\left(r^2-\mid \bar{\alpha}\mid r\right)} \\
&=\frac{1}{r} \left( \frac{r+\mid \alpha \mid }{r-\mid \alpha \mid} \right) < \frac{1}{r} \left( \frac{r+r_0}{r-r_0} \right)\\
&<\frac{r+r_0}{r-r_0}\quad (= C \quad \mbox{say}), \quad \because \frac{1}{r} >1. 
\end{align*}
This implies that $\psi \textquoteright_\alpha$ is uniformly bounded on $\{\, z ; \mid z \mid \leq r \, \} $. So $$\max\{\, \mid \psi  \textquoteright_\alpha(z) \mid ; \mid z \mid \leq r , \mid \alpha \mid \leq r_0 < r \,\} <k,  \quad k=k(r,r_0).$$
Since $\psi_\alpha(re^{\iota \theta}) =re^{\iota \tau}$ gives $\mid \frac{d \tau}{d \theta} \mid = \mid \psi \textquoteright_\alpha \mid <k$
\\
We know that $\psi^{-1}_\alpha = \psi_{-\alpha} $ and $\psi \textquoteright_{-\alpha}(z)=r^2 \frac{r^2-\mid \alpha \mid^2}{(r+ \bar{\alpha}z)^2}$ thus we have
$$\psi_{-\alpha}(re^{\iota \theta}) =re^{\iota \tau }$$
$$\psi  \textquoteright_{-\alpha}(re^{\iota \theta}) (rie^{\iota \theta}) d \theta = (r i e^{\iota \tau}) d \tau$$
$$\Rightarrow \frac{d \tau}{d \theta} =\frac{r^2-\mid \alpha \mid^2}{(r+\bar{\alpha}e^{\iota \theta})^2} \frac{e^{\iota \theta}}{e^{\iota \tau}} = \psi  \textquoteright_{-\alpha}(re^{\iota \theta}) \frac{e^{\iota \theta}}{e^{\iota \tau}}$$
$$\Rightarrow \mid \frac{d\tau}{d\theta} \mid = \mid \psi  \textquoteright_{-\alpha} \mid$$
and $$\mid \psi \textquoteright_\alpha \mid < k \Rightarrow  \mid \psi  \textquoteright_{\alpha} \mid^{-1} > \frac{1}{k}$$ 
$$\mid (\psi  \textquoteright_\alpha )^{-1}\mid > \frac{1}{k} $$
$$\Rightarrow \mid \left[(\psi _\alpha)^{-1}\right]^{\textquoteright} \mid > \frac{1}{k}$$ 
$$\Rightarrow \mid (\psi_{-\alpha})^{\textquoteright} \mid >\frac{1}{k} $$
$$\Rightarrow \mid \frac{d\tau}{ d\theta} \mid > \frac{1}{k}.$$
Thus $$\frac{1}{k}< \mid \frac{d\tau}{d\theta}\mid<k.$$
\end{proof}
\begin{lem} \label{cor5}
Let $\mid \alpha \mid \leq r_0$. Then if $g$ is meromorphic in $\mid z \mid \leq r; \quad r> r_0$ we have
$$\frac{1}{k}m(r,g)< m(r,g o \psi_\alpha)< km(r,g),$$
$\theta , \tau $ are related to each other like in above lemma.
\end{lem}
\begin{proof} 
We have $$\frac{d\tau}{d\theta}= \frac{r^2-\mid \alpha \mid^2}{(r-\bar{\alpha}r^{\iota \theta})^2} \frac{e^{\iota \theta}}{e^{\iota \tau}}.$$
Since $$re^{\iota \tau} =\psi_\alpha(re^{\iota \theta})$$
$$\Rightarrow r \iota e^{\iota \tau} d\tau =\frac{r^2-\mid \alpha \mid^2}{r-\bar{\alpha}e^{\iota \theta}} d(re^{\iota \theta}) = \frac{r^2-\mid \alpha \mid^2}{r-\bar{\alpha}e^{\iota \theta}} r \iota e^{\iota \theta} d \theta.$$
We get
\begin{equation}\label{eq45}
 d \tau= \frac{d\tau}{d\theta} d \theta.
\end{equation}
Now consider
\begin{align*}
m(r,g)&= \frac{1}{2\pi} \int_{0}^{2\pi} \log^{+} \mid g(re^{\iota \tau}) \mid \, d \tau\quad( \mbox{put} \quad re^{\iota \tau} = \psi_\alpha(\iota \theta) \quad \mbox{and }  \quad \mbox{by} \quad (\ref{eq45})) \\
&=\frac{1}{2\pi} \int_{0}^{2\pi} \log^{+}\mid g(\psi_\alpha(re^{\iota \theta})) \mid \frac{d\tau}{d\theta} \, d\theta \\
&=\frac{1}{2\pi} \int_{0}^{2\pi} \log^{+} \mid  g o \psi_\alpha(re^{\iota \theta}) \mid \left( \frac{r^2-\mid \alpha \mid^2}{\mid r-\bar{\alpha}e^{\iota \theta} \mid^2} \right) \, d\theta \\
&\leq \frac{1}{2\pi} \int_{0}^{2\pi} \log^{+} \mid g o \psi_\alpha(re^{\iota \theta})\mid \left( \frac{r^2-\mid \alpha \mid^2}{(r-\mid \bar{\alpha}e^{\iota \theta}\mid)^2} \right) \, d \theta \\
&= \frac{r^2-\mid \alpha \mid^2}{(r-\mid \bar{ \alpha} \mid)^2} \left[ \frac{1}{2\pi} \int_{0}^{2\pi} \log^{+} \mid g o \psi_\alpha(re^{\iota \theta}) \mid \, d\theta \right] \\
 &= \frac{r+\mid \alpha \mid}{r-\mid \alpha\mid} m(r,g o \psi_\alpha) \\
 &< \frac{r+r_0}{r-r_0}m(r,g o \psi_\alpha) = km(r,g o \psi_\alpha).
\end{align*}
Similarly 
\begin{align*}
m(r,go\psi_\alpha) &=\frac{1}{2\pi} \int_{0}^{2\pi} \log^{+} \mid g o \psi_\alpha(re^{\iota \theta}) \mid \, d\theta \\
&= \frac{1}{2\pi} \int_{0}^{2\pi} \log^{+} \mid g(re^{\iota \tau}) \mid \frac{d\theta}{d\tau} \, d\tau \quad \mbox{by}(4.22) \\
&= \frac{1}{2\pi} \int_{0}^{2\pi} \log^{+} \mid g(re^{\iota \tau}) \mid \left( \frac{\mid r-\bar{\alpha} e^{\iota \theta} \mid^2}{r^2-\mid \alpha \mid^2} \right) \, d\tau \\
&= \frac{\mid r-\bar{\alpha}e^{\iota \theta} \mid^2}{r^2-\mid \alpha\mid^2} m(r,g) \\
&\geq \frac{(r-\mid \bar{\alpha}e^{\iota \theta})\mid^2}{r^2-\mid \alpha \mid^2} m(r,g) = \left( \frac{r-\mid \alpha \mid}{r+\mid \alpha\mid} \right) m(r,g) \\
&= \frac{1}{\frac{r+\mid \alpha \mid}{r-\mid \alpha \mid}} m(r,g) > \frac{1}{k} m(r,g).
\end{align*}
Thus $$\frac{1}{k} m(r,g)< m(r,g o \psi_\alpha)< k m(r,g).$$
\end{proof}

Further we proceed to proof some lemmas only for family of holomorphic functions. These lemmas unfortunately fails for family of meromorphic functions.
\begin{note}
If a family $\mathcal{F}$ is not normal in $\Delta$ then there is an $r_0<1 $ and an infinite subfamily $\mathcal{F}_1 \subseteq \mathcal{F}$ such that for each $f \in \mathcal{F}_1$ there corresponds some $\alpha=\alpha(f) , \mid \alpha\mid \leq r_0$ with 
$\mid f(\alpha) \mid \geq 1$.

If this not happened then for each $r<1,$ $\mid f(\alpha) \mid \geq 1$ holds for finitely many $f \in \mathcal{F}$  that mean $\mid f(z) \mid <1 $ for $\mid z \mid <r \quad \forall f \in \mathcal{F}$. Which give that $\mathcal{F}$ is normal in $\mid z \mid < r $ for each $r<1$ thus $\mathcal{F}$ is normal in $\Delta$ which is a contradiction.

Now if $\mathcal{F}$ is not normal on $z=0$ then there is a sequence $\{\,\ f_n \,\} \subseteq \mathcal{F} $ and $\alpha_n=\alpha_n(f_n)$ with $\alpha_n \rightarrow 0 \quad n \rightarrow \infty$ such that $\mid f_n(\alpha_n) \mid \geq 1$.
\end{note}
 \begin{lem} \label{lemma6}
If $\mathcal{F}$ is a family of holomorphic functions in $\Delta$ , and $\mathcal{F}$ is not normal then there is an $r_0<1 $ with
$$m\left(r,\frac{f \textquoteright}{f}\right) < A+B \log^{+}T(R,f)+C \log \frac{1}{R-r},$$ $\frac{1}{2}(1+r_0)<r<R<1 $,
for $f$ in an infinite subfamily $\mathcal{F}_1 $ of $\mathcal{F}$.
\end{lem}
\begin{proof}
Since $\mathcal{F}$ is not normal in $\Delta$ this implies that there exist $r_0<1$ such that for infinite number of $f \in \mathcal{F} \quad \exists \quad \alpha= \alpha(f), \mid \alpha \mid \leq r_0$ with $\mid f(\alpha) \mid \geq 1.$ Let $\mathcal{F}_1$  denote subfamily of $\mathcal{F}$. Consider $$g=f o \psi_{-\alpha}, f \in \mathcal{F}_1$$  we have
$$g(0)=fo\psi_{-\alpha}(0)=f(\alpha).$$

If $g(0)=0 \Rightarrow f(\alpha)=0 $ which is not true as this give $\mid f(\alpha) \mid =0<1$ and if $g(0)=\infty \Rightarrow f(\alpha) =\infty  \Rightarrow f$ has pole at $\alpha$ but $f$ is holomorphic on $\Delta$ so this is not true also. Thus $g(0) \neq 0, \infty$ so we can apply Lemma (\ref{lemmaC}) for $g.$ By Lemma (\ref{cor5}) applied for $f \in \mathcal{F}_1$ we get
\begin{equation}\label{eq46}
 m\left(r,\frac{f \textquoteright}{f}\right)< km \left(r,\frac{f \textquoteright}{f} o \psi_{-\alpha}\right).
\end{equation}
Now
$$g=fo\psi_{-\alpha} \Rightarrow f=go\psi_\alpha $$
\begin{equation}\label{eq47}
\Rightarrow f \textquoteright =g \textquoteright(\psi_\alpha) \psi \textquoteright_\alpha \Rightarrow \frac{f \textquoteright }{f} = \left( \frac{g \textquoteright}{g} o \psi_\alpha \right) \psi \textquoteright_\alpha. 
\end{equation}
This gives 
\begin{equation}\label{eq48}
m\left(r,\frac{f \textquoteright}{f} o \psi_{-\alpha}\right) \leq m\left(r,\frac{g \textquoteright }{g}\right)+ m(r,\psi \textquoteright_\alpha o \psi_{-\alpha}) +\log2.
\end{equation}
By (\ref{eq46}), (\ref{eq48}) and Lemma (\ref{lemmaC}) for $g$ we get 
\begin{equation}\label{eq50}
m\left(r,\frac{f\textquoteright}{f}\right) < k m\left(r,\frac{f\textquoteright}{f} o \psi_{-\alpha}\right) <k \left[ 4 \log^{+}T(r\textquoteright ,f o \psi_{-\alpha})+6 \log \frac{1}{r\textquoteright -r} +D \right],
\end{equation}
 where $ \frac{1}{2}(1+r_0)<r<r\textquoteright<r^2 \frac{1+r_0}{r_0+r^2}$.
 \\
\textbf{ To show} : 
 $$r^2 \frac{r\textquoteright-r_0}{r^2-r_0r\textquoteright} =r_1 >r \textquoteright.$$
\\
If $r_1<r\textquoteright$ that means $r^2 \frac{r\textquoteright-r_0}{r^2-r_0 r\textquoteright} <r\textquoteright$ then $r^{2} r\textquoteright-r^2 r_0  <r\textquoteright r^2-r_0 r\textquoteright^{2}$
$$\Rightarrow r^2> r\textquoteright ^2 \quad \mbox{i.e.} \quad r> r\textquoteright $$ which is not true. Thus we have that the circle $\mid \psi_{-\alpha}(z) \mid =r \textquoteright$ is contained in 
$$\mid z \mid = r^2 \frac{r\textquoteright -r_0 }{r^2 -r_0 r \textquoteright } =r_1.$$
Next we use the inequality $$T(r,f) \leq \log^{+}M(r,f) \leq \frac{R+r}{R-r}T(R,f),$$ $r<R<1,$ for $f$ holomorphic. We get 
\begin{align*}
T(r\textquoteright, f o \psi_{-\alpha}) & \leq \log^{+}M(r\textquoteright,f o \psi_{-\alpha}) \\
& \leq \log^{+} M(r_1,f) \quad \mbox{since } r\textquoteright < r_1 \\
&\leq \frac{R+r_1}{R-r_1} T(R,f), 
\end{align*}
where $r_1<R<1$. Applying $ \log^{+} $ on both side 
\begin{equation}\label{eq51}
\log^{+}T(r \textquoteright, f o \psi_{-\alpha}) \leq \log^{+} \frac{1}{R-r_1} + \log^{+}T(R,f) +A \quad(=\log^{+}(R+r_1)),
\end{equation}
here $A$ is a constant independent of $f$. Now for $R>r$ choose $r\textquoteright,$ so that 
$$\frac{1}{r\textquoteright-r} =\frac{1}{R-r_1}=\frac{1}{R-r_1} \frac{R-r}{R-r}= E (R-r),$$
where $E>0$ and depends on $r_0$.So we get
$$\frac{1}{r\textquoteright-r}=\frac{1}{R-r_1}\leq E(R-r)$$ which gives $$\log^{+} \frac{1}{r\textquoteright-r} =\log^{+} \frac{1}{R-r_1} < \log E + \log^{+} (R-r).$$ 
In (\ref{eq51})
$$\log^{+}T(r\textquoteright,fo\psi _{-\alpha}) <\log^{+}( R-r) + \log^{+} T(R,f) + A. $$
In (\ref{eq50})
\begin{align*}
m\left(r,\frac{f\textquoteright}{f}\right)& < k \left\lbrace 4( \log^{+}(R-r)+ \log^{+}T(R,f)+A)+6 \log^{+} (R-r) +D \right\rbrace \\
&= A+B \log^{+}T(R,f) +C \log \frac{1}{R-r},
\end{align*}where $\frac{1}{2}(1+r_0)<r<R< 1$.
\end{proof}
\begin{rem}
Lemma (\ref{lemmac}) and Lemma (\ref{lemma6}) differs only by the terms $\log^{+} \log^{+} \mid \frac{1}{g(0)} \mid$.

The condition for $\mathcal{F}$ to be not normal is essential for this we consider the family $\mathcal{F} = \{\,\ e^{n(z-1)} ,n=1,2,3\ldots \,\}$. Since any sequence in $ \mathcal{F}$ converges to $\infty$ forall $z \in \Delta$ so $\mathcal{F}$ is normal in $\Delta$.
\\
Now here $f_n(z)= e^{n(z-1)}$ so $\frac{f \textquoteright_n}{f_n} = n,$ forall $ n$ this gives $m\left(r,\frac{f \textquoteright_n}{f_n}\right) = \log^{+} n $ and $N(r,f_n) =0 \quad 0<r<1$ 
\begin{align*}
m(r,f_n) &= \frac{1}{2\pi} \int_{0}^{2\pi} \log^{+} \big| e^{n(re^{\iota \theta}-1)} \big| \, d \theta \\
&= \frac{1}{2 \pi} \int_{0}^{2 \pi} \log^{+} e ^{n(r \cos \theta -1)} \, d \theta.
\end{align*}
\\
Now $\log^{+} e^{n(r \cos \theta -1)} $ has a non zero value when 
$$ e^{n(r \cos \theta -1)} > 1 \Rightarrow n(r \cos \theta -1) > 0 $$
Or $r \cos \theta > 1 $ which is not true as $-1 < r \cos \theta < 1 $ here therefore $\log^{+} e^{n(r \cos \theta -1)} $ has value zero. Thus $m(r,f_n) =0; 0<r<1$  gives that $\log^{+} T(R,f_n) =0, $ 
\\
where $\frac{1}{2}(1+r_0)<r<R<1$.

Therefore right hand side of inequality in Lemma (\ref{lemma6}) is a finite value but left hand side is an increasing function $\log^{+} n $ of $n$. Thus we get a contradiction of Lemma (\ref{lemma6}).
\end{rem}
 \begin{lem} \label{lemma6`}
Let $\mathcal{F}$ be a family of holomorphic functions in $\mid z \mid < \delta , \delta >0$ which is not normal in any neighbourhood of $z_0=0$. Then there are constants $A=A(\delta), B=B(\delta),  C=C(\delta)$ with
$$m\left(r,\frac{f\textquoteright}{f}\right) < A+BT(R,f)+C \log \frac{1}{R-r},$$
\\
when $\frac{1}{2} \delta < r <R< \delta, $ for $f$ in an infinite subfamily of $\mathcal{F}$.
\end{lem}
\begin{proof}
Since $\mathcal{F}$ is not normal at $z_0=0$. Thus there exists a sequence $\{f_n \}$ in $\mathcal{F}$ and $\alpha_n=\alpha_n(f_n)$ with $\alpha _n \rightarrow 0 $ as $n \rightarrow \infty$ such that $\mid f_n(\alpha_n) \mid \geq 1$.

Let $\{f_n :n=1,2,3\ldots \} =\mathcal{F}_1$ subfamily of $\mathcal{F}$ and above $\alpha_n $ as $\alpha = \alpha(f), f \in \mathcal{F}_1$ such that $\mid f(\alpha)\mid \geq 1 \quad \forall \quad f \in \mathcal{F}_1$.
\\
Now as Lemma (\ref{lemma6}) done we can complete our proof.
\end{proof} 
Here we present Zalcman lemma for not normal family of meromorphic functions.
\begin{lem}\label{zal}
\textbf{Zalcman Lemma}(for the proof see ~\cite{sch})
\\
Let $\mathcal{F}$ be a family of holomorphic (meromorphic) functions in $\Delta$ then $\mathcal{F}$ is not normal in $\Delta$ if and only if there exists 
\\
\item[(i)] a number $r$ with $0< r< 1$; \item[(ii)] points $z_n$ satisfying $\mid z_n \mid < r$; \item[(iii)] functions $f_n \in \mathcal{F}$ and \item[(iv)] positive numbers $\rho_n \rightarrow 0 $ as $n \rightarrow \infty $ such that
$$f_n(z_n+ \rho_n w) \rightarrow g(w)$$ as $n\rightarrow \infty $
uniformly (spherically uniformly ) on compact subsets of $\mathbb{C}$ where $g$ is a nonconstant entire (meromorphic ) function in $\mathbb{C}$.
\end{lem}
\begin{rem}
Lemma (\ref{lemma6`}) is a local version of Lemma (\ref{lemma6}).
\end{rem}
\begin{rem}
For $\mathcal{F}$ family of meromorphic functions these lemmas are not true. For this  consider the family $$\mathcal{F} = \{\,\ \frac{e^{n(z-1)}}{z+\frac{1}{n}} ;n \in \mathbb{N}, z \in \Delta \,\}.$$ It is to show that $\mathcal{F}$ is not normal in $\Delta$ .
\\
For this let $z_n = - \frac{1}{n}$ and $\rho_n=e^{-n} $
\begin{align*}
f_n(z_n+ \rho_n z) &= \frac{e^{n (- \frac{1}{n}+e^{-n} z -1)}}{e^{-n} z} = \frac{e^{-1+ne^{-n} z -n}}{e^{-n}z} \\
& =\frac{e^{n e^{-n}z}}{ez} \rightarrow \frac{1}{ez} 
\end{align*}
spherically uniformly on compact subsets of $\mathbb{C}$
where $g(z) = 1/ez$ is meromorphic with pole at zero in $\mathbb{C}$. Thus by Zalcman lemma $\mathcal{F}$ is not normal in $\Delta$.
\\
Here $$\frac{f \textquoteright_n }{f_n} = \frac{n^2z}{nz +1}.$$
\begin{align*}
 m\left(r,\frac{f_n\textquoteright}{f_n} \right)&=\frac{1}{2\pi} \int_{0}^{2\pi} \log^{+}\big|\frac{n^2re^{\iota \theta}}{nre^{\iota \theta+1}} \big| \, d \theta \\
 &= \frac{1}{2\pi} \int_{0}^{2\pi} \log^{+} n^2  \,d\theta + \frac{1}{2\pi} \int_{0}^{2\pi} \log^{+}r -\frac{1}{2\pi} \int_{0}^{2\pi} \log^{+} \big | nre^{\iota \theta}+1 \big | \\
 & \geq \log^{+}n^2-\log^{+}(nr+1) \\
 &\geq \log^{+}n^2-\log^{+}n -\log 2 \\
 &= \log^{+}n -\log 2
 \end{align*}
 that means $m\left(r,\frac{f_n \textquoteright}{f_n} \right)$ increase as $\log^{+} n$. In a similar way we can show that right hand side of Lemma (\ref{lemma6`}) is $O(\log^{+} \log^{+} n)$. This proves the assertion. 
\end{rem}
\begin{lem} \label{lemma7}
$\mathcal{F}$ be a family of holomorphic functions in $\Delta$ (respectively in $\mid z \mid < \eta$) which is not normal in $\Delta$ (respectively in any neighbourhood of the origin). Then there is an $r_1<1$ and a constant $k \textquoteright $  with the property that if $r> r_1$ (respectively $r> \frac{\eta}{2}$) such that
$$m\left(r,\frac{1}{f}\right) \leq k  \textquoteright m(r,f), $$
for infinite number of $f \in \mathcal{F}$.
\end{lem}
\begin{proof}
Since $\mathcal{F }$ is not normal (not normal at $ z_0 =0$) then there is an $r_0<1$  such that to each $f$ of a subfamily $\mathcal{F}_1$ of $\mathcal{F}$ there corresponds 
$\alpha=\alpha(f),\quad \mid \alpha \mid < r_0$ such that $\mid f(\alpha) \mid \geq 1$. 

If $f(\alpha)=0 \Rightarrow \mid f(\alpha) \mid =0$ which is not true and if $f(\alpha) =\infty\Rightarrow \alpha$ is a pole of $f$ in $\Delta$ but $f$ is holomorphic in  $\Delta$. Therefore $f(\alpha) \neq 0,\infty$.(respectively there is a sequence $\{\,\ f_n \,\}$ in $\mathcal{F}$ and $\alpha_n(f_n)= \alpha_n, \quad \alpha_n \rightarrow 0$ with $\mid \alpha_n \mid< \frac{\eta}{3}, n=1,2,3 \ldots $ such that $\mid f_n (\alpha_n) \mid \geq 1$ and $f_n(\alpha_n) \neq 0, \infty $ as done earlier).

Let $r_1 =\frac{1}{2} (1+r_0)$. If $r> r_1 \quad (r> \frac{\eta}{2})$ then by Lemma (\ref{cor5}) we have 
\begin{align*}
m\left(r,\frac{1}{f}\right)&< k m\left(r,\frac{1}{f} o \psi_{-\alpha}\right) \leq k T\left(r,\frac{1}{f} o \psi_{-\alpha}\right) \\
&= k\left\lbrace T(r,fo \psi_{-\alpha}) -\log \mid f(\alpha) \mid \right\rbrace
 \\
&< kT(r,f o \psi_{-\alpha}) =k m(r,f o \psi_{-\alpha})\quad \mbox{since } \log \mid f(\alpha) \mid \geq 0 \quad \mbox{as } \mid f(\alpha) \mid \geq 1 \\
& < k^2 m(r,f) \quad \mbox{by \quad Lemma \quad(\ref{cor5}) }\\
&=k \textquoteright m(r,f)
\end{align*} where $k^2= k \textquoteright $.
\end{proof}
\begin{lem} \label{lemma8}
Let $\mathcal{F}$ be a family of functions holomorphic in $\Delta$ which is not normal and $k \geq 1$ be fixed integer then there exists $r_0<1 , r = r_0(k)$ such that to each fixed  $\rho , r_0< \rho<1$ corresponds constants $A(\rho) , B(\rho) $ and $C(\rho) $ and infinite subfamily $\mathcal{F}_1$ of $\mathcal{F}$ with 
$$m\left(r,\frac{f^{(k)}}{f}\right)< A+B \log^{+} T(R,f)+C \log \frac{1}{R-r},$$  $f \in \mathcal{F}_1,$ for $r_0<r<R<\rho$.
\end{lem}  	
\begin{proof}
The proof of this lemma given here is by Schwick ~\cite{schw}, using induction on $k \in \mathbb{N}$.
For $k=1$, this lemma is nothing but Lemma (\ref{lemma6}) Let us assume that the result is true for $k-1$.
\\
Since $\mathcal{F}$ is not normal in $\mid z \mid < \rho$ therefore by Zalcman lemma \ref{zal} there exists $z_n\rightarrow 0$, $\rho_n \rightarrow0$, functions $f_n \in \mathcal{F}$ and a non constant entire function $g(z)$ such that $$f_n(z_n+\rho_n z) \rightarrow g(z)\quad \mbox{normally in } \mathbb{C}.$$
Assume $z_0\in \mathbb{C}$ satisfying $g(z_0) \neq 0, \infty$ and for $r_0<r<\rho<1$.
Let $$\psi_n(z)= r^2 \frac{z+(z_n+\rho_nz_0)}{r^2+ \bar{(z_n+\rho_nz_0)}z}.$$
By Lemma \ref{cor5} there is constant $k$ such that for any meromorphic function $h$ and for any $n$ sufficiently large (without loss of generality for all $n \in \mathbb{N}$) we have $$\frac{1}{K}m(r,h)\leq m(r,ho\psi_n) \leq Km(r,h)$$ for $r_0<r<\rho<1.$ Also 
\\
$$(f_no\psi_n)^{(k)}=(f_n^{(k)}o\psi_n)(\psi_n\textquoteright)^k+ \sum_{i=1}^{k-1}(f_n^{(i)}o\psi_n)P_i(\psi_n)$$ where $P_i$ is a differential polynomial.
\\
So
\begin{align*}
m\left(r,\frac{f_n^{(k)}}{f_n}\right) & \leq K m\left(r,\frac{f_n^{(k)}}{f_n}o\psi_n\right) \\
& \leq K \left\lbrace m\left(r,\left(\frac{f_n^{(k)}}{f_n}o\psi_n\right)(\psi_n \textquoteright)^k \right) + m \left(r,\frac{1}{(\psi_n\textquoteright)^k} \right)\right\rbrace \\
& \leq K\left\lbrace m\left(r,\frac{(f_no\psi_n)^{(k)}}{f_no\psi_n} \right) +\sum_{i=1}^{k-1} m \left(r,\frac{f_n^{(i)}}{f_n}o\psi_n \right) +\sum_{i=1}^{k-1} m (r,P_i(\psi_n))+A \right\rbrace \\
& \leq K \left\lbrace m\left(r,\frac{(f_no\psi_n)^{(k)}}{f_no\psi_n} \right)+\sum_{i=1}^{k-1} m\left(r,\frac{f_n^{(i)}}{f_n} \right)+A \right\rbrace,
\end{align*} $r_0<r<\rho<1.$
The induction hypothesis is now applicable to the summation on the right. By Hiong\textquoteright s estimate (\ref{hio est}) we have
\begin{align*}
m\left(r,\frac{(f_n o \psi_n)^{(k)}}{f_n o \psi_n}\right) &< C\left\lbrace\log^{+}T(R,f_no\psi_n)+ \log \frac{1}{R-r} +\log^{+} \log^{+} \frac{1}{\mid f_n(z_n+\rho_n z_0) \mid}+A \right\rbrace \\
& < C \left\lbrace \log^{+}T(R,f_no\psi_n) + \log \frac{1}{R-r} +A \right\rbrace,
\end{align*} for $r_0<r<R<\rho<1.$ Now we can proceed same as in Lemma (\ref{lemma6}) to complete the proof. 
\end{proof}
From above lemma we get a standard comparison between $T(r,f^{(k)}) $ and $T(r,f)$ for a family of holomorphic functions which is not normal.
\begin{cor} \label{cor8}
Let $\mathcal{F}$ be a family of functions holomorphic in $\Delta$ which is not normal, and let $k$ be a fixed integer $\geq 1$. Then there exists $r_0 < 1 $ such that to each fixed $\rho, \quad r_0< \rho <1$ corresponds $A(\rho) , B(\rho) ,C(\rho)$ and an infinite subfamily $\mathcal{F}_1$ of $\mathcal{F}$ with 
$$T(r,f^{(k)}) \leq A+BT(R,f) +C \log \frac{1}{R-r} , f \in \mathcal{F}_1,$$ if $r_0<r<R< \rho.$
\end{cor} 
\begin{proof}
Since $f$ holomorphic give $f^{(k)}$ is holomorphic so 
$$T(r,f^{(k)}) =m(r,f^{(k)}) \leq m\left(r,\frac{f^{(k)}}{f}\right)+m(r,f)= m\left(r,\frac{f^{(k)}}{f}\right) +T(r,f)$$
By  Lemma \ref{lemma8} we get $$T(r,f^{(k)}) \leq A+B \log^{+} T(R,f)+C \log \frac{1}{R-r} + T(r,f),$$ 
$ r_0<r<R<\rho, \quad f \in \mathcal{F}_1$. And since $\log^{+} x < x $ we get $$T(r,f^{(k)}) \leq A+BT(R,f) +C \log \frac{1}{R-r},$$
\\
$ r_0<r<R<\rho, \quad f \in \mathcal{F}_1$.
\end{proof}
\chapter{Nonvanishing Families of Holomorphic Functions }
\quad By non-vanishing family of holomorphic functions we mean family of holomorphic functions non-vanishing holomorphic functions. Here we give normality criterion for this family and an extension of the result is supposed.
\section{A New Family }

\quad In this section normality of a family $\mathcal{F}$ of non-vanishing holomorphic functions is compared with the normality of a new family $\mathcal{G}$ each function of which is  obtained by $\mathcal{F}$ as follows:   

$\mathcal{F}$ is family of holomorphic functions on $\Delta$ and $a_0(z),a_1(z),\ldots,a_{k-1}(z) $ are fixed holomorphic functions then $\mathcal{G}$ consists of the functions 
\begin{equation}\label{eq5.1}
g(z)= f^{(k)}(z) + a_{k-1}(z) f^{(k-1)}(z)+\ldots+ a_0(z)f(z), 
\end{equation}
where $f \in \mathcal{F}$ and $z \in \Delta$. In the same way we give $g_\alpha $ as $$ g_\alpha(z) = f^{(k)}_\alpha(z)+ a_{k-1,\alpha}(z) f^{(k-1)}_{\alpha}(z) + a_{k-2, \alpha} f^{(k-2)}_{\alpha} (z)+ \ldots + a_{0, \alpha}(z) f_\alpha(z),$$ 
for $z \in \Delta , f \in \mathcal{F}$.

Now normality of $\mathcal{F}$ relates with normality of $\mathcal{G}$ which we give in next lemma.
\begin{lem} \label{lemma9}
If $\mathcal{G}$ is normal then $\mathcal{F}$ is normal.
\end{lem}
\begin{proof}
Suppose $\mathcal{G}$ is normal. Let $ \{\, g_n \,\} \subset \mathcal{G}$ such that $g_n \rightarrow g_0 $ normally. To show that the corresponding family $ \{\, f_n \,\} \subset \mathcal{F}$ is normally convergent in some neighbourhood of each point $z_0 \in \Delta $ that mean normal at each point $z_0 \in \Delta $.
\\
We have two cases
\\
\textbf{Case(i)}
\\
$g_0(z_0) \neq 0 $ without loss of generality let $z_0 = 0$ then there is a neighbourhood of $z_0 =0; \mid z \mid  \leq 2 \delta < \frac{1}{2} $ for which $\mid g(z) \mid > \epsilon > 0 $ for an infinite subfamily $\mathcal{G}_1 $ of $\{\, g_n \,\}$. In general 
$$m\left(r,\frac{g}{f}\right) \leq \sum_{j=0}^{k-1} m(r,a_j)+ \sum_{j=0}^{k} m\left(r,\frac{f^{(j)}}{f}\right) +C $$
$$H_0(r) = \sum_{j=0}^{k-1} m(r,a_j) $$ \textquoteleft say \textquoteright. Suppose $\mathcal{F}_1$ is not normal in neighbourhood $\mid z \mid 2 < \delta$ of $z_0 =0$. Then by the local version of Lemma (\ref{lemma8}); for $\delta<r< 2 \delta$ we have
\begin{align*}
m\left(r,\frac{1}{f}\right) &\leq m\left( r,\frac{g}{f}\right)+ m\left(r,\frac{1}{g}\right) \leq H_0(r) + A \log^{+} T( R,f) + B \log \frac{1}{R-r}  \\
&+ m\left(r,\frac{1}{g}\right) +C,
\end{align*}for infinite number of $f \in \mathcal{F}_2 \subset \mathcal{F}_1 $. Now $r < \frac{1}{2}$ and $H_0(r) = \sum_{j=0}^{k-1} m(r,a_j) $ is an increasing function of $r$ therefore $H_0(r) < H_0 (\frac{1}{2} )$ and $$m\left(r,\frac{1}{g}\right) = \frac{1}{2 \pi} \int_{0}^{2 \pi} \log^{+} \big | \frac{1}{g(re^{\iota \theta})} \big | \, d \theta  = O(1),$$ 
as $m\left(r,\frac{1}{g}\right) < \log^{+}\frac{1}{\epsilon} < \infty $.

Since $f$ is non vanishing holomorphic function therefore $1/f$ is holomorphic on $\Delta$ that gives $$ m\left(r,\frac{1}{f}\right)=T\left(r,\frac{1}{f}\right).$$
Summing up all this
$$T\left(r, \frac{1}{f}\right) = m\left(r,\frac{1}{f}\right) \leq A \log^{+} T(R,f) +B \log \frac{1}{R-r} + C + H_0\left(\frac{1}{2}\right) $$
$$T\left(r,\frac{1}{f}\right) \leq A \log^{+} T(R,f) + B \log \frac{1}{R-r} +C, \textquoteright$$
forall $f \in \mathcal{F}_2 $  where $C \textquoteright= C+H_0\left(\frac{1}{2}\right)$ . Again if $\mathcal{F}_2$ is not normal on the neighbourhood $\mid z \mid < 2 \delta $ of $z_0=0 $ then by Lemma (\ref{lemma7}) for $\frac{\delta}{2} <r < 2 \delta $ and $k$ we have
$$m(r,f) \leq k m\left(r,\frac{1}{f}\right),$$ for $f \in \mathcal{F}_3 \subset \mathcal{F}_2$ or	
$$T(r,f) \leq T\left(r,\frac{1}{f}\right)$$ because $f$ is non vanishing holomorphic. Thus 
$$T(r,f) \leq A\textquoteright \log^{+} T(R,f) +B \textquoteright \log \frac{1}{R-r} + C {\textquoteright} { \textquoteright}, $$ for $\frac{\delta}{2} < r< R < 2 \delta, \quad f \in \mathcal{F}_3.$  Now let $U(r) = T(2 \delta r,f) ; \gamma(r) \equiv 0 $ then by Lemma (\ref{lemma2}) we get
$$m(2 \delta r ,f)= T(2 \delta r,f) \leq M+ 2 B \textquoteright \log \frac{1}{R-r},$$ for $f \in \mathcal{F}_3 \quad \frac{\delta}{2}< r <R< 2 \delta$.

By Lemma (\ref{lemmaB}) $\mathcal{F}_3$ is normal in the neighbourhood of origin and hence $\mathcal{F}$ is normal at $z_0 =0$.
\\
\textbf{Case(ii)}
\\
$g_0 (z_0) =0.$ First we recall that
\\
if $h_1(z), h_2(z),\ldots,h_k(z)$ form a  linearly independent set of solution of homogeneous equation 
$$L(h) =a_0(z) h(z) +a_1(z) h \textquoteright (z)+\ldots+h^{(k)} (z)=0 $$ 
then solution set of $L(h)=0 $ is given by linear span of $\{\, h_1(z),h_2(z),...,h_k(z) \,\}$ and solution of the system $L(h) = g(z)$ ~\cite{cod} is given by 
\begin{equation}\label{eq5.2}
h(z) = \sum_{m=1}^{k} \left( \alpha_m + \beta _m(z)\right) h_m (z)
\end{equation}where $\alpha_m$ are constants and 
\begin{equation}\label{eq5.3}
\beta_m(z) = \int_{o}^{z} \left[ \frac{W_m(h_1,h_2,\ldots,h_k)(t)}{W(h_1,h_2,\ldots,h_k)(t)} \right] g(t) \, dt, 
\end{equation}
where $W(h_1,h_2,\ldots,h_k)(t)$ is wronskian of $h_1,h_2,\ldots, h_k$ and $W_m(t)$ is the wronskian of $ h_1,h_2,\ldots,h_{m-1},h_{m+1},\ldots,h_k$.
\\
From derivation of (\ref{eq5.1}) we get that $\beta_m(z)$ satisfies 
\\
$\beta\textquoteright_1(z) h_1(z) + \beta \textquoteright _2 (z) h_2(z) +\ldots+ \beta \textquoteright _k(z) h_k(z)=0$
\\
$\beta \textquoteright _1(z) h \textquoteright _1(z) + \beta \textquoteright _2 (z) h \textquoteright _2 (z) +...+ \beta \textquoteright _k(z) h \textquoteright _k(z) =0 $
.
\\
.
\\
.
\\
$\beta \textquoteright _(z) h^{(k-2)} _1(z) + \beta \textquoteright_2(z) h ^{(k-2)}_2 (z) +...+ \beta \textquoteright _k(z) h^{(k-2)} _k(z) =0$
\\
$\beta\textquoteright_1(z) h^{(k-1)}_1(z) + \beta \textquoteright_2(z) h^{(k-1)}_2(z)+...+ \beta \textquoteright _k (z) h^{(k-1)}_k(z)=g(z).$

Now from the system $L(f_n)=g_n $ and (\ref{eq5.2}) we can say that there are constants $\alpha_{n,m} , m=1,2,3,\ldots,k $ and functions $\beta_{n,m}(z), m=1,2,3,\ldots,k $ determined by (\ref{eq5.2}) with $g=g_n$ for which $$f_n(z) = \sum\left( \alpha_{n,m} + \beta_{n,m}(z) \right) h_m(z)$$ 
since $g_0$ is cluster point of $\{\,g_n \,\}$ and $g_0(0)=0$ therefore for $\epsilon > 0 $ there exists $\delta> 0 $ and $M< \infty$ constant such that $$\big| \frac{w_m(t)}{w(t)} \big| < M , \mid g_n(t) \mid < \epsilon,$$ for $\mid t \mid < \delta $ and $n> n_0$. Thus $$\mid \beta _{n,m}(z) \mid \leq M \epsilon = \epsilon^{*}, \mid z \mid < \delta , n > n_0.$$
Let $ \mathcal{F}^{*}$ be the class of corresponding combinations $\left\lbrace\sum \alpha_{n,m} h_m \right\rbrace$ finding from (\ref{eq5.2}), corresponding to each $F \in \mathcal{F}^{*}$ there is an $f_n \in \mathcal{F}$ such that
$$\mid F(z)-f_n(z) \mid = \big|- \sum_{m=1}^{k} \beta_{n,m}(z) h_m(z) \big| \leq \epsilon^{*} k \max_{\mid z \mid < \delta} \{\, \mid h_m(z) \mid \,\} \leq \epsilon_1,$$
\\
for $\mid  z \mid < \delta_1(\epsilon_1). $ This gives $F(z) \leq f_n(z) + \epsilon_1 $.

From here we have that $\{\, f_n \,\} $ is normal if the corresponding family in $\mathcal{F}^{*}$ is normal. Therefore we have to show that $\mathcal{F}^{*}$ is normal which is next lemma and our present lemma is a local version of next lemma.
\end{proof} 
\begin{lem} \label{lemma10}
Let $h_1(z),h_2(z),\ldots,h_k(z)$ be fixed linearly independent functions holomorphic in $\Delta$. Let $\mathcal{F}^{*}$ be a family of function holomorphic in $\Delta$ such that if $F \in \mathcal{F}^{*} $
$$F(z)=\sum_{i=1}^{k} \alpha_i h_i(z),$$
for suitable constants $\alpha_i= \alpha_i(f)$.
\\
Suppose that there exists an M with the property that for each $F \in f^{*}$ corresponds $ g; g(F)$ holomorphic in $\Delta$ $\mid g(z) \mid < M $ such that the equation 
$$F(z)+g(z) =0 $$
has no solution in $\Delta$. Then $\mathcal{F}^{*}$ is normal. 
\end{lem}
\begin{note}
Lemma (\ref{lemma9}) is a particular case of Lemma (\ref{lemma10}) as if $F(z)= \sum \alpha_i h_i(z)$ then $g(z)= \sum \beta _i h_i(z)$ will give $F(z)+g(z) =f(z)=0 $ has no solution in $\Delta$.
\end{note}
\begin{proof}
Let $F \in \mathcal{F}^{*}$ and $j=j(F)$ be the least integer such that $\mid \alpha_j \mid \geq \mid \alpha_i \mid$ forall $ i \neq j$.
\\
Write $$F(z)= \alpha_j \sum _{i=1}^{k} \frac{\alpha_i}{\alpha_j} h_i(z); j=j(F).$$

Unless $\alpha_j\rightarrow \infty$ for an infinite subfamily $\mathcal{F^{*}}_1 $ of $\mathcal{F^{*}}$ and some fixed
\\
$j \in \{1,2,3,\ldots , k\,\}, \mathcal{F^{*}}$ will consists of uniformly bounded holomorphic functions and thus would be normal by Theorem (\ref{sch montel}).

With no loss of generality assume that $j=1$ then 
\begin{align}
 F&=\alpha_1 h_1+\alpha_1 \sum_{i=2}^{k} \frac{\alpha_i}{\alpha_1} h_i \ notag \\
 &= \alpha_1 \left( h_1+ \sum_{i=2}^{k} \beta_i h_i \right)\label{eq5.5}
\end{align} 
$$\beta_i=\frac{\alpha_i}{\alpha_1}; i=1,2,3,\ldots,k$$ also $\mid \beta_i \mid \leq 1 \quad \forall i=2,3,\ldots,k$.

Let K be compact subset of $\Delta$. Since $(\beta_2, \beta_3, \ldots, \beta_k)$ belong to a compact subset of $\mathbb{C}^{k-1}$ , therefore we can chose a convergent subsequence  $(\beta_{2_l}, \beta_{3_l},\ldots, \beta_{k_l})$. Let $$ h_1+ \sum_{i=2}^{k} \beta_i^{*} h_i=F^{*}$$ be the limit of corresponding subfamily of $\mathcal{F^{*}}_1$ and convergence is normal in $\Delta$. If $F^{*} \equiv 0$ this gives a contradiction as $h_i, i=1,2,3, \ldots, k$ are linearly independent.

\textbf{To show} 
$F^{*}$ never vanishes on the interior of K.
\\
For this proceed as follows: from (\ref{eq5.5}) 
$$\frac{F}{\alpha_1}=h_1+ \sum_{i=2}^{k} \beta_i h_i.$$
Now $$F=F+g-g$$ and $$ F+g $$ never vanishing on $\Delta$ since $\mid g(z) \mid <M$ and $\alpha_1 \rightarrow \infty $; $\frac{g}{\alpha_1}$ tends uniformly to 0. Therefore $\frac{F+g}{\alpha_1}$ tends normally to $F^{*}, F+g \not \equiv 0$. By Hurwitz\textquoteright s theorem $F^{*}$ is never vanishing on the interior of K. Thus for compact subset $K_1$ of $K$, there exists $\eta >0 $ such that $$ \big | h_1 + \sum_{i=2}^{k} \beta_i h_i \big| > \eta; z \in K_1 $$
for an infinite subfamily of $\mathcal{F^{*}}_1$.

Since $\alpha_1 \rightarrow \infty$ from (\ref{eq5.5}) we can say that $\infty$ is a cluster point of $\mathcal{F^{*}}$ hence the result of the Lemma obtained.  
\end{proof}    

Here is meromorphic version of Hurwitz\textquoteright s theorem for the proof refer ~\cite{sch}
\begin{thm}
\textbf{Hurwitz\textquoteright s theorem}

Let $\{ \,f_n \,\}$ be a sequence of meromorphic functions on a domain $\Delta$ which converges spherically uniformly on compact subsets to a function $f$ which may be $\equiv \infty$. If each $f_n$ is zero free in $\Delta$, then $f$ is either zero free in $\Delta$ or $f \equiv 0$.
\end{thm}
It is well known that if $f$ is entire, $f(z) \neq 0$ and $f \textquoteright(z)\neq 1$ then $f$ reduces to constant. In 1934, Montel conjectured that a family of analytic function in a domain satisfying $f(z) \neq 0, \quad f\textquoteright(z) \neq 1 $ would constitute a normal family whose proof was given by Miranda[1935], using Nevanlinna theory and he was able to replace derivative condition by $f^k(z)\neq 1$ for given $k \in \mathbb{N}$. 
We have proved a generalization of Montel-Miranda theorem due to Chuang[1940].  
\begin{thm} \label{th3}
Let $\mathcal{F}$ be a family of non-vanishing holomorphic functions in $\Delta$ and $a_0(z),a_1(z),\ldots,a_k(z)$ fixed holomorphic functions. Let $\mathcal{G}$ consists of the functions 
$$g(z) = f^{(k)}(z) + a_{k-1}(z) f^{(k-1)} (z)+\ldots+ a_0(z) f(z), \quad z \in \Delta , f \in \mathcal{F}$$
and assume that the equation $g(z)=1$ has no solutions for $z \in \Delta$ then $\mathcal{F}$ is normal.
\end{thm}
\begin{proof}
Let $\mathcal{F}_1$ be an arbitrary denumberable subcollection from $\mathcal{F}$; we will show that $\mathcal{F}_1$ contains an infinite subfamily which is normally convergent in $\mid z \mid < \rho ; \quad \rho $ is fixed number $<1$. For convenience let $$\rho \geq \frac{1}{2} \max \{\, 1+r_0,\ldots,1+r_3 \,\},$$ 
 where $r_0,\ldots,r_3 $ are determined below depending on $\mathcal{F}_1$. Suppose that $\mathcal{F}_1 $ has no normally convergent subfamily. $\mathcal{G}_1$ denote the corresponding infinite subfamily of $\mathcal{G}$ and $g_\alpha $ related to $f_\alpha$. Now let $\mathcal{G}_1$ is not normal then
 \\
(by Lemma \ref{lemmaA}) there exists $ r_0 <1 $ and infinite subfamily $\mathcal{F}_2 \subset \mathcal{G}_1$ such that for all number $M < \infty$  and for $g \in \mathcal{G}_2$ there is $\alpha(g), \mid \alpha \mid < r_0 $ for which
$$\frac{\mid g \textquoteright(\alpha) \mid}{1+ \mid g(\alpha) \mid^2} > M,\quad \mid \alpha \mid < r_0.$$
In particular let $M=1$
$$\frac{\mid g \textquoteright(\alpha) \mid}{1+ \mid g(\alpha) \mid^2} >1, \quad \mid \alpha \mid < r_0 $$ 
this implies $$g \textquoteright(\alpha) \neq 0.$$
and $$\frac{1+\mid g(\alpha) \mid ^2}{\mid g \textquoteright (\alpha)} <1.$$
Now $$\mid g(\alpha) -1 \mid \leq \mid g(\alpha) \mid + 1 \leq \mid g(\alpha) \mid^2 +1$$
this implies  $$ \frac{\mid g(\alpha) -1 \mid}{ \mid g \textquoteright (\alpha) \mid} \leq \frac{\mid g (\alpha) \mid^2+1}{\mid g \textquoteright (\alpha) \mid } < 1.$$
Thus we have
\begin{equation}\label{eq5.6}
\frac{\mid g(\alpha)-1 \mid}{ \mid g \textquoteright (\alpha) \mid} < 1 
\end{equation}
we may assume that $g(\alpha) \neq 0,1,\infty, g \textquoteright (\alpha) \neq 0$ for all $g \in \mathcal{G}_2$. We have \begin{align*}
T(R \textquoteright ,g_\alpha) &\leq T(R\textquoteright,f^{(k)}_\alpha(z)) +T(R\textquoteright , a _{k-1,\alpha}(z) f^{(k-1)}_\alpha(z)) +\ldots \\\\
&+T(R \textquoteright ,a_{0,\alpha}z)f_\alpha(z)) + \log (k+1) \\
& \leq T(R \textquoteright,f^(k)_\alpha(z))+T(R\textquoteright,a_{k-1,\alpha}(z)+T(R \textquoteright,f^{(k-1)}_\alpha(z)+\ldots  \\
&+T(R  \textquoteright, a_{0,\alpha}(z)) +T(R \textquoteright ,f_\alpha(z))+ k \log 2 +\log (k+1) \\
&=\sum_{j=0}^{k-1} T(R \textquoteright,a_{j,\alpha}(z)) + \sum_{j=0}^{k} T(R \textquoteright,f^{(j)}_ \alpha(z)) +D.
\end{align*}
This implies that
\begin{align}
\log^{+}T(R \textquoteright,g_\alpha)&\leq \sum_{j=0}^{k-1} \log^{+} T\left(R \textquoteright,a_{j,\alpha} (z)\right)+ \sum_{j=0}^{k} \log^{+} T\left(R \textquoteright,f^{(j)}_\alpha(z)\right)+ \log^{+} D \notag \\
& + \log k+ \log (k+1) \notag \\
&= L_\alpha (R \textquoteright) + \sum_{j=0}^{k} \log^{+} T\left(R \textquoteright, f^{(j)}_\alpha(z)\right)+A,\label{eq5.9}
\end{align}
where $L_\alpha(R \textquoteright)=\sum_{j=0}^{k-1} \log^{+} T\left(R \textquoteright,a_{j,\alpha} (z)\right)$ and $A=\log^{+} D + \log k+ \log (k+1).$ Since $\mathcal{G}_2$ may be assumed to have no convergent subfamily then by Lemma (\ref{lemma8}) and its corollary applying to (\ref{eq5.6}) we have that there is infinite subfamily $\mathcal{G}_3$ of $\mathcal{G}_2$ and $\mathcal{F}_3$ of $\mathcal{F}_2$ such that there exists $r_1 <1 $ having that corresponding to each fixed $\rho,\quad r_1 < \rho<1$ (assuming that $r_1>r_0 $) we have in (\ref{eq5.6}).
\begin{align}
\log^{+}T(R \textquoteright,g_ \alpha) &\leq L_\alpha(R\textquoteright)+ \sum_{j=0}^{k} \log^{+} \left(A\textquoteright+ BT(R,f_\alpha)+C \log \frac{1}{R-R\textquoteright}\right) +A\notag \\
& \leq L _ \alpha(R\textquoteright)+ \sum_{j=0}^{k} \log^{+}{A\textquoteright}+ \sum_{j=0}^{k} \log^{+} \left(C \log \frac{1}{R-R\textquoteright} \right) \notag\\
&+ \sum_{j=0}^{k} \log^{+} BT(R,f_\alpha) +A \notag \\
& \leq L_\alpha(R\textquoteright) + (k+1) \log^{+}{A\textquoteright} \notag  \\
&+ (k+1) \log^{+}C + (k+1) \log 2 +(k+1) \log \frac{1}{R-R\textquoteright} \notag \\
&+(k+1) \log^{+}B+(k+1) \log^{+} T(R,f_\alpha) + (k+1) \log2 \quad \because \log^{+} x \leq x\notag  \\
& = L_\alpha(R\textquoteright) +A \log \frac{1}{R-R\textquoteright} +B \log^{+}T(R,f_\alpha) +C,
\end{align}
for $r_1<R\textquoteright<R<\rho<1$. Now we use equation (\ref{eq24}) of chapter 4 to all $g \in \mathcal{G}_3$ which is as follows: 
\begin{align*}
T(r,g_\alpha) &\leq \bar{N}(r,g_\alpha) +N\left(r,\frac{1}{g_\alpha}\right) +N\left(r,\frac{1}{g_\alpha-1}\right) -N\left(r,\frac{1}{g\textquoteright_\alpha}\right) \\
&+ 2 m\left(r,\frac{g\textquoteright_\alpha}{g_\alpha}\right)+m\left(r,\frac{g\textquoteright_\alpha}{g_\alpha-1}\right) + \log \big | \frac{g(\alpha)\{\,g(\alpha)-1 \,\}}{g\textquoteright_\alpha(0)} \big | +C.\label{eq5.16}
\end{align*}
Here we apply Lemma (\ref{lemma8}) for $g_\alpha$ and $R \textquoteright = \frac{1}{2} (r+R)$ then for each fixed $\rho<1$ there is an $r_1< \rho$ such that
\begin{align*}
T(r,g_\alpha)&\leq N\left(r,\frac{1}{g_\alpha}\right)+N\left(r,\frac{1}{g_\alpha-1}\right)-N\left(r,\frac{1}{g\textquoteright_\alpha}\right) \notag \\
&+2 \left[A+B \log^{+} T(R\textquoteright,g_\alpha) +C \log \frac{1}{R\textquoteright-r})\right] +A \textquoteright+B\textquoteright \log^{+} T(R\textquoteright,g_\alpha-1) \notag \\
& +C \textquoteright \log \frac{1}{R\textquoteright-r} +\log \big | \frac{g(\alpha) \{\,g(\alpha)-1 \,\}}{g\textquoteright(\alpha)} \big |+C \notag \\
& = N\left(r,\frac{1}{g_\alpha}\right)+N\left(r,\frac{1}{g_\alpha-1}\right)-N\left(r,\frac{1}{g\textquoteright_\alpha}\right) +B\textquoteright{\textquoteright} \log^{+} T(R\textquoteright,g_\alpha) \notag \\
&+C_1 \log \frac{1}{R\textquoteright-r} + \log \big | \frac{g(\alpha)\{\,g(\alpha)-1 \,\}}{g\textquoteright(\alpha)} \big |+C\textquoteright_1,
\end{align*}
for $ r_1<r < R \textquoteright<\rho<1 ; \quad g \in \mathcal{G}_3$.
By equation (\ref{eq5.16})
\begin{align}
T(r,g_\alpha)&\leq N\left(r,\frac{1}{g_\alpha}\right)+N\left(r,\frac{1}{g_\alpha-1}\right) -N\left(r,\frac{1}{g\textquoteright_\alpha}\right) \\
&+ B\textquoteright{\textquoteright} \left[L_\alpha(R\textquoteright)+A \log \frac{1}{R-R\textquoteright}+ B \log ^{+} T(R,f_\alpha) + C\right] + C_1 \log \frac{1}{R\textquoteright-r} \\
& + \log \big | \frac{g(\alpha)\{\,g(\alpha)-1\,\}}{g\textquoteright(\alpha)} \big | +C\textquoteright_1, \label{eq5.19}
\end{align}
$ r_1<r<R\textquoteright<R<\rho<1.$
Since $L_\alpha(R\textquoteright) =\sum \log^{+}T(R\textquoteright,a_{j,\alpha})$ and $T(r,f)$ is increasing function of $\log r$, hence for $r<R \textquoteright<R$ we have $L_\alpha(R\textquoteright)< L_\alpha(R)$ and using $R\textquoteright=\frac{1}{2} (R+r)$ ($r< \frac{1}{2}(R+r) <R)$) we get $$\log \frac{1}{R\textquoteright-r}= \log \frac{2}{R-r} =\log2 + \log \frac{1}{R-r}$$ and $$\log \frac{1}{R-R\textquoteright} =\log \frac{2}{R-r}=\log 2 + \log \frac{1}{R-r}.$$
Using all these in (\ref{eq5.19}) we get
\begin{align}
T(r,g_\alpha) &\leq N\left(r,\frac{1}{g_\alpha}\right) +N\left(r,\frac{1}{g_\alpha-1}\right)-N\left(r,\frac{1}{g\textquoteright_\alpha}\right)+B\textquoteright{\textquoteright} L_\alpha(R) +E \log \frac{1}{R-r}\notag  \\
& +B \log^{+}T(R,f_\alpha) + \log \big | \frac{g(\alpha)\{\,g(\alpha)-1 \,\}}{g\textquoteright(\alpha)} \big |+C,\label{eq5.21}
\end{align}
 where $r_1<r<R<\rho , \quad g \in \mathcal{G}_3.$
Since $\mid \alpha\mid < r_0,\quad R< \rho$ and $L_\alpha(R) = \sum \log^{+}T(R,a_{j,\alpha})$. $T(r,f)$ strictly increasing function of $\log r$ (hence of $r$) we have 
$$L_\alpha(R)< A(r_0,\rho).$$ 
Jensen\textquoteright s formula gives $$m\left(r,\frac{1}{g_\alpha}\right)+N\left(r,\frac{1}{g_\alpha}\right)=T(r,g_\alpha)- \log \mid g(\alpha) \mid$$ 
using this in equation (\ref{eq5.21}) 
\begin{align}
T(r,g_\alpha)& = m\left(r,\frac{1}{g_\alpha}\right)+N\left(r,\frac{1}{g_\alpha}\right) + \log \mid g(\alpha) \mid  \\
&\leq N\left(r,\frac{1}{g_\alpha}\right) +N\left(r,\frac{1}{g_\alpha-1}\right)-N\left(r,\frac{1}{g\textquoteright_\alpha}\right)+E \log \frac{1}{R-r} \notag \\
&+ B \log^{+} T(R,f_\alpha) + \log \big | \frac{g(\alpha)-1}{g\textquoteright(\alpha)} \big |+ \log \mid g(\alpha) \mid+C \label{eq5.24}
\end{align}
we get
\begin{align}
m\left(r,\frac{1}{g_\alpha}\right)&< N\left(r,\frac{1}{g_\alpha-1}\right)-N\left(r,\frac{1}{g\textquoteright_\alpha}\right) +E \log \frac{1}{R-r} + B \log^{+} T(R,f_\alpha) \\
&+ \log \big | \frac{g(\alpha)-1}{g\textquoteright(\alpha)}\big | +C,\label{eq5.26}
\end{align}
 if $ r_1<r<R< \rho , \quad g \in \mathcal{G}_3$.
\\
Also 
\begin{align}
T\left(r,\frac{1}{f_\alpha}\right)&= m\left(r,\frac{1}{f_\alpha}\right)+N\left(r,\frac{1}{f_\alpha}\right)=m\left(r,\frac{g_\alpha}{f_\alpha} \frac{1}{g_\alpha}\right)+N\left(r,\frac{1}{f_\alpha}\right)\notag \\
& \leq m\left(r,\frac{g_\alpha}{f_\alpha}\right)+ m\left(r,\frac{1}{g_\alpha}\right) +N\left(r,\frac{1}{f_\alpha}\right).\label{eq5.28}
\end{align}
Now $$g_\alpha=f^{(k)}_\alpha+a_{k-1,\alpha}f^{(k-1)}_\alpha+\ldots+a_{0,\alpha}f_\alpha $$
$$\frac{g_\alpha}{f_\alpha}=\frac{f^{(k)}_\alpha}{f_\alpha}+ a_{k-1,\alpha} \frac{f^{(k-1)}_\alpha}{f_\alpha}+\ldots+a_{0,\alpha}\frac{f_\alpha}{f_\alpha}$$
this implies
\begin{align}
m\left(r,\frac{g_\alpha}{f_\alpha}\right) &\leq m\left(r,\frac{f^{(k)}_\alpha}{f_\alpha}\right)+m\left(r,a_{k-1,\alpha}\frac{f^{(k-1)}_\alpha}{f_\alpha}\right)+\ldots+ m(r,a_{0,\alpha})+ \log (k+1) \notag \\
&=m\left(r,\frac{f^{(k)}_\alpha}{f_\alpha}\right)+m(r,a_{k-1,\alpha}) + m\left(r,\frac{f^{(k-1)}_\alpha}{f_\alpha}\right)+\ldots+ m(r,a_{1,\alpha}) \notag \\
&+m\left(r,\frac{f \textquoteright_ \alpha}{f_\alpha}\right)+ m(r,a_{0,\alpha})+ \log(k+1) \notag \\
&= \sum_{j=0}^{k-1} m(r,a_{j,\alpha})+ \sum_{j=1}^{k} m \left(r,\frac{f^{(j)}_\alpha}{f_\alpha}\right)+C.
\end{align}
Again unless $\mathcal{G}_3$ has no normally convergent subfamily we have by Lemma (\ref{lemma8}) (fixing $\rho<1$) $\exists \quad 0<r_2 < \rho <1$ and an infinite subfamily $\mathcal{G}_4$ and $\mathcal{F}_4$ of $\mathcal{G}_3$ and $\mathcal{F}_3$ respectively such that in (\ref{eq5.24}) we get
\begin{align}
m\left(r,\frac{g_\alpha}{f_\alpha}\right) & \leq \sum_{j=0}^{k-1} m(r,a_{j,\alpha}) +A \log T(R,f_\alpha) +B \log \frac{1}{R-r} +C \notag \\
&= H_\alpha(r) + A\log^{+}T(R,f_\alpha) +B \log \frac{1}{R-r} +C, \label{eq5.34}
\end{align} whenever $r_2<r<R<\rho<1 , \quad g \in \mathcal{G}_4$ and $f \in \mathcal{F}_4$ assume $r_2>r_1$. Since $H_\alpha(r)$ is an increasing function of $\log r $ ($\therefore$ of $r$) and  $\mid \alpha \mid < r_0 , r< \rho < R $, we have that there is a constant $B(r_0,\rho)$ such that 
$H_\alpha(r)< B(r_0,\rho)$ for  $r_2<r<R<\rho<1$ where $B(r_0,\rho)$ is constant depending over $r_0 \quad \mbox{and} \quad \rho$. Inequality (\ref{eq5.34}) becomes
\begin{equation}\label{eq5.34}
m\left(r,\frac{g_ \alpha}{f_\alpha}\right) \leq A \log^{+}T(R,f_\alpha)+B \log\frac{1}{R-r}+C.
\end{equation}
Using (\ref{eq5.26})and (\ref{eq5.35}) in (\ref{5.28}) 
\begin{align}
T\left(r,\frac{1}{f_\alpha}\right) &\leq N\left(r,\frac{1}{f_\alpha}\right)+ N\left(r,\frac{1}{g_\alpha-1}\right)-N\left(r,\frac{1}{g\textquoteright_\alpha}\right)+A \log ^{+}T(R,f_\alpha) \notag \\
&+B \log \frac{1}{R-r}+C + \log \big | \frac{g(\alpha)-1}{g\textquoteright_\alpha} \big |,\label{eq5.37}
\end{align}   for $r_2<r<R<\rho<1 \quad f \in \mathcal{F}_4 ,\quad g \in \mathcal{G}_4$ since $f(z) \neq 0 \quad g(z) \neq 1 \quad \mbox{and} \quad g\textquoteright(z) \neq 0$ for any $z$ in $\Delta$ we get in (\ref{eq5.37}):
\begin{align}
T\left(r,\frac{1}{f_\alpha}\right) \leq A \log^{+} T(R,f_\alpha)+B\log \frac{1}{R-r}+C+ \log \big | \frac{g(\alpha) -1}{g\textquoteright(\alpha)}\big |,\label{eq5.38}
\end{align}
for $r_2<r<R<\rho<1 \quad f \in \mathcal{F}_4 ,\quad g \in \mathcal{G}_4$. By Lemma (\ref{lemma7}) unless $\mathcal{F}_4$ is not normal $\exists \quad r_3 < \rho \quad \mbox{assume} \quad r_3> r_2$ and constant $k$ such that
$$m(r,f_\alpha) < k m\left(r,\frac{1}{f_\alpha}\right),$$ 
for $f \in \mathcal{F}_5$; infinite subset of $\mathcal{F}_4$, $r_3<r<\rho$. Since $\mathcal{F}$ is family of non vanishing holomorphic functions therefore $$m(r,f_\alpha)=T(r,f_\alpha)$$ and $$m\left(r,\frac{1}{f_\alpha}\right)=T\left(r,\frac{1}{f_\alpha}\right).$$ This implies 
\begin{align}
T(r,f_\alpha) & \leq kT\left(r,\frac{1}{f_\alpha}\right) \notag \\
& \leq A \log^{+} T(R,f_\alpha) +B \log \frac{1}{R-r} +C + k \log \mid \frac{g(\alpha)-1}{g\textquoteright(\alpha)}\mid \quad [\mbox{by} \quad (\ref{eq5.38})] \notag \\
& = A \log^{+} T(R,f_\alpha) +B \log \frac{1}{R-r} +C \textquoteright \label{eq5.41}
\end{align} whenever $0<r_0<r_1<r_2<r_3<r<R<\rho <1 $ and $f \in \mathcal{F}_5$.
\\
Now $r \rho< \rho^2 <\rho$ , $r \rho  < R \rho < R $ and $ r_3 <r \rho< R \rho<R $ this implies
$ r_3<r \rho< R< \rho. $ This gives that equation (\ref{eq5.32}) can be written for $r \rho >0 $ thus 
\begin{align*}
T(\rho r,f_\alpha) \leq A \log^{+} T(R,f_\alpha) +B \log \frac{1}{R-\rho r}+C \textquoteright, r_3< \rho r <R , f \in \mathcal{F}_5.
\end{align*}
Let $U(r) = T(\rho r,f_\alpha) $ and $\gamma(r) \equiv 0 $ 
by Lemma (\ref{lemma2}) we have
\begin{align*}
T(\rho r,f_\alpha) &=m(\rho r,f_\alpha) \\
&\leq M_1 +2B \log \frac{1}{R-\rho r} = \sum(\rho r), r_3<\rho r<R< \rho , f \in \mathcal{F}_5,
\end{align*}
 where $\sum(\rho r) $ is an increasing function of $r$  thus by Lemma (\ref{lemmaA}) $\mathcal{F}_5$ is normal in $\mid z \mid < \rho$. Thus $\mathcal{F}_1$ has a normally convergent subfamily in $\mid z\mid < \rho$.
\end{proof}
\begin{note}
The condition that $\mathcal{F}$ is a family of non-vanishing holomorphic functions can not be relaxed. To see this let $$\mathcal{F} =\{\, n,n=2,3,4 \ldots \,\}$$ For $k=1, {f_n}={nz}$ and $a_0\equiv 0$ $$ f\textquoteright_n(z)=n$$ so $$ g_n(z) =f\textquoteright_n(z)+ a_0(z)f_n(z)=0$$
$g_n=1$ has no solution in $\Delta$ but $\mathcal{F}$ is not normal in $\Delta$.
\end{note}
\section{Extension }

\quad In this section it has been try to drop the non vanishing assumption for family of holomrphic functions in previous section in special cases.  

For $\mathcal{F}$ to be a family of holomorphic functions a new family $\mathcal{G}$ of holomorphic functions is defined as follows:
$$g(z)= f^{(k)}(z)+a_{k-1}(z)f^{(k-1)}(z)+\ldots+a_1(z)f \textquoteright(z)+a_0(z) f(z),$$
where $a_0,a_1,\ldots,a_{k-1}$ are fixed holomorphic functions over $\Delta$ , $f \in \mathcal{F}$ and $z \in \Delta$. Like wise we define $g_\alpha $ by $f_\alpha$ as:
$$g_\alpha(z)= f^{(k)}_\alpha(z)+a_{k-1,\alpha}(z)f^{(k-1)}_\alpha(z)+\ldots+a_{1,\alpha}(z)f \textquoteright_\alpha(z)+a_{0,\alpha}(z)+f_\alpha(z), $$ 
$ z \in \Delta , f \in \mathcal{F}$. 
\\
\begin{thm} \label{th4}
Let $\mathcal{F}$ be a family of functions holomorphic in $\Delta$ and $\mathcal{G}$ be the family related to $\mathcal{F}$ as above. Suppose the zeros of each $f$ in $\mathcal{F}$ are of multiplicities $\geq m$ and the zeros of $g(z)-1$ are of multiplicities $\geq p$ and 
\begin{align}\label{eq5.42}
\frac{k+1}{m}+ \frac{k+1}{p} =\tau <1.
\end{align}
 Then $\mathcal{F}$ is normal.
\end{thm}

Before proving the theorem we need a preliminary result which states 
\\
\begin{lem}\label{lemma11}
$\mathcal{F}$ be the family of holomorphic functions over $\Delta$ an d $\mathcal{G}$ be the family consisting of functions given by $$ g(z)=f^{(k)}(z)+a_{k-1}(z) f^{(k-1)}(z)+\ldots+ a_1(z) f\textquoteright(z)+ a_0(z) f(z),$$ 
where $ a_0, a_1,\ldots, a_{k-1}$ are fixed holomorphic functions and satisfy (\ref{eq5.42}) then if $\mathcal{G}$ is normal then $\mathcal{F}$ is normal.
\end{lem}
\begin{proof}
We will show the normality of $\mathcal{F}$ locally. For this let $g_n\rightarrow g_0$ normally in a neighbourhood of $z_0$ say $z_0=0$.
\\
\textbf{Case (i)}$$g_0(z_0)\neq 0$$
since (\ref{eq5.42}) gives $$ \frac{k+1}{m}<1 \Rightarrow m> k+1$$
so $f$ has zeros of multiplicities atleast $k$ which gives that zeros of $f$ are also zeros of $g$.
And because $g_0(z_0) \neq 0$ this gives that in a neighbourhood of $z_0$ no zeros of the corresponding $f_n \in \mathcal{F}$ lies. Therefore case(i) of the Lemma (\ref{lemma9}) of previous section can be applies in this case.
\\
\textbf{Case(ii)} $$g_0(z_0)=0$$ and there is a neighbourhood of $z_0$ such that $f_n \in \mathcal{F}$ has no zero in that neighbourhood of $z_0$ for infinitely many $f_n$. Then case(ii) of Lemma (\ref{lemma9}) is applicable.
\\
\textbf{Case(iii)} $$g_0(z_0) =0$$ and there exist $z_n \in \Delta$ such that $z_n \rightarrow 0 \quad \mbox{as } n \rightarrow \infty, f_n(z_n)=0$ and thus $g_n(z_n)=0$. Again let $h_1(z), h_2(z), \ldots, h_k(z)$ be linearly independent solution of homogeneous equation $g(z)=0$. As obtained before 
\begin{align}\label{eq5.43}
f_n(z)= f_{n,h}(z)+ \sum_{m=1}^{k} \left\lbrace \int_{z_n}^{z} \frac{W_m(t)}{W(t)} g_n(t) \, dt \right\rbrace h_m(z),
\end{align} 
where $f_{n,h}(z)= \sum \alpha_{n,m} h_m(z)$ is a solution of the homogeneous equation $g(z)=0$.
\\
\textbf{Claim} $$ f_{n,h} \equiv 0; n> n_0 $$
so that (\ref{eq5.43}) becomes 
$$ f_n(z) = \sum_{m=1}^{k}\left\lbrace \int_{z_n}^{z} \frac{W_m(t)}{W(t)}g_n(t) \, dt \right\rbrace h_m(z)$$
let us complete the proof of Lemma (\ref{lemma11}) chose $\delta>0$ and then determine constant $M< \infty$ such that
\\
$$\mid g_n(z) \mid <M,  \big| \frac{W_m(t)}{W(t)} \big| <M,  \mid h_m(z) \mid <M$$
whenever $\mid z \mid < \delta; m=1,2,3, \ldots,k\quad n>n_0$.
Then $$ \mid f_n(z) \mid <M^{*}, $$ 
for $\mid z \mid < \delta$ and $n=1,2,3,\ldots$ this concludes that $\{\,f_n\,\} $ is normal in a neighbourhood of origin. Implies $\mathcal{F}$ is normal at each point of $\Delta$.
In order to prove claim write $f_n=f, f_{n,h}=f_h, g_n=g$ then rewrite (\ref{eq5.44}) as
\begin{equation}\label{eq5.44}
f(z)=f_h(z)+ \sum_{m=1}^{k} \beta_m(z) h_m(z),
\end{equation}
where $$ \beta_m(z)= \int_{z_n}^{z} \frac{W_m(t)}{W(t)} g_n(t) \, dt; m=1,2,3,\ldots,k$$
and $\beta_m(z_n)=0, m=1,2,3,\ldots,k$.
Differentating (\ref{eq5.44}) leads to 
\begin{equation}\label{eq5.45}
f\textquoteright(z)=f\textquoteright_h(z)+ \sum_{m=1}^{k} \left(\beta_m(z)h\textquoteright_m(z)+ \beta\textquoteright_m(z)h_m(z) \right)
\end{equation}using
\\
$\beta\textquoteright_1(z) h_1(z) + \beta \textquoteright _2 (z) h_2(z) +\ldots+ \beta \textquoteright _k(z) h_k(z)=0$
\\
$\beta \textquoteright _1(z) h \textquoteright _1(z) + \beta \textquoteright _2 (z) h \textquoteright _2 (z) +\ldots+ \beta \textquoteright _k(z) h \textquoteright _k(z) =0 $
\\
.
\\
.
\\
.
\\
$\beta \textquoteright _1(z) h^{(k-2)} _1(z) + \beta \textquoteright_2(z) h ^{(k-2)}_2 (z) +\ldots+ \beta \textquoteright _k(z) h^{(k-2)} _k(z) =0$
\\
$\beta\textquoteright_1(z) h^{(k-1)}_1(z) + \beta \textquoteright_2(z) h^{(k-1)}_2(z)+\ldots+ \beta \textquoteright _k (z) h^{(k-1)}_k(z)=g(z)$
\\ 
of previous section, $\beta_m(z_n)=0, m=1,2,3,\ldots,k$ and multiplicity of zero of $f\geq m>k+1$ we obtain from (\ref{eq5.45}) that $$ f\textquoteright_h(z_n)=0$$
continuning sucessive derivatives together with above conditions with $\beta_m(z_n)=0$ and $f^{(p)}(z_n)=0; p=1,2,3,\ldots, k-1 $ we deduce that
$$ f^{(2)}_h(z_n)=f^{(3)}(z_n)= \ldots = f^{(k-1)}_h(z_n)=0.$$
Also $f_h(z_n)=0$ so claim follows from uniqueness of the solution to the homogeneous equation $g(z)=0$ in a neighbourhood of $z_0$.
\end{proof}
\begin{proof}
\textbf{Proof of the Theorem (\ref{th4})} 
Let $\mathcal{F}_1$ be denumerable infinite subcollection from $\mathcal{F}$. To show that $\mathcal{F}_1$ is normal in a neighbourhood of $z_0=0.$ Unless $z_0$ is a limit point of zeros of $\mathcal{F}_1$ this theorem is a consequence of Theorem (\ref{th3}). If $f_\alpha(0)=f(\alpha)\neq0$
\begin{align}
N\left(r,\frac{1}{f_\alpha}\right)-N^{*}\left( r,\frac{1}{g\textquoteright_\alpha}\right)& \leq (k+1) \bar{N} \left(r,\frac{1}{f_\alpha}\right) \notag \\
&\leq \frac{(k+1)}{m} N\left(r,\frac{1}{f_\alpha}\right) \leq \frac{(k+1)}{m} T\left(r,\frac{1}{f_\alpha}\right) \notag \\
&= \frac{(k+1)}{m} \left( T(r,f_\alpha)-\log \mid f_\alpha(0) \mid \right),\label{eq5.48}
\end{align} 
where $ N^{*}\left(r,\frac{1}{g\textquoteright_\alpha}\right)$ counts the zeros which are common zeros of $ g \textquoteright_\alpha$ and $f_\alpha$ and
\begin{align}
N\left(r,\frac{1}{g_\alpha-1}\right)-N^{**}\left(r,\frac{1}{g\textquoteright_ \alpha}\right) &\leq \frac{1}{p}N\left(r,\frac{1}{g_\alpha-1}\right)\notag \\
& \leq \frac{1}{p} T\left(r,\frac{1}{g_\alpha-1}\right) \notag \\
& =\frac{1}{p} \left( T(r,g_\alpha)-\log \mid g_\alpha(0)-1 \mid +A \right),\label{eq5.51}
\end{align}
where $N^{**}\left(r,\frac{1}{g\textquoteright_\alpha}\right)$ counts the zeros of $g\textquoteright_\alpha$ which arises from zeros of $g_\alpha-1.$

Analysis carried out in the proof of Theorem (5.1.3) required here. Since only behaviour of $\mathcal{F}_1$ in a neighbourhood of origin is relevant, take $\rho =\frac{1}{2} $ and $\mid \alpha \mid < \frac{1}{4}$ in Theorem (\ref{th3}). Put (\ref{eq5.48}) and (\ref{eq5.51}) in (\ref{eq5.37})
\begin{align}
t\left(r,\frac{1}{f_\alpha}\right) & \leq \frac{(k+1)}{m} \left(T(r,f_\alpha)- \log \mid f(\alpha) \mid \right) \\
&+ \frac{1}{p} \left(T(r,f_\alpha)-\log \mid g(\alpha)-1 \mid+D \right) +A \log^{+} T(R,f_\alpha)  \\
&+B \log \frac{1}{R-r} + \log \big | \frac{g(\alpha)-1}{g\textquoteright(\alpha)} \big | +C.\label{eq5.54}
\end{align}   
\\
By Jensen\textquoteright s formula 
$$T(r,f_\alpha)= T\left(r,\frac{1}{f_\alpha}\right)+ \log \mid f_\alpha(0) \mid$$ 
we find (\ref{eq5.54}) as
\begin{align}
T(r,f_\alpha) & \leq \frac{(k+1)}{m}\left (T(r,f_\alpha) -\log \mid f(\alpha) \mid \right) + \log \mid f(\alpha) \mid\notag \\
& + \frac{1}{p} \left(T(r,g_\alpha) -\log \mid g(\alpha)-1 \mid+ D \right) +A \log^{+}T(R,f_\alpha) \notag \\
& +B \log \frac{1}{R-r} +\log \big | \frac{g(\alpha)-1}{g\textquoteright(\alpha)}\big | +C,\label{eq5.57}
\end{align}
where $r< R<\rho$. To deal with term $T(r,g_\alpha)$ by using Lemma (\ref{lemma8}) and $$T(r,f^{(k)}_\alpha) \leq m\left(r,\frac{f_\alpha^{(k)}}{f_\alpha}\right) +T(r,f_\alpha)$$
we get
\begin{align}
T(r,g_\alpha) & \leq (k+1) T(r,f_\alpha)+A \log^{+}T(R,f_\alpha) +B \log \frac{1}{R-r} +C,\label{eq5.58}
\end{align}
$r<R< \rho.$ Putting this estimate of $T(r,g_\alpha)$ in (\ref{eq5.57}) 
\begin{align}
(1-\tau)T(r,f_\alpha) &\leq A \log^{+}T(R,f_\alpha) +B \log \frac{1}{R-r} \notag 
\\
&+ \log \big| (f(\alpha))^{1-\frac{k+1}{m}}(g(\alpha)-1)^{1-\frac{1}{p}}(g\textquoteright(\alpha)^{-1})) \big| +C,
\end{align} 
$\delta<r<R<\rho$ for $f \in \mathcal{F}_2 \subseteq \mathcal{F}_1$.
It is clear from Lemma (\ref{lemma2}) family $\mathcal{F}$ is normal if there exists sequence $\{ \, \alpha_n \,\}$ such that $\alpha_n \rightarrow0$ and 
\begin{equation}\label{eq5.61}
 \big| (f(\alpha))^{1-\frac{k+1}{m}} (g(\alpha)-1)^{1-\frac{1}{p}} g\textquoteright(\alpha)^{-1} \big| <M,
 \end{equation}
where $\alpha=\alpha_n, f=f_n \in \mathcal{F}_2$ corresponding to $g=g_n \in \mathcal{G}_2$. For proving this we proceed: let $f_n$ has a zero $z_n$ with $z_n \rightarrow 0$ and since $m>k+1$ this gives that $z_n$ is also zero of $g_n$ and $g\textquoteright_n$.
\\
For given $z_n$ determine $y_n$ by
$$ \min \mid z_n-y_n\mid$$
subject to $$\mid g\textquoteright_n(y_n) \mid=1,$$
If this optimization problem has no solution then $g\textquoteright_n$ would be uniformly bounded in $\Delta$ and $g_n(z_n)=0, \quad z_n \in \Delta$ by Lemma (\ref{con f`lcb})  $g_n$ would have a convergent subfamily. So let this optimization problem has solution, then let 
$$\mid z_n-y_n \mid=\delta_n$$ and if $ \mid z-z_n\mid <\delta_n$ by maximum modulus principle $\mid g\textquoteright_n(z) \mid<1$ therefore
\begin{equation}\label{eq5.62}
\mid g_n(z) \mid < \int_{z_n}^{z} \mid g\textquoteright_n(z) \mid \mid dz \mid \, <\delta_n
\end{equation}  
Suppose that there were an $\eta >0 $ with the property that $\delta_{n_k}> \eta > 0$ for a subsequence $\delta_{n_k}$ then from the definition of $\delta_{n_k}, \quad z_{n_k} \rightarrow 0$ and (\ref{eq5.62}) it follows that $$\mid g_{n_k} \mid \leq 1, \quad
 \mid z \mid < \frac{\eta}{2} $$
 which by Lemma (\ref{sch montel}) gives that $\mathcal{G}$ is normal and thus by Lemma (\ref{lemma11}), $\mathcal{F}$ is normal.
 \\
Therefore let $\delta_n \rightarrow 0$ then by $$f_n(z)= \sum \left\lbrace \int_{z_n}^{z} \frac{W_m(t)}{W(t)}g_n(t)dt \, \right\rbrace h_m(t)$$ we have $f_n(y_n) \rightarrow 0$. Also  $\mid g\textquoteright(y_n) \mid =1$ and $$\big | \frac{g(y_n)}{g\textquoteright(y_n)} \big| \leq \int_{z_n}^{y_n} \mid g\textquoteright (z) \mid dt \, \leq \mid z_n-y_n\mid =\delta_n .$$
Thus (\ref{eq5.61}) follows actually expression (\ref{eq5.61}) tends to zero. This proof the Theorem (\ref{th4}) completely
\end{proof}
It is well known result due to Hayman ~\cite{haym} for $ n\geq 2 $ and Cluine ~\cite{clu} for $ n \geq 1$ that if $f$ is entire function satisfying $f\textquoteright(z)f^n(z) \neq 1$ forall $z \in \mathbb{C}$ then $f \equiv 
$ constant.
\\
This result and next corollary permits a further illustration of Bloch\textquoteright s principle.
\begin{cor} 
Let $\mathcal{H}$ be a family of holomorphic functions in $\Delta$ such that the equation
$$h \textquoteright(z) h(z)^n =1,$$ where $n \geq 2 $ is fixed integer,
has no solution in $\Delta$ then $\mathcal{H}$ is normal. 
\end{cor}
\begin{proof}
Consider the family 
$$\mathcal{F}= \left\lbrace f(z)= \frac{1}{n+1} h(z)^{n+1} ,h \in \mathcal{H} \right\rbrace$$ over $\Delta$.
Let $k=1$ and $a_0=0$ for $\mathcal{F}$ to give new family $\mathcal{G}$ which consists of the functions $$g(z)=f\textquoteright(z) \quad \forall z \in \Delta, f \in \mathcal{F}.$$ This implies 
$$g(z)= h'(z)h(z)^n$$
which by the given condition implies that $g(z)-1$ has no solution in $\Delta$. Also consider
$$f(z)=0 \Rightarrow h(z)^{n+1}=0.$$
If $z_0 \in \Delta$ is simple zero of $h(z)$ then
$$h(z)=(z-z_0)t(z) ; \quad t(z_0) \neq 0,$$ where \textquotedblleft $t$ \textquotedblright is holomorphic in $\Delta$ implies
$$ f(z)=\frac{1}{n+1}(z-z_0)^{n+1}t(z)^{n+1}=(z-z_0)^{n+1}q(z),$$
where $q(z) = \frac{1}{n+1} t(z)^{n+1}$ is holomorphic on $\Delta$ and $q(z_0) \neq 0$ that means $f$ has zero at $z_0$ of multiplicity $n+1$. If $h(z)$ at $z_0$ has multiplicity 2 then $f$ has at $z_0$ multiplicity $2(n+1)$ and so on. Thus $f$ has zeros of multiplicities $\geq n+1 =m $ therefore in (\ref{eq5.42})  $$\frac{k+1}{m}= \frac{2}{n+1}$$  as $g(z)-1$ has no zero in $\Delta $.
Also $n\geq 2 \Rightarrow n+1 \geq 3 \Rightarrow 1/ {n+1} \leq 1/3$ 
\\
which gives 
$$\frac{k+1}{m}=\frac{2}{n+1} \leq \frac{2}{3}< 1$$
\\
this gives that $\mathcal{F}$ is normal in $\Delta$. Also $$h(z) =\left( (n+1) f(z) \right)^{\frac{1}{n+1}},$$ $f \in \mathcal{F}, h \in \mathcal{H}$ and if $\{\,f_n \,\}$ is sequence in $\mathcal{F}$ then 
$f_{n_k} \rightarrow f$ uniformly on compact subset of $\Delta$ which implies
$$(n+1) f_{n_k} \rightarrow (n+1)f$$ uniformly on compact subset of $\Delta$
implies
$$ \left((n+1)f_{n_k}\right)^{\frac{1}{n+1}} \rightarrow \left( (n+1) f \right)^{\frac{1}{n+1}}$$ uniformly on compact subset of $\Delta$
which again implies
$$h_{n_k} \rightarrow h $$ uniformly on compact subset of $\Delta$.
\\
Thus $\mathcal{H}$ is normal in $\Delta$.
\end{proof}
The case $n=1$ has been given by Oshkin in 1982.

\chapter{A New Normality Criterion}

\quad We have go through many normality criterion in our previous chapters. Now here is a normality criterion for a family which confirms a conjecture of Hayman.
\section{Result due to Hayman}

Hayman ~\cite{haym} gave a result which is as following: If $f(z)$ is entire  and $n \geq 3$ is fixed integer, $a\neq 0$ and $$ f\textquoteright(z)-af(z)^{n}\neq b$$ for some $b \in \mathbb{C}$, then $f \equiv \mbox{constant }$.   
Mues in 1979 gave counterexamples to show that if $f$ is meromorphic function and $n=3,4$ then this result is no longer valid.

In terms of normal families we have:
\begin{thm}\label{hay con}

Let $\mathcal{F}$ be a family of holomorphic (meromorphic) function in $\Delta$ and for a fixed $n\geq 3 $ ($n\geq 5$) and $a\neq 0$ suppose that $$ f\textquoteright-a f^{n}=b; \quad f \in \mathcal{F}$$ has no solution in $\Delta$, then $\mathcal{F}$ is normal. 
\end{thm}
The analytic case is due to Drasin ~\cite{dra} and meromorphic case is due to Langley ~\cite{lan}, S. Li[1984] and X. Li{1985]. Despite the counterexamples of Mues, Pang ~\cite{pan} has demonstrated the validity of Theorem (\ref{hay con}) for family of meromorphic functions and for $n\geq4$. This was achieved using Zalcman lemma (\ref{zal}). With the additional condition that zeros of every $f\in \mathcal{F}$ have multiplicity not less than 2, Pang proved the case $n=3.$
We follow Hayman ~\cite{haym} and for $f \in \mathcal{F}$ determine $$ h= a \frac{f^n}{f\textquoteright-b-af^n}$$ $\mathcal{H}$ denotes the family of these functions and hypothesis of the above theorem implies that $\mathcal{H}$ is a family of holomorphic functions.

Next, we require a preliminary result given as follows:
\begin{lem}\label{lemma12}
If $\mathcal{H}$ is normal so is $\mathcal{F}$.
\end{lem}
\begin{proof}

$\mathcal{H}$ is a family of holomorphic functions and for $\mid \alpha \mid < 1$ we have $$ h_{\alpha} =a \frac{f_\alpha^n}{f_\alpha \textquoteright-b-af_\alpha ^n}.$$
If $h_\alpha(0) \neq -1$ then 
\begin{align}
nT(r,f_\alpha) &= T(r,f_\alpha^n) = T\left(r,\frac{f_\alpha \textquoteright-b}{a} \frac{h_\alpha}{h_\alpha+1}\right)\notag \\
&= T\left(r,\frac{f_\alpha \textquoteright -b}{a} \left(1-\frac{1}{h_\alpha+1}\right)\right) \notag \\
& \leq T(r,f_\alpha \textquoteright)+T\left(r,\frac{1}{h_\alpha+1}\right)+K(a,b) \notag \\
&= m(r,f_\alpha \textquoteright)+T\left(r,\frac{1}{h_\alpha+1}\right)+K(a,b)\notag  \\
& \leq m\left(r,\frac{f-\alpha \textquoteright}{f_\alpha}\right)+T(r,f_\alpha)+T(r,h_\alpha)-\log \mid h(\alpha)+1\mid+K(a,b)\label{eq6.5}
\end{align} 
\begin{align}
T(r,f_\alpha)\leq \frac{1}{n-1} \left[ \, m\left(r,\frac{f_\alpha \textquoteright}{f_\alpha}\right)+T(r,h_\alpha)-\log \mid h(\alpha)+1 \mid + K(a,b) \right].\label{eq6.6}
\end{align}

Let $\mathcal{F}_1$ be a denumerable subfamily of $\mathcal{F}$ and $\mathcal{H}_1$ be corresponding subfamily of $\mathcal{H}$. Let $h^{*}$ be cluster function of $\mathcal{H}_1$. Three cases arises:
\\
\textbf{Case(i)}

$h^{*}$ not identically equal to $ -1, \infty$. Let $\{ \, h_n \, \} = \mathcal{H}_2 \subseteq \mathcal{H}_1$ tends to $ h^{*}$ normally
$$\Rightarrow h_n+1  \rightarrow h^{*} +1\quad \mbox{normally}$$
$$\Rightarrow \frac{1}{h_n+1} \rightarrow \frac{1}{h^{*}+1} \quad \mbox{normally}.$$ 
Since $\{\, \frac{1}{h_n+1} \, \}$ tends to a holomorphic function $\frac{1}{h^{*}+1}$ normally therefore by converse of Theorem (\ref{con montel}) $ \{\,\frac{1}{h_n+1} \, \} $ is locally bounded on $\Delta$. Therefore there exists $M< \infty, r_0<1, \alpha_n=\alpha_n(h_n) \mid \alpha_n \mid <r_0$ with $$ \big | \frac{1}{h_n(\alpha_n)+1} \big |<M$$ this gives $$ - \log \mid h_n(\alpha_n) +1 \mid< M, $$ for an infinite family $\mathcal{H}$.
Further unless $\mathcal{F}_3$ normal, by Lemma (\ref{lemma6}) inequality (\ref{6.6}) becomes $$ T(r,f_{m,\alpha_m}) \leq \frac{1}{n-1} \left[ \,  A \log^{+}T(R,f_{m,\alpha_m}) +B \log \frac{1}{R-r}+T(r,h_{m,\alpha_m}) +M \right]$$ $\frac{1}{2} (1+r_1) <r<R<1$, for infinite subfamily $\mathcal{F}_4 \subseteq \mathcal{F}_3$.
\\
Let $U(r) =T(r,f_{m,\alpha_m})$ and $\gamma(r) = \frac{1}{n-1} T(r,h_{m,\alpha_m})$ by Lemma (\ref{lemma2}) 
\begin{equation}\label{eq6.7}
T(r,f_{m,\alpha_m}) \leq A + B T(R,h_{m,\alpha_m})+C \log \frac{1}{R-r},
\end{equation}
$\frac{1}{2}(1+r_1)<r<R<1$. Also for $r<1$ $$ T(r,h_{m,\alpha_m}) \rightarrow T(r,h^{*}_\alpha) \quad \mbox{as} \quad m \rightarrow \infty $$ $$ \Rightarrow \big| T(r,h_{m,\alpha_m})-T(r,h^{*}_\alpha) \big| < \epsilon \quad \forall m>m_0.$$ 
\\
So $$T(r,h_{m,\alpha_m}) < \epsilon+T(r,h_\alpha^{*}) \quad \forall m>m_0$$
\\
(\ref{eq6.7}) implies that $$T(r,f_{m,\alpha_m})\leq \sum(r),$$ $\frac{1}{2}(1+r_1)<r<1$, Lemma (\ref{lemmaB}) implies normality of $\mathcal{F}$.
\\
\textbf{Case(ii)}

$h^{*} \equiv -1$ gives $$ a \frac{f_m^{n}}{f_m\textquoteright-b-af_m^{n}} \rightarrow -1$$
$$ \frac{f\textquoteright _m-b-af_m^{n}}{af_m^{n}} \rightarrow-1$$
$$ \frac{f\textquoteright_m-b}{af_m^{n}} \rightarrow0.$$
Then for each $r_0<1$ and $\epsilon>0$ $$\big | \frac{f\textquoteright(z)-b}{af(z)^{n}} \big |<\epsilon, \quad \mid z \mid <r_0,$$ for infinitely many $f \in \mathcal{F}$.

Let us now show that for each $r_0<1 \quad \exists \quad \eta(r_0)$ such that if $\mid z_1\mid <r_0,$
\\
$ \mid z_2 \mid < r_0$ with 
\begin{equation}\label{eq6.8}
 \mid f(z_1) \mid \leq 1 
\end{equation} and 
\begin{equation}\label{eq6.9}
\mid f(z_2) \mid \geq 2,
\end{equation} for $f \in \mathcal{F}$ then 
\begin{equation}\label{eq6.10}
\mid z_1-z_2 \mid \geq \eta.
\end{equation}

Let $z_1 $ and $z_2$ be chosen according to (\ref{6.8}) and (\ref{eq6.9}). Suppose that
\\
$\mid z_1-z_2 \mid $ is minimized subject to these conditions. Let $\gamma$ denotes the line segment joining $z_1$ and $ z_2$ then $\gamma$ will lie in $\mid z \mid< r_0$, for $z$ on $\gamma$ we have $$\mid f\textquoteright (z) \mid \leq \epsilon \mid a f(z) ^{n}\mid +B \leq K \epsilon +B \quad \because \mid f(z) \mid \leq 2$$ $ K=K(a,n), B= \mid b \mid $.
\\
Then 
\begin{align*} 
1 & \leq \mid f(z_2)-f(z_1) \mid \\
& \leq \int_{\gamma} \mid f\textquoteright(z) \mid \, \mid dz \mid \\ 
& \leq (K\epsilon+B) \mid z_1-z_2 \mid
\end{align*} which gives (\ref{eq6.10}).

Thus in a prescribed neighbourhood of any point in $\mid z \mid <r_0$ every $f \in \mathcal{F}$ is either uniformly bounded above or below so that $\mathcal{F}$ is normal in $\Delta$.
\\
\textbf{Case(iii)}

$h^{*} \equiv \infty$ then $$\frac{af_m^{n}}{f\textquoteright_m-b-af_m^{n}} \rightarrow \infty$$
$$\frac{f\textquoteright_m-b-af_m^{n}}{af_m^{n}} \rightarrow 0$$
$$ \frac{f\textquoteright_m-b}{af_m^{n}} -1 \rightarrow 0$$ then for each $r_0<1, \epsilon>0$ 
$$ \mid \frac{f\textquoteright(z) -b}{af(z)^{n}} \mid < 1+ \epsilon, \quad
 \mid z \mid r_0$$ for infinitely many $f \in \mathcal{F}$. 
 
 Now same arguments of case (ii) works here, gives that $\mathcal{F}$ is normal in $\Delta$. 
\end{proof}
\begin{proof}
\textbf{Proof of the Theorem (\ref{hay con}) }
Let $\mathcal{F}_1$ be denumerable subfamily of $\mathcal{F}$ and $\mathcal{H}_1 $ be corresponding subfamily of $\mathcal{H}$. It is to show that $\mathcal{H}_1$ is normal in a neighbourhood of the origin. We choose $\alpha, \beta$ such that $\mid \alpha \mid + \mid \beta \mid $ is small. Initially let $\mid \alpha \mid < \frac{1}{4}$ and $\beta$ is small, their precise choices will be made later. 

Applying the fundamental inequality(\ref{sft}) for $h_\alpha$ we get
$$m(r,h_\alpha)+m\left(r,\frac{1}{h_\alpha}\right) +m\left(r,\frac{1}{h_\alpha-1}\right) \leq 2T(r,h_\alpha)-N_1(r,h_\alpha)+S(r,h_\alpha),$$ 
where 
\begin{align*}
N_1(r,h_\alpha)& =N\left(r,\frac{1}{h\textquoteright_\alpha}\right)+2N(r,h_\alpha)-N(r,h\textquoteright_\alpha)\\
&=N\left(r,\frac{1}{h\textquoteright_\alpha}\right) \quad \mbox{because } h \mbox{ is holomorphic} 
\end{align*}
$$S(r,h_\alpha) = 2m\left(r,\frac{h\textquoteright_\alpha}{h_\alpha}\right)+m\left(r,\frac{h\textquoteright_\alpha}{h_\alpha-1}\right)+ \log \big | \frac{1}{h\textquoteright_\alpha(0)} \big | +C$$ and $h(\alpha) \neq 0, \infty; h\textquoteright(\alpha) \neq 0$.
If $ \mathcal{H}_1$ not normal then by Lemma (\ref{lemma6})  $\exists \delta_1$ and infinite subfamily $\mathcal{H}_2$ of $\mathcal{H}_1$ we get
\begin{align}
m(r,h_\alpha)+m\left(r,\frac{1}{h_\alpha}\right)+m\left(r,\frac{1}{h_\alpha-1}\right) &\leq 2 T(r,h_\alpha)-N\left(r,\frac{1}{h\textquoteright_\alpha}\right) + A \log^{+} T(R,h_\alpha) \notag \\ &+B \log \frac{1}{R-r}+C+ \log \big | \frac{1}{h\textquoteright(\alpha)} \big |,\label{eq6.12}
\end{align}
where $\delta_1<r<R<1$ and $A,B$ and $C$ is independent of $r$.
Lemma (\ref{cor5}) implies that given $\epsilon>0$ chose $\delta_2< \frac{\delta_1}{2}$ with the property that 
$$ (1-\epsilon)m\left(r,\frac{1}{h_\alpha+1} o \psi_\beta\right) \leq m\left(r,\frac{1}{h_\alpha+1}\right),$$
where $\mid \beta \mid < \delta_2$and $\delta_1<r<R<1$.
In similar way as $h$ is holomorphic 
$$T(r,h_\alpha o \psi_\beta)> (1-\epsilon)T(r,h_\alpha),$$ $\mid \beta \mid < \delta_2, r> \delta_1$.
Now on both side of (\ref{eq6.12}) we add
$$N\left(r,\frac{1}{h_\alpha}\right)+\log \mid h(\alpha)\mid+(1-\epsilon) \left[N(r,\frac{1}{h_\alpha+1} o \psi_\beta)+ \log \mid (h_\alpha+1) o \psi_\beta(0) \mid \right]$$
to get
\begin{align}
(1-\epsilon)^2T(r,h_\alpha) &\leq \bar{N}\left(r,\frac{1}{h_\alpha}\right)+(1-\epsilon)N\left(r,\frac{1}{h_\alpha+1} o \psi_\beta\right)+ A \log^{+} T(R,h_\alpha)\notag  \\
&+B \log \frac{1}{R-r}+C+ (1-\epsilon) \log \mid (h_\alpha+1)(-\beta)\mid+ \log \big | \frac{h(\alpha)}{h\textquoteright(\alpha)} \big |,\label{eq6.14}
\end{align}
$\mid \beta \mid < \delta_2, \delta_1<r<R<1$. 

By the given hypothesis we ensures that the zeros of $h_\alpha$ have multiplicity atleast $n$ that means
\begin{align}
\bar{N}\left(r,\frac{1}{h_\alpha}\right)& \leq \frac{1}{n}N\left(r,\frac{1}{h_\alpha}\right) \notag \\
& \leq \frac{1}{n}T\left(r,\frac{1}{h_\alpha}\right)\notag \\
& \leq \frac{1}{n} \left\lbrace T(r,h_\alpha)-\log \mid h(\alpha) \mid \right\rbrace.\label{eq6.17}
\end{align}

Similarly consider
\begin{align}
N\left(r,\frac{1}{h_\alpha+1} o \psi_\beta\right) &= N\left(r,\left\lbrace \frac{a f^n}{h_\alpha} \frac{1}{f\textquoteright_\alpha}-b \right\rbrace o \psi_ \beta\right) \notag \\
& \leq N\left(r,\frac{1}{f\textquoteright_\alpha-b} o \psi_\beta\right) \leq T\left(r,\frac{1}{f\textquoteright_\alpha-b} o \psi_\beta\right) \notag \\
&\leq T(r,(f\textquoteright_\alpha-b) o \psi_\beta) -\log \mid (f\textquoteright_\alpha-b)(-\beta) \mid \notag \\ &=T(r,(f\textquoteright_\alpha-b) o \psi_\beta \frac{f_\alpha}{f_\alpha}) -\log \mid (f\textquoteright_\alpha-b)(-\beta) \mid \notag \\
&=T\left(r,\frac{f_\alpha f\textquoteright_\alpha-bf_\alpha}{f_\alpha} o \psi_\beta\right) -\log \mid (f\textquoteright_\alpha-b)(-\beta) \mid \notag \\
&\leq T\left(r,f_\alpha o \psi_\beta\right)+m\left(r,\frac{f\textquoteright_\alpha}{f_\alpha} o \psi_\beta\right) -\log \mid (f\textquoteright_\alpha-b)(-\beta) \mid \notag \\
&+A.\label{eq6.24}
\end{align}

Again by Lemma (\ref{cor5}) we get 
\begin{equation}\label{eq6.25}
T(r,f_\alpha o \psi_\beta) \leq (1+\epsilon)T(r,f_\alpha)
\end{equation}
\begin{equation}\label{eq6.26}
m\left(r,\frac{f\textquoteright_\alpha}{f_\alpha} o \psi_\beta\right) \leq (1+\epsilon)m\left(r,\frac{f\textquoteright_\alpha}{f_\alpha}\right).
\end{equation}

So (\ref{eq6.24}) becomes
\begin{align}
N\left(r,\frac{1}{h_\alpha+1} o \psi_\beta\right) & \leq (1+\epsilon)T(r,f_\alpha)+ (1+\epsilon)m\left(r,\frac{f\textquoteright_\alpha}{f_\alpha}\right) \notag \\
& -\log \mid (f\textquoteright_\alpha-b)(-\beta) \mid+A.\label{eq6.28}
\end{align}

Unless $\mathcal{F}_2$ normal in neighbourhood of origin by Lemma (\ref{lemma6`}) we get
\begin{align}
N\left(r,\frac{1}{h_\alpha+1} o \psi_\beta\right) &\leq (1+ \epsilon)\left[ T(r,f_\alpha)+A \log^{+}T(\rho,f_\alpha)+ B\log \big | \frac{1}{\rho-r}\big | \right] \notag \\
&-\log \mid (f\textquoteright_\alpha-b)(-\beta)\mid+A,\label{eq6.30}
\end{align} 
for $f \in \mathcal{F}_3 \subseteq \mathcal{F}_2(h \in \mathcal{H}_3) \subseteq \mathcal{H}_2)$, $\delta<r<\rho, \mid \beta \mid < \delta_2$.

Using all these inequalities in (\ref{eq6.14}) we get
\begin{align}
\left( (1-\epsilon)^2-\frac{1}{n} \right) T(r,h_\alpha)& \leq (1+\epsilon)\left[T(r,f_\alpha)+A \log^{+}T(\rho,f_\alpha)+B \log \frac{1}{\rho-r} \right] \notag \\
&+A \log^{+}T(R,h_\alpha)+B \log \frac{1}{R-r} \notag \\
& +(1-\epsilon) \log \big | \frac{(h_\alpha+1)(-\beta)}{(f\textquoteright_\alpha-b)(-\beta)} \big | + \log \big | \frac{h(\alpha) ^{1- \frac{1}{n}}}{h\textquoteright(\alpha)} \big | +C,\label{eq6.33}
\end{align} 
where $\delta_1<r<\rho; r<R<1 \quad \mbox{and } \mid \beta \mid <\delta_2 $

Again if $\mathcal{F}_3 \quad$ is not normal $\mbox{hence } \mathcal{H}_3$ is not normal at origin by Marty\textquoteright s theorem (\ref{lemmaA}) there exists a subfamily $ \{\, h_n \, \}= \mathcal{H}_4 \subseteq \mathcal{H}_3$ and points $\alpha_n= \alpha(h_n)$ with $\alpha \rightarrow 0$ such that
$$\frac{\mid h(\alpha) \mid^{1-\frac{1}{n}}}{\mid h\textquoteright(\alpha) \mid} <1, \quad h(\alpha) \neq 0,-1 \quad \mbox{and } h\textquoteright(\alpha) \neq 0, $$ where $h=h_n,\alpha=\alpha_n \quad n=1,2,3, \ldots $

Suppose $\{ \, f\textquoteright-b-af^n; f \in \mathcal{F}_4 \, \} $ is normally convergent to 0 in a neighbourhood of origin. Then for every $r_0<1$ and $\epsilon>0$ 
$$ \big | \frac{f\textquoteright-b}{af^n} \big | < \epsilon; \quad \mid z \mid <r_0 \quad \mbox{and } f \in \mathcal{F}_4 $$
then as done earlier from equation (\ref{eq6.8}) and (\ref{eq6.9}) we get equation (\ref{eq6.10}) which gives that $\mathcal{F}_4$ is normal thus our assumption was wrong and consequently there exists a subsequence $\{ \, f_n \,\}= \mathcal{F}_5 \subseteq \mathcal{F}_4$ and a sequence of points $\{ \, \beta_n \,\}$ with $\beta_n \rightarrow 0$ corresponding to $f_{n,\alpha_n}$ such that 
$$ \big| \frac{1}{f\textquoteright_{n,\alpha_n}-b-af^n_{n,\alpha_n}} (-\beta_n) \big|<M; \quad n=1,2,3,\ldots$$
and $$ \frac{h_{m,\alpha_m}+1}{f\textquoteright_{m,\alpha_m}-b}= \frac{1}{f\textquoteright_{m,\alpha_m}-b-af^n_{m,\alpha_m}}$$
therefore
$$\big | \frac{h_{m,\alpha_m}+1}{f\textquoteright_{m,\alpha_m-b}}(-\beta_m) \big |= \big |\frac{1}{f\textquoteright_{m,\alpha_m}-b-af^n_{m,\alpha_m}}(-\beta_m) \big | <M. $$

Thus equation (\ref{eq6.32}) becomes 
\begin{align}
 \left( (1-\epsilon)^2-\frac{1}{n} \right) T(r,h_\alpha)& \leq (1+ \epsilon) \left[ T(r,f_\alpha)+A \log^{+}T(\rho, f_\alpha)+ B \log \frac{1}{\rho-r } \right]\notag \\ 
 & +A \log^{+}T(R,h_\alpha)+B \log \frac{1}{R-r}+C,\label{eq6.35}
\end{align}
$\delta_1<r< \rho; r<R<1$ for $f \in \mathcal{F}_5, h \in \mathcal{H}_5 $.

Next we wish to majorize $T(r,f_\alpha)$ in terms of $T(r,h_\alpha)$ for this we replace $f_\alpha$ by $f_\alpha o \psi_\beta$ in (\ref{eq6.5}) then we get 
\begin{align}
T(r,f_\alpha o \psi_\beta) & \leq  \notag \\
&\frac{1}{n-1} \left[ m\left(r,\frac{f\textquoteright_\alpha}{f_\alpha} o \psi_\beta\right) +T(r,h_\alpha o \psi_\beta)- \log \mid h_\alpha(-\beta)+1\mid +K(a,b) \right].\label{eq6.37}
\end{align}

It may be assumed that $\mathcal{H}_5$ do not tend normally to $-1 \mbox{or } \infty$ in a neighbourhood of origin, for otherwise $\mathcal{F}_5$ would be normal there. Then we can find $\{\, \beta_n \,\}$ with $ \beta_n \rightarrow 0$ associated with $h_{n,\alpha_n}$ where $\{\, h_n \,\}= \mathcal{H}_6 \subseteq \mathcal{H}_5$ such that 
$$ - \log \mid h_{n,\alpha_n}(-\beta_n)+1 \mid <M$$

Hence (\ref{eq6.35}) becomes using (\ref{eq6.24}) and (\ref{eq6.25})
\begin{align}
T(r,f_\alpha) & \leq (1+\epsilon)T(r,f_\alpha o \psi_{\beta_n}) \notag \\
& \leq \frac{(1+\epsilon)^2}{n-1} \left[ m\left(r,\frac{f\textquoteright_\alpha}{f_\alpha}\right)+T(r,h_\alpha)+M \right]; \delta_1<r<1, f \in \mathcal{F}_6.\label{eq6.39}
\end{align}

If $\mathcal{F}_6$ is not normal at origin then 
\begin{align}
m\left(r,\frac{f\textquoteright_\alpha}{f_\alpha}\right) & \leq A \log^{+} T(R,f_\alpha)+B \log \frac{1}{R-r} +C,\label{eq6.40}
\end{align}
$\delta_1<r<R<1; f \in \mathcal{F}_7 \subseteq \mathcal{F}_6$.

Thus (\ref{eq6.37}) becomes 
\begin{equation}
T(r,f_\alpha) \leq \frac{(1+\epsilon)^2}{n-1} \left[ A \log^{+}T(R,f_\alpha)+B \log \frac{1}{R-r} +T(r,h_\alpha)+M \right]\label{eq6.41}
\end{equation}
by Lemma (\ref{lemma2}) $[U(r)= T(r,f_\alpha), \gamma(r)= T(r,h_\alpha)]$ and replacing $r$ by $R$
\begin{equation}\label{eq6.42}
T(R,f_\alpha) \leq B \log \frac{1}{\rho\textquoteright-R} +AT(R,h_\alpha) +C,
\end{equation}
$\delta_1<R<\rho\textquoteright<1. $

Put (\ref{eq6.41}) in (\ref{eq6.39}) and let $R= \frac{\rho\textquoteright+r }{2}$ this gives $\frac{1}{\rho\textquoteright-R}= \frac{1}{R-r}$ 
\begin{equation}
m\left(r,\frac{f\textquoteright_\alpha}{f_\alpha}\right) \leq A\log^{+} T(R,h_\alpha)+B \log^{+} \log \frac{1}{\rho\textquoteright-R}+B \log \frac{1}{R-r}+C.
\end{equation}

Then (\ref{eq6.38}) becomes
\begin{align}
T(r,f_\alpha) & \leq \frac{(1+\epsilon)^2}{n-1}T(r,h_\alpha)+A \log^{+} T(R,h_\alpha) +B \log \frac{1}{R-r}+C,
\end{align} 
$\delta_1<r<R<1, f \in \mathcal{F}_7 $ using this in (6.34) and taking $\rho=R$ we obtain
\begin{align}
\left((1-\epsilon)^2-\frac{1}{n}\right)&\leq \frac{(1+\epsilon)^2}{n-1} \left[T(r,h_\alpha)+A \log^{+}T(R,h_\alpha)+B \log \frac{1}{R-r} +C \right].
\end{align}
So
\begin{align*}
\left((1-\epsilon)^2-\frac{1}{n} - \frac{(1+\epsilon)^3}{n-1} \right) \log T(r,h_\alpha) &\leq A \log^{+} T(R,f_\alpha)+B \log \frac{1}{R-r}+C,
\end{align*}
$\delta_1<r<R<1, h \in \mathcal{H}_7$ chose $\epsilon>0$ sufficiently small then by Lemma (\ref{lemma2}) $\mathcal{H}_7$ is normal hence $\mathcal{F}_7$ is normal implying $\mathcal{F}$ is normal at origin.
\end{proof} 
\begin{exmp}
$\mathcal{F}= \{\, z+m; m=1,2,3,\ldots \, \}$ is a normal family of holomorphic function on $\Delta$ and $f\textquoteright-af^{n}=b$ for $a \neq 0, b \in \mathbb{C}$ and $n \geq 3$ has no solution in $\Delta.$
\end{exmp}

\end{document}